\documentclass{amsart}
\usepackage{fontenc}
\usepackage[english]{babel}
\usepackage[margin=1in]{geometry}
\usepackage{tikz}

\usepackage{amssymb, amsmath,mathtools}

\usepackage{mathrsfs}
\usepackage{amscd}
\usepackage{verbatim}
\usepackage[dvipsnames]{xcolor}
\mathtoolsset{showonlyrefs}
\usepackage{ esint }
\newcommand{\llangle}{\langle \! \langle}

\newcommand{\rrangle}{\rangle \! \rangle}

\usepackage{enumerate}

\usepackage[colorlinks,linkcolor={blue},citecolor={blue},urlcolor={purple},]{hyperref}

\usepackage[colorinlistoftodos,prependcaption,textsize=tiny]{todonotes}

\newcommand{\G}{\mathscr{G}}

\numberwithin{equation}{section}

\usepackage{forest}
\usepackage{tikz}
\tikzset{>=latex}

\usepackage{tabularx}
\newcolumntype{L}{>{\arraybackslash}X}
\usepackage{multirow}

\theoremstyle{plain}
\newtheorem{theorem}{Theorem}[section]
\theoremstyle{remark}
\newtheorem{remark}[theorem]{Remark}

\theoremstyle{plain}
\newtheorem{corollary}[theorem]{Corollary}
\newtheorem{lemma}[theorem]{Lemma}
\newtheorem{proposition}[theorem]{Proposition}

\newtheorem{definition}[theorem]{Definition}

\newtheorem{assumption}[theorem]{Assumption}

\numberwithin{equation}{section}
\newcommand{\wh}{\check}
\newcommand{\wt}{\widetilde}

\newcommand{\N}{\mathbb {N}}
\newcommand{\Z}{\mathbb{Z}}

\newcommand{\HH}{\mathcal{H}}
\newcommand{\R}{\mathbb{R}}

\renewcommand{\P}{\mathbb{P}} 
\newcommand{\E}{\mathbb{E}} 
\newcommand{\T}{O} 
\newcommand{\A}{\mathscr{A}} 
\newcommand{\vp}{\varphi}
\newcommand{\FF}{\mathfrak{F}}
\renewcommand{\d}{\mathrm{d}} 
\newcommand{\EE}{\mathcal{E}}

\newcommand{\supp}{\mathrm{supp}}
\newcommand{\spa}{\text{span}}
\newcommand{\MM}{\mathcal{M}}

\newcommand{\D}{\mathscr{D}}
\renewcommand{\div}{\mathrm{div}}

\renewcommand{\L}{\mathscr{L}}
\newcommand{\px}{\partial_x}
\newcommand{\CC}{\mathscr{C}}
\newcommand{\pv}{p\text{-}\mathrm{Var}}
\newcommand{\ptv}{\tilde{p}\text{-}\mathrm{Var}}
\newcommand{\grad}{\mathrm{grad}}
\newcommand{\X}{\mathcal{X}}

\newcommand{\diff}{\mathrm{diff}}

\newcommand{\Ca}{\mathrm{Cap}}

\newcommand{\EEw}{\mathcal{E}_{\text{weak}}}
\newcommand{\EEf}{\mathcal{E}_{\text{flat}}}
\renewcommand{\tilde}{\widetilde}
\allowdisplaybreaks

\title[Dirichlet form approach to STFEs]{Dirichlet form approach to stochastic thin-film equations}
\date{\today}
\author[Gess,  Otto,  Sauerbrey]{Benjamin Gess\textsuperscript{1}, Felix Otto\textsuperscript{2}, Max Sauerbrey\textsuperscript{3}}
\thanks{\textsuperscript{1}TU Berlin and MPI MiS Leipzig, Germany (\href{mailto:benjamin.gess@gmail.com}{benjamin.gess@gmail.com}), \textsuperscript{2}MPI MiS Leipzig, Germany (\href{mailto:felix.otto@mis.mpg.de}{felix.otto@mis.mpg.de}), \textsuperscript{3}MPI MiS Leipzig, Germany (\href{mailto:maxsauerbrey97@gmail.com}{maxsauerbrey97@gmail.com})}

\subjclass[2020]{60J46, 60H17, 76A20} 
\keywords{Singular SPDEs, Dirichlet forms, thin-film equation, Mosco convergence, Gaussian analysis, energy solutions, approximation of SPDEs, regularity}

\renewcommand{\eta}{F}
\renewcommand{\zeta}{G}

\begin{document}
	
	\begin{abstract}
		The thin-film equation describes the evolution of a thin liquid film driven by capillarity and limited by viscosity. 
		For sufficiently small film heights, thermal fluctuations may not be neglected. Since the thin-film equation as an overdamped 
		evolution equation has a gradient flow structure, the informal  implementation of the fluctuation-dissipation principle is straightforward. In this paper, we work out that the concept of Dirichlet forms and the related construction of  Markov processes are tailor-made for this type of overdamped continuum evolution with thermal noise. However, the concept of Dirichlet forms relies on the availability of a Gibbsian equilibrium measure, which, in case of the thin-film equation, limits the treatment to the $(1+1)$-dimensional case.

		 By means of Dirichlet form theory, we obtain the existence and uniqueness of the evolution of laws on configuration space. 
		  We show that the regularity of the realizations of this Markov process  matches the heuristics. In a simplified situation with a Gaussian reference measure and a non-degenerate mobility function, we characterize it as  what is called an energy solution.  The  latter provides a pathwise decomposition of the driving forces into a deterministic  and a thermal one, which was the starting point.     
		
		Applying the same techniques in the spatially discretized setting proposed in [Math. Comp. 92 (2023), 1931--1976], we obtain the same results uniformly in the discretization parameter. In particular, Galerkin as well as spectral approximations of the stochastic thin-film equation are shown to be uniformly tight in optimal topologies. In the simplified setting, we also prove Mosco convergence of the associated Dirichlet forms and thereby obtain convergence of these approximations.
	\end{abstract}

	\maketitle
	\section{Motivation}

	\subsection{A finite dimensional analogue}\label{SS_fin_dim_analog}
	This work revolves around the analysis of thermodynamical models of thin-film type based on the following ansatz: Let us consider a {Riemannian manifold} $(\MM, g)$ equipped with a {normalized  measure of Gibbs type} 
	\begin{equation}\label{e:mu_beta}
		\d\mu  (h) \,\propto\, \exp( -  E(h))\, \d h.
	\end{equation} Here, $\d h$ denotes a  reference measure on a flat ambient space of $\MM$ and $E\colon \MM\to [0,\infty] $ a given  energy functional. Naturally, the metric $g_h $ on $\mathcal{T}_h\MM$ induces a dual metric on the space of cotangent  vectors $ \mathcal{T}_h^*\MM$ via the Legendre transformation 
	\begin{equation}\label{e:Legendre}
		\frac{1}{2} g^h  (p,p ) \,=\, \sup_{\dot h \in \mathcal{T}_h\MM} \biggl\{ 
		p_\cdot \dot{h } \,-\, \frac{1}{2} g_h( \dot{h },\dot{h }) 
		 \biggr\},\qquad p \in \mathcal{T}_h^*\MM.
	\end{equation}
	By polarization, the above extends to an inner product on $\mathcal{T}_h^*\MM$ and allows us to pair the differentials of two observables $\eta,\zeta  \colon \MM \to \R$. Using the measure from \eqref{e:mu_beta}, we can define a Dirichlet form by
	\begin{equation}\label{e:Dirichlet_form}
		\EE (\eta,\zeta ) \,=\,\int_{\MM} g^h ( \diff|_h \eta ,\diff |_h\zeta  ) \, \d  \mu(h).
	\end{equation}
	 The latter gives rise to a generator through the integration by parts formula
	\begin{equation}\label{e:generator}
		 \int_{\MM}  \eta \L \zeta\, \d  \mu \,=\,  - 	\EE (\eta,\zeta ) ,
	\end{equation}
	which, under mild additional assumptions, can be associated to an $\MM$-valued time-homogeneous Markov process with transition semigroup $(T_t)_{t>0}$ by the identity
	\begin{equation}\label{e:generation}
		\L \eta \,=\, \lim_{t\searrow 0} \frac{ T_t \eta -\eta} {t}.
	\end{equation}
	That the Dirichlet form \eqref{e:Dirichlet_form} is  symmetric ensures reversibility of the Markov process with respect to the Gibbs measure $\mu$. Defining moreover the Kolmogorov forward operator  $\L^*$ by
	\[
	\int_{\MM } \eta \,\d \L^* {\nu}  \,=\, 	\int_{\MM } \L \eta \,\d {\nu},
	\]
	we find that 
	\[
	 - \int_{\MM}  \eta \, \d \L^* {\nu}  \,=\,   \EE( \eta , \d {\nu} /\d \mu ) \,=\, \int_{\MM}  \diff  \eta_\cdot  \grad \frac{\d {\nu}} {\d \mu} \, \d  \mu.
	\]
	For $\nu =\mu$, the probability current vanishes therefore, so  the detailed balance condition is   satisfied. 
	Moreover, the locality\footnote{A Dirichlet form has the local property if $\EE(F,G) = 0$ for $F$ and $G$ with disjoint support.} of \eqref{e:Dirichlet_form} ensures  that the associated Markov process has continuous paths and therefore is a diffusion process. See Appendix \ref{app:DFs} for this and more facts from Dirichlet form theory.
	
	Many overdamped equations of continuum physics have the structure of a gradient flow. We believe that Dirichlet form theory provides an efficient framework to introduce thermal fluctuations into these evolutions. As an example, we treat here the thin-film equation. 
	
	\subsection{The thin-film equation as a gradient flow}\label{SS:example}
	Let us specify an informal choice of $(\MM,g)$ and $E$ describing the evolution of a thin liquid film supported on a flat substrate and driven by capillarity. We consider a container with bottom\footnote{We choose to work with non-periodic domains because this choice leads to a simpler variational interpretation of the spatially discretized setting proposed in \cite{GGKO21}.} $O\subset \R^d$, for $d\in \{1,2\}$. 
	 While $d=2$ is the physical (effective) dimension\footnote{The total dimension of the problem is $d+1$.}, $d=1$ corresponds to a fluid-film that is confined in one spatial direction, cf.\ \cite{DMES2005}. Thereby, the manifold of possible fluid-film configurations described by their height profile $h$ is
		\begin{equation}\label{e:M_example}
			\mathcal{M}^+\, =\, \biggl\{ h\colon \T \to [0,\infty)  \, \bigg|\, \int_{\T} h\, =\, 1\biggr\}, 
		 \end{equation}
	 {where $ \int_{\T} h\, =\, 1$ can be assumed  by a suitable non-dimensionalization.}
	 The tangent space at a point $h \in \MM^+$ is accordingly given by all mass-preserving directions
	 \begin{equation}\label{e:TM_example}
	 	\mathcal{T}_h\mathcal{M}^+\, =\, \biggl\{\dot  h\colon \T \to \R  \, \bigg|\, \int_{\T} \dot  h\, =\, 0\biggr\},
	 \end{equation} 
 	and equipped with the metric tensor
 	\begin{align}\label{tensot_j}
 		g_h(\dot{h},\dot{h}) \,=\, \inf_j \biggl\{
 		\int_{\T} \frac{|j|^2}{m(h)}  \,\bigg|\,\forall\tilde{p}: \, 
 		\int_O j\cdot \nabla \tilde{p}  \,=\, \int_O  \dot{h} \tilde{p}  
 	 		\biggr\},
 	\end{align}
 In addition to imposing that $ \nabla \cdot j + \dot{h}= 0$ in $O$, the above variational formulation also enforces that  $\nu \cdot j =0$ on $\partial O$,
 where $\nu$ denotes the outer unit normal vector. The latter no-flux boundary condition models the non-penetrability of the vertical container walls. 
 	The   {mobility function} $m \colon (0,\infty) \to (0,\infty)$ is typically given by a power law $m(h) = h^n$ for $n\ge 0$ depending on different modeling choices: Assuming that the fluid is governed by the Stokes equation and  satisfies a no-slip boundary condition at the fluid-substrate interface gives $n=3$, while a Navier slip condition results in $n=2$ for heights below the slippage length. Darcy's law instead of viscous dissipation results in the choice $n=1$, in which case the above induces the Wasserstein distance. 
 	
 	The Euler--Lagrange equation for \eqref{tensot_j} yields  that $j/m(h)$ admits a potential $p$ satisfying natural boundary conditions, so that we may  alternatively express
 		\begin{equation}\label{e:metric}
 			g_h (\dot{h},\dot{h}) \,=\, \int_{\T} m(h) |\nabla p|^2 , 
 			\qquad \forall\tilde{p}: \, 
 			\int_O m(h) \nabla p \cdot \nabla \tilde{p}  \,=\, \int_O  \dot{h} \tilde{p}  .
 	\end{equation}
 	   Then, the gradient flow of the surface energy 
 	\begin{equation}\label{e:Energy}
 		E\colon \MM^+\to [0,\infty] , \, h\mapsto \frac{1}{2}\int_{\T} |\nabla h|^2,
 	\end{equation}  
 	on $(\MM^+,g)$ is seen to be the thin-film equation\footnote{The function $-p$ is the capillary pressure.}
 	\begin{equation}\label{e:TFE}
 		\begin{cases}
 		\partial_t h + \nabla \cdot( m(h)\nabla p)\,=\, 0 \quad \text{and}\quad p = \Delta h, & \text{in $O$},\\
 		\nu \cdot m(h)\nabla p = 0  \quad \text{and}\quad \nu \cdot \nabla h = 0, & \text{on $\partial O$}, 
 	\end{cases}
 	\end{equation}
 which traditionally is obtained from  lubrication approximation of the full fluid dynamics, see  \cite{ODB97}. The second boundary condition is a natural one, an equilibrium boundary condition that 
 	arises from the modeling choice that the surface energy \eqref{e:Energy} does not contain a term coming from the interface between the fluid and the vertical container wall. For  more information on the underlying  geometry of $(\MM^+,g)$ for $n=1$ we refer  to 
 	\cite{Otto_PME}.
 	
 	\subsection{Rigorous definition of the pre-Dirichlet form}\label{SS:pre_DF}
 	Following  Subsection \ref{SS_fin_dim_analog}, we use the Dirichlet form ansatz \eqref{e:Dirichlet_form}  to model  thermal fluctuations for thin films in line with the detailed balance condition. Here we demonstrate that despite the infinite-dimensional setting, the Dirichlet form may be rigorously defined on cylinder functions. Indeed, for the differential one-form   
 	\begin{equation}\label{E10}
 		\dot{h} \mapsto \int_{\T}  \dot{h} p
 	\end{equation} 
 	we may evaluate 
 	\begin{align}\label{E39872}
 	g^{h} (p,p ) \,=\, \int_{\T} m(h) |\nabla p|^2 ,
 	\end{align}
 	based on \eqref{e:Legendre} and \eqref{e:metric}.
 	For a cylinder function
 	\begin{equation}\label{e:cyl_fct}
 		\eta(h)\,=\, \Phi\biggl(\int_{\T}  h p^1,\dots , \int_{\T}  h p^k \biggr),
 	\end{equation}
 	the above gives then rise to the formula 
 	\begin{align}&
 		g^{h} (\diff\eta|_h ,\diff \eta |_h)
 		 \\&\quad=\,\sum_{i,j=1}^k  
 		\biggl(	\int_{\T} m(h) \nabla p^i \cdot \nabla p^j \biggr) 
 		\partial_i \Phi \biggl(\int_{\T}  h p^1 ,\dots , \int_{\T}  h p^k \biggr)
 		 \partial_j \Phi \biggl(\int_{\T}  h p^1 ,\dots , \int_{\T}  h p^k \biggr),\label{eqn412}
 	\end{align}
 	for the integrand appearing in  \eqref{e:Dirichlet_form}. To also give meaning to the Gibbs measure \eqref{e:mu_beta}, we use that $2 E(h)$  is the squared Cameron--Martin norm of the Gaussian free field, so that one way to define
	\begin{equation}\label{eqn9}
	\d \mu^+(h) \,\propto\,  \exp( -  E(h))\, \d h \,=\, \mathbf{1}_{\MM^+}(h)
	\exp( -  E(h))\, \d h,
	\end{equation}
	 is to center said   field at $\mathbf{1}_{\T}$ and to subsequently restrict it to $\MM^+$. While this amounts to conditioning on a set of positive probability if $d=1$, it is unclear how to achieve this  in the critical dimension $d= 2$, in which the point evaluation is not continuous in the topology of the Cameron--Martin space. 
	 We remark that writing 
	 \begin{align}\label{eqn_N}
		\mathbf{1}_{\MM^+}( h ) = \exp( -V (h)) ,\qquad V(h) =\begin{cases}0 ,& h\in \MM^+ ,\\
			\infty , & \text{else},
		\end{cases} 	 
 	 \end{align}
  suggests a similarity of \eqref{eqn9} to the  $\phi_d^4$-measure, which can in $d\in \{2,3\}$ despite similar problems be constructed as a renormalized measure. The worse singularity of the potential in \eqref{eqn_N} compared to $V(h)= \int_O h^4 $ prevents, however, a direct adaption of various renormalization  methods.

	 For these reasons, we restrict ourselves for the remainder of this paper   to the subcritical $(1+1)$-dimensional setting by taking 
	\begin{align}
		\T \,= \,(0,1).
	\end{align} 
 	In this case, integrability of  \eqref{eqn412} with respect to $\mu^+$ is guaranteed  if $p^i \in H^1(\T)$,  $\Phi $ has continuous and  polynomially bounded derivatives, and $m(h)$ is polynomially bounded in $h$. We can therefore rigorously define \begin{equation}\label{e:Dirichlet_form_+}
	 		\EE (\eta,\zeta ) \,=\,\int_{\MM^+} g^h ( \diff|_h \eta ,\diff|_h \zeta ) \, \d  \mu^+(h)\,=\, 
	 		\int_{\MM^+} \int_{\T} m(h) \partial_x  \diff|_h \eta \partial_x \diff|_h \zeta  \,\d x\,  \d  \mu^+(h)
	 		,
	 	\end{equation}for such cylinder functions,  giving rise to a pre-Dirichlet form on $L^2(\mu^+)$.
 	
As a disclaimer, we point out  that the evolution of a Markov process obtained from Dirichlet form theory is  determined only up to so-called exceptional sets of initial data, which are in particular of zero $\mu^+$-measure. 
In particular, the construction does not guarantee the existence of a diffusion process starting from every prescribed deterministic initial datum.
 Since by entropic repulsion\footnote{This follows from the law of the infimum being absolutely continuous, which may, e.g., be seen using explicit calculations \cite[Section 23]{janson} or proved using Malliavin calculus \cite[Theorem 2.1.2]{Nualart_related}.} 
\begin{align}\label{ent_rep}
\inf_{x\in \T} h(x)>0,\qquad \text{$\mu^+$-a.e.},
\end{align}
we also make no claims in this manuscript about the stochastic thin-film equation started from initial values with a contact line, i.e., a triple junction of gas, liquid and solid.
\section{Results and discussion}
\renewcommand{\thetheorem}{\Roman{theorem}}
\setcounter{theorem}{0}

\subsection{Existence and regularity of an associated diffusion process.}As  \eqref{e:Dirichlet_form_+} is merely a pre-Dirichlet form, one needs to pass to its closure in order to obtain an associated diffusion process. In this infinite dimensional and degenerate setting, closability is a non-trivial issue. A common strategy to show the latter is to prove an integration by parts formula, leading to an identification of its generator via \eqref{e:generator}. As we detail in Subsection \ref{SS:generator_is_diff} below, non-constant cylinder functions typically do not lie in the domain of the generator of \eqref{e:Dirichlet_form_+}, so that this strategy is not applicable in our setting of interest. Instead, we use the results from \cite{AR_90_closability,AR_91} to argue that for mobility functions such that {
\begin{align} \label{eqn111}\text{$m\colon(0,\infty) \to (0,\infty)$ is continuous,}\qquad
	\inf_r \frac{m(r)}{\rho(r)} >0, \qquad\text{and}\qquad  \sup_r \frac{m(r)}{(1 \vee r)^{n}}<\infty ,
\end{align}
for some  $n\in [0,\infty)$ and $\rho\colon (0,\infty) \to (0,\infty)$, which is bounded away from $0$ on intervals $[r,\infty)$, the pre-Dirichlet form has a closed extension.} In particular, we rely on a combination of the above lower bound with \eqref{ent_rep} leading to the lower bound\footnote{Throughout, dotted spaces refer to spaces of mean-free functions or distributions equipped with their homogeneous norm.} 
\begin{equation}
	g^h(\diff|_h\eta ,\diff|_h \eta)\,\gtrsim\, \rho(h) \|\diff \eta|_h\|_{\dot{H}^{1}(\T)}^2,
\end{equation}
for a positive density function {still denoted by $\rho$}. 
The upper bound ensures then tightness of the induced capacity in the topology of $(H^{1}(\T))^*$ by \cite{rock_schmu_tightness}, which yields the 
existence of  an $\MM^+$-valued diffusion associated to the closure of \eqref{e:Dirichlet_form_+}. All this is executed in a more general quasi-Gaussian setting, covering for instance also stochastic thin-film equations affected by van der Waals forces between fluid and substrate leading to an additional potential term in \eqref{e:Energy}.

Coming back to \eqref{e:Dirichlet_form_+}, the obtained Markov process is guaranteed to have continuous paths in $(H^{1}(\T))^*$. Comparing to the case  $m\equiv 1$, one may expect it however to lie in 
\begin{equation}\label{e:Holder}
C^{1/8-,1/2-}_{t,x}\,=\,\Bigl\{
h\colon [0,\infty) \times \T\to \R\,\Big|
\, \forall T\in (0,\infty),  \,\gamma\in (0,1/2):\,
h\in C^{\gamma/4,\gamma}([0,T]\times \T)
\Bigr\},
\end{equation}
at least
when started from equilibrium. To show this regularity assertion in the more general setting, we interpolate the spatial regularity of $\mu$ with annealed estimates on the increments of suitable additive functionals resulting from decomposing them into forward and backwards martingales, cf.\ \cite{LyonsZheng}. The latter uses the upper bound of \eqref{eqn111} as well as the reversibility of the diffusion process. All this results in Theorem \ref{thm_existence_deg_case}, which can be stated as follows. 
\begin{theorem}[Existence \& Regularity]\label{thmA}If the mobility function satisfies \eqref{eqn111} for some  $n$ and $\rho$ as above, the pre-Dirichlet form \eqref{e:Dirichlet_form_+} is closable on $L^2(\mu^+)$ and  its closure admits an associated $\MM^+$-valued diffusion process. The paths of the latter lie almost surely  in $C_{t,x}^{1/8-, 1/2-}$ when started from $\mu^+$.
\end{theorem}\noindent
Since, as argued at the beginning of this introduction, the detailed balance condition is satisfied, the resulting diffusion process  provides a natural model for  thermal fluctuations of  thin fluid-films. Moreover, comparison to the case  $m\equiv 1$ shows that the 
obtained $C_{t,x}^{1/8-, 1/2-}$-regularity is optimal.

\subsection{Identification and uniform tightness of Markov processes resulting from discretization} \label{SS_tightness_dircscrete} We also follow up on the work \cite{GGKO21}, where a spatially discrete situation is considered. More precisely, one may discretize the above by restricting the set of configurations to
\begin{equation}\label{eqn418_new}
	\MM^+_N \,=\, \bigl\{ h\in \MM^+ \,\big|\, \text{$h$ is linear on $[x_i,x_{i+1} ]$}\bigr\},
\end{equation}  
for {$N+1$ equidistant  points  $0=x_0< \dots < x_N =1$} and replacing the metric \eqref{e:metric} by
\begin{equation}\label{E9}
	\tilde{g}_{h} (\dot{h},\dot{h}) \,=\, \int_{\T} m(h) |\px  {p}|^2   , \qquad
	\forall\tilde{p}:\, 
	\int_O m(h) \px p \px \tilde{p}  \,=\,  \frac{1}{N}\biggl(\frac{\dot {h}(0)\tilde p ({0})+\dot {h}(1)\tilde{p}(1)}{2}+ \sum_{i=1}^{N-1} \dot {h}(x_i)\tilde{p}(x_i)\biggr).
\end{equation}  
Then, 
the gradient flow of the surface energy coincides with the positivity preserving scheme for  thin-film equations from  \cite{gruen_rumpf}. See also Subsection \ref{SS:GR-scheme} for a discussion of this choice of metric and, e.g., \cite{wuchen} for other discretizations tailored  toward the Wasserstein diffusion \cite{Renesse_Sturm}. As metric and energy already determine a gradient Dirichlet form, the latter gives rise to the approximation
\begin{equation}\label{e:Dirichlet_form_+_discr}
\widetilde{\EE}_{N} (\eta, \eta )\,=\, 
\int_{\MM_N^+} \wt{g}^h ( \diff |_h\eta  ,\diff |_h \eta ) \, \d  \mu_{N}^+(h),
\end{equation}
of \eqref{e:Dirichlet_form_+}, 
where 
	\begin{equation}\label{defi_mu_N+}\d \mu_{N}^+ \,\propto\,  \exp( -  E(h))\, \d h \,=\, \mathbf{1}_{\MM_N^+}(h)
	\exp( -  E(h))\, \d h.\end{equation}
Unlike in the infinite dimensional setting, for continuously differentiable mobility functions, one can  integrate by parts in \eqref{e:Dirichlet_form_+_discr} leading to an identification of the generator as
\begin{align}\label{eqn1}
	\widetilde{\L}_N \eta \,=\, \wt{g}^{\alpha \alpha'} \partial_{\alpha\alpha'} \eta \,+\, \partial_\alpha \eta \bigl(\partial_{\alpha'}(\wt{g}^{\alpha'\alpha}) -\wt{g}^{\alpha \alpha'} \partial_{\alpha'}E\bigr),\qquad \eta \in C^\infty_c(\MM_N^+),
\end{align} where $(h^\alpha)_\alpha$ is a flat coordinate chart and  we use  Einstein's summation convention. This implies that for any $\eta \in C^\infty(\MM_N^+)$, until an associated Markov process reaches the boundary of $\MM_N^+$,
\begin{align}\label{Eq191}
	\eta(h_t) - \eta (h_0) - \int_0^t \wt{g}^{\alpha \alpha'} \partial_{\alpha\alpha'} \eta(h_s) \,+\, \partial_\alpha \eta \bigl(\partial_{\alpha'}(\wt{g}^{\alpha'\alpha}) -\wt{g}^{\alpha \alpha'} \partial_{\alpha'}E\bigr)(h_s)\, \d s 
\end{align} 
is a martingale in the filtration generated by $h$. 
One of the main results of \cite{GGKO21} is that said martingale problem preserves strictly positive film heights, i.e., that $h_t(x_i)>0$ for all $i=0,\dots , N$, $\P$-a.s., if $h_0(x_i)>0$, under the condition  \begin{equation}\label{ass_n}
		\sup_r \frac{ m(r)}{  r^{n}}<\infty ,
	\end{equation}
for $n\ge 3$. 
In particular, the martingale problem is solved globally in time. Passing to the equivalent formulation as an SDE, this leads to a positivity preserving numerical scheme for its solution.
 In terms of Dirichlet form theory this result may be interpreted  as Markov uniqueness for the operator \eqref{eqn1}: Any Markov process associated to an extension of $\wt{\L}_N$ is uniquely determined by the solution to the above martingale problem and therefore there is only one such extension. 
 
 For $N=1$, however,  the  Feller test shows that  solutions to the resulting SDE preserve positivity also for $n\in [2,3)$. Motivated by said gap in exponents, we further extend these insights in Theorem \ref{thm_nonnegativity}, which asserts the following.
\begin{theorem}[Improved positivity of numerical scheme]\label{thmC}
	Assume \eqref{ass_n} for $n \ge 2$ and a continuously differentiable mobility function $m$. Then, if  $h$ is a solution to the martingale problem \eqref{Eq191}  
	with $\P(\forall i :  h_0(x_i) >0) = 1$, we have  $\P(\forall i, t\ge 0 : h_t(x_i) >0) =1$. 	 
\end{theorem}
The above is obtained by  refining  the argument from \cite{GGKO21}. 
Involving additionally Dirichlet form methods, we obtain also the assertion of Theorem \ref{thmA}  uniformly in $N$, see  Theorem \ref{thm_existence_deg__fin_case} below, and connect the Dirichlet form \eqref{e:Dirichlet_form_+_discr} to the above martingale problem for   $n\ge 1$ in Theorem \ref{thm_nonnegativity}.
\begin{theorem}[Tight, non-negative solutions to numerical scheme]\label{thmB}Assume \eqref{eqn111} for some  $n\in [0,\infty)$. 
	\begin{enumerate}[(i)]
		\item\label{I1}
	The pre-Dirichlet form \eqref{e:Dirichlet_form_+_discr} is closable on $L^2(\mu_N^+)$ and  its closure admits an associated $\MM^+_N$-valued diffusion process. The latter are uniformly in $N$ tight on $C_{t,x}^{1/8-, 1/2-}$ when started from $\mu_N^+$.
	\item\label{I2} Assume additionally \eqref{ass_n} for $n\ge 1$ and  that $m$ is continuously differentiable and increasing.  
	Then the  diffusion associated to the closure of  \eqref{e:Dirichlet_form_+_discr} solves the martingale problem \eqref{Eq191} when started from $\mu_N^+$.
\end{enumerate}
\end{theorem}
While the former assertion \eqref{I1} relies on similar arguments as Theorem \ref{thmA}, the latter \eqref{I2} is obtained using Fukushima's decomposition of additive functionals.
In combination, they yield that if $n$ from \eqref{ass_n} is sufficiently large,  then the solutions to the martingale problem are tight in the space in which the $\MM^+$-valued diffusion associated to  \eqref{e:Dirichlet_form_+} is supported.  This suggests that one can show convergence by identifying  weak limit points, implying convergence of the numerical algorithm from \cite{GGKO21}.  Achieving this poses, however, challenges due to the infinite dimensional setting of the limiting dynamics: On the one hand, working directly with the generator would require  identifying functions in its domain, while, as we detail in Subsection \ref{SS:generator_is_diff} below, typical observables do not allow for an integration by parts in \eqref{e:Dirichlet_form_+}. Working on the other hand with Mosco convergence would require an identification of weak and strong Sobolev spaces on $\MM^+$ to argue that natural upper and lower limiting forms coincide. Also this is difficult as we deal with a degenerate and infinite dimensional setting as well as with the presence of a boundary $\partial \MM^+$. 

\subsection{Approximation and characterization of the associated diffusion in a non-degenerate setting.}\label{SS_non-deg} Given the just mentioned obstructions, we also consider a simplified setting, where we remove the non-negativity condition as well as the degeneracy of $m$. More precisely, in order to gain understanding of the diffusion process in the continuum we consider also the manifold 
	\begin{equation}\label{e:M_example_flat}
	\mathcal{M}\, =\, \biggl\{ h\colon \T \to \R  \, \bigg|\, \int_{\T} h\, =\, 0\biggr\}
\end{equation}
equipped with the mean-free Gaussian free field $\mu$,   
and impose the assumption that
\begin{equation}\label{assumption_simplified}
 0< c\le  m(r) \le C<\infty.
\end{equation}
In this case, since positivity is immaterial, next to 
\begin{equation}\label{E1}
	\wt{\EE}_N (\eta ,\eta ) = \int_{\MM_N} \wt{g}^h ( \diff |_h\eta ,\diff |_h\eta ) \, \d  {\mu}_{N}(h),
\end{equation}
defined analogously to \eqref{e:Dirichlet_form_+_discr}, one may also consider approximations to 
	\begin{equation}\label{e:Dirichlet_form-2}
	\EE (\eta,\eta ) \,=\,\int_{{\MM}} g^h ( \diff  |_h\eta ,\diff|_h  \eta) \, \d  \mu(h)
\end{equation}
	based on projections onto the cosine basis. Namely, we may set $\check{\MM}_N= \Pi_{\le N}\MM$, where 
	\begin{equation}\label{eqn73343}
	\Pi_{\le N} h(x) \,=\, 2\sum_{k=1}^N\biggl( \int_\T  \cos(\pi k \cdot) h  \biggr)  \cos(\pi k x),
	\end{equation}
	 so that also $\mathcal{T}_h \wh{\MM}_N = \Pi_{\le N} \mathcal{T}_h\mathcal{M}$. We adjust also here the metric tensor \eqref{e:metric}, but this time  by 
		\begin{equation} \label{eqn_spectral_Galerkin_metric}
		\wh{g}_h (\dot{h},\dot{h}) \,=\, \int_{\T} m(h) |\partial_x {p}|^2 , \qquad \forall \tilde{p}:\, \int_O  m(h) \partial_x  p \partial_x \tilde p \,=\,\int  \dot {h}\tilde{p},
	\end{equation}
	where we restrict ourselves in the above to $p,\tilde{p} \in \Pi_{\le N}\mathcal{T}_h^*\mathcal{M}$.
	This yields another approximate  form 
	\begin{align}\label{E2}
		\check{\EE}_N (\eta,\eta ) \,=\,\int_{{\check{\MM}_N}} \check{g}^h ( \diff|_h \eta  ,\diff|_h \eta ) \, \d  \check{\mu}_N(h).
	\end{align}

	The assertions of Theorem \ref{thmA} and Theorem \ref{thmB}~\eqref{I1} carry over to the current setting. This  implies again uniform in $N$ tightness of the processes associated to \eqref{E1} and \eqref{E2} in $C_{t,x}^{1/8-, 1/2-}$. In the non-degenerate setting, we succeed, however, in  identifying their limit by showing that $\tilde{\EE}_N, \check{\EE}_N\to \EE$ in Mosco sense and therefore the associated processes converge, cf.\  Theorem \ref{thm_convergence}.
	\begin{theorem}[Convergence in a simplified setting]
		\label{thmD} Assume \eqref{assumption_simplified} and that $m$ is continuous.  When started from equilibrium, the diffusion processes associated to \eqref{E1} and \eqref{E2} converge in law on  $ C_{t,x}^{1/8-, 1/2-}$ to the $\MM$-valued diffusion  associated to \eqref{e:Dirichlet_form-2}.
	\end{theorem}
	
	As for the spatial discretization, assuming that $m$ is continuously differentiable and that $m'$ is additionally polynomially bounded, one finds that the  Markov process associated to \eqref{E2} is a solution to the  problem that \eqref{Eq191}, with $\tilde{g}$ replaced by $\check g$,
	is a martingale
	for any flat chart $(h^\alpha)_\alpha$ of $\check{\MM}_N$.
	
	Classically, the latter can be reduced to linear observables if we also prescribe the quadratic variation, since then the martingality  for general $\eta$ follows by It\^o's formula. More precisely, after calculating the energy gradient, we equivalently ask that for any smooth $p$ satisfying Neumann boundary conditions
	\begin{align}\label{E5}\quad&
			\int_{\T} h_tp \,-\, 
		\int_{\T} h_0p   -\,\int_0^t   \int_{\T} m( h_s)\px^3h_s 
	 \px	\Pi_{\le N} p 	
			\,+\, \frac{\partial_x (D_N(2\pi  \cdot ))}{4}      m'(  h_s)  \partial_x \Pi_{\le N } p \,\d x \,\d s 
	\end{align}
	is a martingale with respect to the filtration generated by $h$ and has the quadratic variation process
	\begin{align} \label{E35}2 \int_0^t  \check{g}^{h_s}( p,p )\,\d s \,=\,    {2}
		\int_0^t \int_{\T} 
		m(h_s) | \px\Pi_{\le N} p|^2
		\,\d x\, \d s.
	\end{align}
We remark that the term involving the Dirichlet kernel
\begin{equation}\label{eqn_def_Dirichlet kernel}
	D_N (\theta)  = 1+ 2\sum_{k=1}^N\cos(k \theta) 
\end{equation}
corresponds to the coefficient $\partial_{\alpha'}(\check{g}^{\alpha'\alpha})$ in \eqref{Eq191}. Moreover, essentially due to the unitary action of shifts on periodic functions, the latter would precisely cancel if we assumed $O$ to be periodic. This is reflected by the Dirichlet kernel, and therefore the geometric correction term in \eqref{E5},  concentrating as $N\to\infty$  near the boundary of $O$,  where this shift invariance is broken.\footnote{For a  further discussion, see Subsection \ref{SS:ren_of_SPDE} below.}

	In  the recent theory of energy solutions to singular SPDEs like the KPZ equation, more general martingale problems are considered in which the term subtracted in \eqref{E5} is merely a limit of explicit additive functionals, cf.\ \cite{GJ_ARMA}. The need for  this relaxation comes from the limited regularity of solutions to singular SPDEs, and also the diffusion associated to \eqref{e:Dirichlet_form-2} does not allow for an evaluation of e.g.\
	\[\int_0^t \int_{\T} m( h_s)\px^3 h_s 
 \px p \,\d x	\,\d s ,
	\]
	since it lies only in $C^{1/8-,1/2-}_{t,x}$, $\P$-almost surely. Nevertheless,  solutions to such generalized martingale problems may  be shown to be unique \cite{GubPer_uniqueness} and the energy solution concept connects naturally to Dirichlet form theory, as observed in \cite{zhuzhu_phi43} in the context of the $\phi^4_3$-equation. In Theorem \ref{thm_characterization} we  characterize the diffusion associated to \eqref{e:Dirichlet_form-2} in the same spirit---we state a concise version here, for which we set
	\begin{align} \label{eqn_H3N}
		{H}^3_N(\T) \,=\, \{
		p\in H^3(\T) | \partial_x p (0 )  = \partial_x p(1) = 0
		\}.
	\end{align}
\begin{theorem}[Characterization in a simplified setting]
	\label{thmE} Assume \eqref{assumption_simplified} and that $m' $
	is continuous and bounded. 
	When started  from $\nu$ with $\frac{\d \nu}{\d \mu}\in L^2(\mu)$, the $\MM$-valued diffusion associated to \eqref{e:Dirichlet_form-2} is characterized as an energy solution. In particular, for every  $p\in {H}^3_N(\T)$, the process
	\begin{align}\begin{split}\label{E6}
			&
			\int_{\T} h_tp \,-\, 
			\int_{\T} h_0p \\&\quad  -\, \lim_{N\to \infty} \int_0^t   \int_{\T} m( \Pi_{\le N} h_s)\px^3\Pi_{\le N}h_s 
			 \px\Pi_{\le N} p \,+\, \frac{\partial_x (D_N(2\pi  \cdot ))}{4}     m'(\Pi_{\le N}  h_s)  \partial_x \Pi_{\le N }p    \,\d x	\,\d s 
		\end{split}
	\end{align}
	is a martingale with respect to the filtration generated by $h$ and has quadratic variation
	\[{2}
	\int_0^t \int_{\T} 
	m(h_s) |\px p|^2
	 \,\d x\, \d s.
	\]
\end{theorem}
In particular, showing that the diffusion associated to \eqref{e:Dirichlet_form-2} is an energy solution follows readily from Fukushima's decomposition of the additive functional $	\int_{\T} h_tp $. A crucial step to show the reverse direction is to justify It\^o's formula in order to obtain a (generalized) semimartingale decomposition also for more general observables than the one in \eqref{E6}. For this we use that we understand the generator at least in its weak form  $\L\colon \D(\EE) \to \D(\EE)^*$, to then apply arguments developed in \cite{GubPer_generator} for the stochastic Burgers equation.

\setcounter{theorem}{0}
\renewcommand{\thetheorem}{\thesection.\arabic{theorem}}

\subsection{On the choice of discretization leading to the Gr\"un--Rumpf scheme}\label{SS:GR-scheme}
The precise choice of metric \eqref{E9} for the discretization in the   degenerate setting described in Subsection \ref{SS_tightness_dircscrete} is motivated by its interplay with the entropy function 
\[
S(h) = \frac{1}{N}\biggl( \frac{s(h(0)) + s(h(1))}{2}+\sum_{i=1}^{N-1}  s(h(x_i)) \biggr),\qquad s(h) = \int_1^h \int_1^r \frac{1}{m(r')}\,\d r'\,\d r,
\]
as discovered in \cite{gruen_rumpf}, see also \cite{GGKO21}. Index the intervals $I_\alpha=[x_i,x_{i+1}]$ by ${\alpha=1,\dots, N}$,  and write $\alpha+=i+1$ and $\alpha-=i$. Then, one may introduce the zigzag basis 
of $\mathcal{T}_h\MM_N^+$ through
\begin{align}\label{eqn_zz}
	\bar{\vp}_\alpha ( x_{\alpha \pm}) \,=\, \pm {N^{3/2}(1+\mathbf{1}_{\{0,N\}}(\alpha\pm) )} .
\end{align}
Taking  $\dot{h} = 	\bar{\vp}_\alpha $, the equation 
 \eqref{E9} may be solved explicitly, so that 
\[
	\wt{g}_h (\bar{\vp}_\alpha,\bar{\vp}_\alpha)\,=\, N \int_{x_{\alpha-}}^{x_{\alpha+}} \frac{1}{m(h)}\,=\, \fint_{I_\alpha} \frac{1}{m(h)},
\]
and $	\wt{g}_h (\bar{\vp}_\alpha,\bar{\vp}_{\alpha'}) = 0$ for $\alpha\ne \alpha'$. In particular, the metric tensor diagonalizes and may be inverted leading to 
\begin{align} \label{E11}
	\wt{g}^{\alpha\alpha}(h) = \wt{g}^h(\diff h^\alpha , \diff h^{\alpha}) \,&=\, \biggl(\fint_{I_\alpha} \frac{1}{m(h)} \biggr)^{-1}
	\,=\,  \frac{h^{\alpha+} - h^{\alpha-}      }{s'(h^{\alpha+}) -s'(h^{\alpha-})} ,
\end{align}
for the coordinates defined through  $h=1+\sum_{\alpha=1}^N h^\alpha \bar{\vp}_\alpha$. Introducing also
the finite difference operator $A^\alpha_i h^i = N(h^{\alpha+} - h^{\alpha-})$, where $h^i = h(x_i)$, we find that
\[
\partial_\alpha S(h)\, =\, \diff S|_h \bar{\vp}_\alpha \,= \,\delta_{\alpha\alpha'}A^{\alpha'} s'(h) /\sqrt{N}
\] 
and by  \eqref{E11} 
\begin{equation}\label{E12}
\wt{g}^{\alpha\alpha'}(h) \partial_\alpha S(h)= A^{\alpha'}  h/\sqrt{N}.
\end{equation}
To compute also the differential of the energy \[
E(h) \,=\, \frac{1}{2 N}\sum_{\alpha} ((Ah)^\alpha )^2,
\]
we introduce $(A^\top )^i_\alpha =
A_i^\alpha$, so that 
\begin{align}\label{eqn_honig}
	\partial_i E(h) \,=\, 
	\frac{1}{		N} 
	(A^\top A h)^i \qquad \text{and} \qquad 
	\partial_\alpha  = \sqrt{N} (A^{\top})^i_\alpha (1+\mathbf{1}_{\{0,N\}}(i)) \partial_i .
\end{align}
 Thereby, we may evaluate the pairing in \eqref{Eq191} as
\begin{align}
	\label{eqn43ubv}
	\partial_\alpha S \wt{g}^{\alpha \alpha'} \partial_{\alpha'}E(h) \,=\,&
 \frac{1}{N}{\biggl({2((A^\top A h)^0)^2
 		+2((A^\top A h)^N	)^2
 	}+\sum_{i=1}^{N-1} \bigl((A^\top A h)^i \bigr)^2 \biggr),}
\end{align}
{mimicking the dissipation in the entropy  estimate}
\[
\int_{\T}s(h_t)\, + \,  \int_0^t \int_{\T} |\Delta h_s|^2 \,\d x\, \d s
\,\le\, \int_{\T} s(h_0)
\]
for solutions to the thin-film equation \eqref{e:TFE}, used in \cite{BERNIS1990} for instance.
In the discrete case, this allows to obtain  estimates on 
$\sup_{t } S(h_t)$
for solutions to the numerical scheme suggested in \cite{gruen_rumpf} as well as solutions to the martingale problem \eqref{Eq191}, cf.\ \cite{GGKO21}. These estimates  then imply strict positivity, at least in the discrete setting.  This is precisely the starting point for \cite[Theorem 8.4]{GGKO21} and Theorem \ref{thmC}.
\subsection{Renormalization of SPDEs of  stochastic thin-film  type}
\label{SS:ren_of_SPDE}
 Independently, in \cite{DMES2005} and \cite{GruenMeckeRauscher2006}, the  SPDE
 \begin{equation}\label{e:STFE}
 	\partial_t h\,=\, -\nabla \cdot (m(h)\nabla \Delta  h ) \,+\,\sqrt{2 } \nabla \cdot \bigl(m^{1/2}(h)\xi\bigr) ,
 \end{equation} 
was suggested as a model for a thin liquid film driven by capillarity and  thermal fluctuations. While the authors of \cite{DMES2005} derived \eqref{e:STFE} from the fluctuation dissipation relation, the above equation was obtained in \cite{GruenMeckeRauscher2006} as a lubrication limit of the Landau--Lifshitz--Navier--Stokes equation.  The driving noise $\xi$ is assumed to be white in space and time and therefore \eqref{e:STFE} is already  in dimension $d=1$ a singular SPDE,
while in $d=2$ the above becomes even scaling-critical. 

If $m$ is taken non-degenerate as in \eqref{assumption_simplified}, Theorem \ref{thmD} yields insights into the need of renormalizing  the SPDE \eqref{e:STFE}. More precisely, using the martingale representation theorem\footnote{See \cite[Theorem 8.2]{dapratozabczyk}, up to possibly  enlarging the probability space.}, we find that solutions to \eqref{E5}--\eqref{E35}
 solve the SDE
\begin{equation}\label{Eq_SDE1_new}
	\d h \,=\, -\Pi_{\le N} ( \px (m (h) \px^3 h ))\, \d t\, -\,\frac{1}{4} \Pi_{\le N } \partial_x  \bigl(  m'(  h) \partial_x( D_N(2\pi  \cdot ))  \bigr) \, \d t \,+\, \sqrt{{2}} \Pi_{\le N}(\px (\sqrt{m}(h) \d W)) ,
\end{equation}
for a cylindrical Wiener process $W$. This SDE may be regarded as a spectral Galerkin approximation of \eqref{e:STFE}. 
Similarly, representing  \eqref{Eq191} as a stochastic integral yields, as shown in \cite{GGKO21}, that\footnote{The additional factors of the form ${(1+\mathbf{1}_{\{0,N\}}(i))}$ compared to \cite[eqn. (8.3)]{GGKO21} are due to defining $\bar{\vp}_\alpha$ as in \eqref{eqn_zz} ensuring that they are mean free and the resulting correction in the chain rule \eqref{eqn_honig}. Because of these factors, the resulting SDE \eqref{Eq_SDE2_new} does preserve the mean $\int_O h = \frac{h^0 + h^N }{2N} +\frac{1}{N} \sum_{i=1}^{N-1} h^i$.} 
\begin{align}\begin{split} \label{Eq_SDE2_new}\qquad 
		\d h^i\,=\,&(1+\mathbf{1}_{\{0,N\}}(i)) \biggl(- (A^{\top})^{i}_\alpha \wt{g}^{\alpha}_{\alpha '}(h) 
		A_{j}^{\alpha'} (1+\mathbf{1}_{\{0,N\}}(j))
		(A^{\top})^j_{\alpha''} A_k^{\alpha''}h^k \,+\, {\sqrt{N}} (A^{\top})^{i}_\alpha  \partial_{\alpha'}\widetilde{g}^{\alpha \alpha '}(h)\biggr) \d t
		\\& +\,  \sqrt{{2N}}(1+\mathbf{1}_{\{0,N\}}(i))(A^{\top})^i_\alpha \sqrt{\widetilde{g}}^\alpha_{\alpha'}(h)\d w^{\alpha'} ,
	\end{split}
\end{align}
where  $(w^\alpha)_\alpha$ are independent real-valued Brownian motions, and $\wt{g}^{\alpha \alpha'} = \wt{g}^\alpha_\gamma \delta^{\gamma \alpha'}$. Both approximations of \eqref{e:STFE} display  additional correction terms. To compare them, it is insightful to calculate from \eqref{eqn_zz} and \eqref{E11} that 
\begin{align}\label{eqn_AG}
 {\sqrt{N}} (A^{\top})^{i}_\alpha 	\partial_{\alpha' } \widetilde{g}^{\alpha \alpha '}(h) \,=\, N (A^{\top})^{i}_\alpha    \biggl(	\fint_{I_\alpha} \frac{1}{m(h)}  \biggr)^{-2} 	\fint_{I_\alpha} \frac{m'(h)}{m^2(h)}  \frac{\bar{\vp}_\alpha }{\sqrt{N}},
\end{align}
which we interpret as  a discrete approximation\footnote{This can be seen by an integration by parts in the last integral of \eqref{eqn_AG}.} of 
\begin{align}
	\begin{cases}
	A^\top\bigl( \frac{	N}{6}  m^2(h) \partial_x  \bigl(
		\frac{m'(h)}{m^2(h)}
		\bigr) \bigr) , &  \alpha\pm \notin \{0,N\} , \\A^\top\bigl( \frac{	N}{4}  m^2(h) \partial_x  \bigl(
		\frac{m'(h)}{m^2(h)}
		\bigr)  \pm   \frac{	N^2}{2}  m'(h)  \bigr)  , &  \alpha\pm \in \{0,N\} . \\
	\end{cases}
\end{align}
 Away from the boundary, the above coincides with the correction on a periodic domain, cf.\ \cite[Remark 9.1]{GGKO21}, and arises from the nonlinear interpolation of the mobility function. Near the boundary, however, an additional term emerges that is  similar to the correction in \eqref{Eq_SDE1_new}. This correction would also vanish on a periodic domain, so that we have identified two distinct mechanisms causing a non-trivial geometric correction  of approximations to \eqref{e:STFE}: Firstly, the spatially non-stationary nature of Neumann boundary conditions imposed in both approximations and, secondly, a non-flat approximation of the mobility as used  in the Gr\"un--Rumpf scheme \eqref{Eq_SDE2_new}.

 That different spatial discretizations of an SPDE can lead to the emergence of additional terms is consistent with observations from \cite{HMW}, where these are interpreted as a spatial version of the It\^o-to-Stratonovich correction in the context of stochastic Burgers equations. Regarding \eqref{e:STFE} they  are precisely the geometric coefficients of the form $\partial_{\alpha'}{\tilde{g}}^{\alpha'\alpha}$ and  $\partial_{\alpha'}{\check{g}}^{\alpha'\alpha}$, respectively,    in \eqref{Eq191} and  ensure, by construction, covariance of the SDEs under changes of variables and reversibility with respect to the underlying Gibbs measure. We conjecture that the counterterms arising in the study of \eqref{e:STFE} using regularity structures in  \cite{gvalani2023stochastic} serve the same purpose.

	\subsection{Challenges in identifying the generator domain}\label{SS:generator_is_diff} Next, let us explain why we limit our analysis to Dirichlet form techniques and do not work, for example, with the generator $\L$ defined through the integration by parts formula \eqref{e:generator}. 
	Since difficulties already arise in the non-degenerate setting from Subsection \ref{SS_non-deg}, we continue in that setting and rewrite \eqref{e:Dirichlet_form-2} as
	\begin{equation}
	\label{e:form_alt}
		\EE (\eta,\zeta ) \,=\, \int_{\MM} \diff \zeta_\cdot \grad \eta    \, \d  \mu.
	\end{equation}
	To calculate $\L \eta $, we would need to remove the differential from $\zeta$. While $\grad \eta $ is determined by the  metric structure, the latter relies only on the measure $\mu$. More precisely, we may characterize the $\mu$-divergence of a vector field $X$  via the formula 
	\begin{equation}\label{e:div}
		\int_{\MM} \diff \zeta_\cdot X \, \d  \mu \,=\, - \int_{\MM} \zeta \div_{\mu } X \,\d\mu,
	\end{equation}
	for   test functions $\zeta $. Combining this with \eqref{e:form_alt}, we find that \eqref{e:generator} yields in our situation
	\begin{equation}\label{e:gen_alt}
		\L \eta \,=\, \div_{\mu}  \grad \eta ,
	\end{equation}
so that at least the form of the generator is clear.  To find a concrete expression  for the $\mu$-divergence, we assume for simplicity that $X =\dot{h}$ is independent of $h$. Then,
	\begin{align}\label{e:1}
			\int_{\MM} \diff \zeta_\cdot \dot{h} \, \d  \mu\,=\, \lim_{t\to  0} 
			\int_{\MM} \frac{\zeta ( h + t \dot {h}) - \zeta(h)}{t} \, \d  \mu(h).
	\end{align}
	In light of the Cameron--Martin theorem, we have moreover
	\[
	\int_{\MM}  \zeta ( h + t \dot {h}) \,\d \mu (h) \,=\, \int_{\MM}  \zeta ( h )  \exp \biggl(  t \int_{\T} \partial_x h \partial_x\dot{h}  -\frac{ t^2 }{2} \int_{\T} |\px  \dot{h}|^2\biggr)\,\d \mu(h),
	\]
	so that \eqref{e:1} can be evaluated as
		\begin{align}&\lim_{t\to  0} 
		\int_{\MM}\zeta(h) \frac{1}{t} \biggl(
		\exp \biggl(  t \int_{\T} \partial_x h \partial_x\dot{h}   -\frac{ t^2 }{2} \int_{\T} |\px  \dot{h}|^2\biggr)
		\,-\,1
		\biggr) \, \d  \mu(h)
		\\& \quad =\, 
		\int_{\MM}\zeta(h)\biggl(   \int_{\T} \partial_x h \partial_x\dot{h} \biggr)\, \d  \mu(h),
	\end{align}
	which identifies 
	\begin{equation}\label{e:mu_div}
		\div_{\mu }\dot{h} \,=\, -  \int_{\T} \partial_x h \partial_x\dot{h}  .
	\end{equation}
	By the Paley--Wiener isomorphism, the  above makes sense as an element of $L^2(\mu)$ if and only if $\dot{h} \in H^1(\T)$. This regularity requirement on $\dot{h}$ is not expected to improve if we allow for $X$ to depend on $h$. This is exactly what limits us in identifying elements in the domain of the generator \eqref{e:gen_alt}: Starting for example from a cylinder function \eqref{e:cyl_fct}, the spatial regularity of
	\begin{align}
		\grad \eta \,=\, -\sum_{i=1}^k\partial_i \Phi \biggl(\int_{\T}  h p^1,\dots , \int_\T  h p^k\biggr)\partial_x ( m(h) \partial_x p^i)
	\end{align} is too low in order to evaluate its $\mu$-divergence. This is not  an artifact of choosing $\eta$ as a cylinder function, but applies  for example also to  $\eta(h) = \int_{\T} F(h) $ for which we have $\grad \eta = - \partial_x (m(h) \partial_x (F'(h)))$, so that the regularity  is again  limited by  $\px h $. 

	We also observe that, due to the singular nature of  \eqref{e:STFE}, this is expected by comparing to the theory of energy solutions \cite{GubPer_generator}: Elements of the generator domain are typically rough, since this allows for  cancellations with the ill-defined nonlinearities in the equation.

	\subsection{Mathematical works on stochastic thin-film equations with spatially colored noise}
	We close this introduction by reviewing so far unmentioned mathematical works on the analysis  of \eqref{e:STFE}  with a spatially colored noise $\xi$. This line of research was initiated in \cite{fischer_gruen_2018}, where the existence of probabilistically and analytically weak solutions to the equation was shown. In that work also the repulsive forces between solid and fluid molecules were included in the model, leading  to positivity preserving dynamics also in the continuum. It focused moreover on the case $m(h) = h^2$, in which the noise term in \eqref{e:STFE} becomes linear. Results in the nonlinear noise case were first obtained in \cite{dareiotis2021nonnegative} and in \cite{GG_LK} a finite speed of propagation was shown for sufficiently strong It\^o-to-Stratonovich type corrections, both without said positivity preserving mechanism. In dimension $d=2$, again for $m(h)=h^2$, weak solutions were shown to exist independently in \cite{metzger2022existence} and \cite{Sauerbrey_2021}. By further leveraging the repulsive forces between liquid and solid, global in time well-posedness in dimension $d=1$ was  obtained in \cite{Agresti_Sauerbrey}. 
\section{An Albeverio--R\"ockner type setup}
\subsection{Abstract setup and result}
In this section, we provide a version of the framework from \cite{AR_90_closability,AR_91} for Dirichlet forms of diffusion type tailored towards stochastic thin-film equations. We work under the following set of assumptions.
\begin{assumption}
	\label{Ass_abstract}
	We assume that spaces $ \mathcal{H}_{\pm 1}$ and $\mathcal X$ and mappings $V$, $\rho$, $\sigma$ and $g$ are given as follows.
\begin{enumerate}[(1)]
	\item\label{A1} Let $\mathcal{H}_{\pm 1}$ be separable Hilbert spaces and $\mathcal{X}$ a topological vector space   such that  
\begin{align}
	\mathcal{H}_1 \hookrightarrow \mathcal{X} \hookrightarrow\mathcal{H}_{-1}.
\end{align}
The former embedding is assumed to be such that the centered Gaussian measure $\mu$ with Cameron--Martin space $\mathcal{H}_1$ is a Radon measure on $\mathcal{X}$. We also assume that $\mathcal{H}_1$ is dense in $\mathcal{H}_{-1}$. 
\item \label{A2} Let $V\colon \mathcal X \to [C,\infty]$, with $C>-\infty$, be convex such that \begin{align}
	\d \wt{\mu} = \exp(-V)\d \mu
\end{align}  is   a probability measure on $\mathcal{X}$. We moreover impose that $\{V<\infty\}$ is open  and that $V$ is continuous on the latter set, in the topology of $\mathcal{X}$. 
\item\label{A3} We assume to be given another $\mathcal{X}$-continuous function $\rho \colon \{V<\infty\}\to (0,\infty)$ together with  $\sigma \in L^1(\wt{\mu})$ and a mapping $g^{\cdot}(\cdot,\cdot) \colon \{V<\infty\} \times \mathcal{H}_{-1}^*\times \mathcal{H}_{-1}^*\to \mathbb{R}$ such that for each $h\in \{V<\infty\}$, $g^h$ defines an inner product on $\mathcal{H}_{-1}^*$, depending measurably on $h$, with 
	\begin{align} \rho (h) |p|_{\mathcal{H}_{-1}^*}^2\le 
		g^h( p,p) \le \sigma (h) |p|_{\mathcal{H}_{-1}^*}^2.
	\end{align} 
\end{enumerate}
\end{assumption}
The above ensures that we may define the pre-Dirichlet form 
\begin{equation}\label{eqn_DF_abstract}
\EE(\eta,\eta ) \,=\, \int_{\mathcal X} g^h(\diff|_h \eta ,\diff|_h \eta  )\,\d\wt{\mu}(h),
\end{equation}
for 
\begin{align}\label{eqn_eta}
	 \eta\, = \,
	\Phi\bigl( (\cdot, p^1)_{\mathcal{H}_{-1}\times \mathcal{H}_{-1}^*} , \dots,(\cdot, p^k)_{\mathcal{H}_{-1}\times \mathcal{H}_{-1}^*}  \bigr) ,\qquad 
	p^i \in \mathcal{H}_{-1}^*,\,  \Phi \in C^1_b(\R^k),
\end{align}
where we define canonically 
\begin{align}
	\diff \eta\, =\, \sum_{i=1}^k \partial_i \Phi \bigl( (\cdot, p^1)_{\mathcal{H}_{-1}\times \mathcal{H}_{-1}^*} , \dots,(\cdot, p^k)_{\mathcal{H}_{-1}\times \mathcal{H}_{-1}^*}  \bigr) p^i. 
\end{align}
Since $g^h$ is bilinear, we may expand
\begin{align}
	\EE(\eta, \eta) \,=\, \int_\mathcal{X} g^{h}(p^i,p^j) \partial_i\Phi(\dots) \partial_j \Phi(\dots) \,\d\wt{\mu},
\end{align}
where the dots stand for the argument of $\Phi$. The displayed expression is finite by the boundedness of $\partial_i \Phi$ together with Assumption  \ref{Ass_abstract}~\eqref{A3}. We also observe that $\eta$ as in \eqref{eqn_eta} is a well-defined element of $L^2(\widetilde{\mu})$. Indeed, since $p^i \in \mathcal{H}_{-1}^*\subset \mathcal{X}^*$, each $(\cdot, p^i)_{\mathcal{H}_{-1}\times \mathcal{H}_{-1}^*} $ is a well-defined function on $\mathcal{X}$.  Composing these functions with the bounded function $\Phi$ gives an element of $L^2(\widetilde{\mu})$. We denote the collection of all $\eta$ as in \eqref{eqn_eta} in the following by ${\D}_0$ and refer for the  terminology from Dirichlet form theory appearing below to Appendix \ref{app:DFs}.

\begin{theorem}\label{thm_abstract}
	Under Assumption \ref{Ass_abstract}, $(\EE,\D_0)$ as  above is densely defined and closable on $L^2(\widetilde{\mu})$.  Its closure $(\EE,\D(\EE) )$ is a local, conservative Dirichlet form, which is quasi-regular in the topology of $\mathcal{H}_{-1}$, and admits in particular an associated $\mathcal{H}_{-1}$-valued diffusion. 
\end{theorem}

\subsection{Comments on the proof}
The most challenging part of Theorem \ref{thm_abstract} is to prove that $(\EE,\D_0)$ is closable. For this, we rely on the abstract setup from \cite{AR_90_closability,AR_91}, where closability is shown instead by lifting \emph{Hamza's condition}, which characterizes closable forms in one dimension, to the infinite dimensional setting. The latter states that the quadratic form 
\begin{align}\label{eqn18263n}
	L^2(\R,\rho \d r) \supset \D_0 \to \R, \, \eta \mapsto \int_\R |\partial_r  \eta(r) |^2 \rho(r)  , 
\end{align}
defined on those functions in $C_c^\infty(\R)$ for which the above is finite,
admits a closed extension if and only if for $\d r$-a.e.\ $r\in \R$ with $\rho(r)>0$ we have $\int_{r-\epsilon}^{r+\epsilon} 1/\rho\, \d s <\infty$ for some $\epsilon>0$. 
The version \cite[Theorem 3.2]{AR_90_closability} of this for directional derivatives in infinite  dimensional state spaces and its generalization \cite[Theorem 3.1]{AR_91} regarding curved settings such as \eqref{eqn_DF_abstract} provide rather mild sufficient criteria for a given pre-Dirichlet form   to be closable. We argue that they are always satisfied in the quasi-Gaussian setting from above. In particular, the continuity properties of $V$ and $\rho$ as well as the  lower bound on $g$ from Assumption \ref{Ass_abstract} play a central role there.

The assertions that the closure $(\EE,\D(\EE))$ is  local and conservative are standard. Their proofs are included below for the reader's convenience. Quasi-regularity, a vital ingredient for the existence of an associated Markov process, follows from the main result of \cite{rock_schmu_tightness}. For this, the upper bound on $g$ from Assumption \ref{Ass_abstract} is essential. 

\subsection{Closability }\label{SS:closable}
Here, we show that $(\EE,\D_0)$ is densely defined and admits a closed extension. Let us already observe the former: Because $\mu$ is supported on $\HH_{-1}$, the embedding $\HH_1\subset\HH_{-1}$ is Hilbert--Schmidt and therefore compact, see \cite[Theorem 2.3.1]{Boga_Gaussian}. It admits in particular a singular value decomposition, which in this case constitutes an orthonormal basis $(\varphi_k)_{k\in \N}$ of $\HH_{-1}$, such that $(\sigma_k \varphi_k)_{k\in \N}$ is orthonormal in $\HH_{1}$ for $\sigma_k\searrow 0$. The dual basis defined by 
\begin{equation}
	\label{eqn_SVD}
	(h,p^k)_{\HH_{-1}\times \HH_{-1}^*} \,=\,
	(h,\varphi_k)_{\HH_{-1}}
\end{equation}
provides with $(p^k/\sigma_k)_{k\in \N}$ also the dual basis to $(\sigma_k\varphi_k)_{k\in \N}$ in $\HH_1$. Since polynomials along $(\sigma_k\varphi_k)_{k\in \N}$ are dense in $L^2(\mu)$, cf.\ \cite[Theorem 1.1.1]{Nualart_related}, and  we may approximate them with elements from $\D_0$ using dominated convergence, also  $\D_0$  is dense in $L^2(\mu)$. Finally, let $\zeta\in L^2(\wt{\mu})$, which we may assume to be bounded by the same reason. Then there exist $\eta_k \in \D_0$ which converge to $\zeta$ in $L^2(\mu)$ and therefore also in $L^2(\wt{\mu})$, since $\exp(-V(h))$ is bounded from above.

Regarding the latter assertion, as in \cite{AR_91}, we argue first that 
the operator 
\begin{align}\label{eqn37n}
	\diff \colon \D_0\subset L^2(\widetilde{\mu}) \to L^2(\rho \d \wt{\mu} ; \mathcal{H}_{-1}^*)
\end{align}
is closable.
\begin{lemma}\label{lemm_clos_gradn}
	The operator \eqref{eqn37n} is closable. 
\end{lemma}
Before giving the proof of Lemma \ref{lemm_clos_gradn}, let us detail how it yields the closability of $(\EE,\D_0)$: 
Denoting the smallest closed extension of \eqref{eqn37n} by 
$(\diff^\rho,\D(\diff^\rho))$ we may define
\begin{align}\label{eqn40n}
	\EE^{\rho} ( \eta,\eta )\,=\, \int_{\mathcal{X}} g^h(\diff|^\rho_h \eta ,\diff|^\rho_h \eta ) \,\d\wt{\mu}(h),
\end{align}
defined on $\D(\EE^\rho)$ consisting of these   $\eta \in \D(\diff^\rho) $ for which \eqref{eqn40n} is finite. To convince ourselves that this defines a closed bilinear form,  let $\eta_k  $ be an $\EE^{\rho} $-Cauchy sequence converging to $F$ in  $L^2(\wt{\mu})$. Using that
\begin{equation}
	\EE^{\rho} ( \eta_k-\eta_l, \eta_k-\eta_l )\,\gtrsim\, \|\diff^{\rho} (\eta_k -\eta_l) \|_{ L^2(\rho \d\wt{\mu};\mathcal{H}_{-1}^* )}^2,
\end{equation}
by Assumption \ref{Ass_abstract}~\eqref{A3}, we deduce that $\diff^\rho \eta_k$ converges in $L^2(\rho\d \wt{\mu};\mathcal{H}_{-1}^*) $ to $\diff^\rho \eta $ by closedness. Up to passing to a subsequence, the latter convergence takes  place $\rho\widetilde{\mu}$-almost everywhere and, as $\rho >0$, $\wt{{\mu}}$-a.e., this entails also $\wt{\mu}$-a.e.\ convergence in $\mathcal{H}_{-1}^*$. Fatou's lemma implies then
\begin{equation}
	\int_{\mathcal{X}}g^h \bigl( \diff|^{\rho}_h (\eta -\eta_l ),
	\diff|^{\rho}_h (\eta -\eta_l )
	\bigr) \d\wt{\mu}(h)
	\,\le\,\liminf_{k\to\infty} \EE^{\rho} ( \eta_k-\eta_l, \eta_k-\eta_l ),
\end{equation}
and the latter can be diminished by taking $l$ large. In particular, we find that $\eta \in \D(	\EE^{\rho} ) $ and that $\EE^{\rho} (  \eta_l  - \eta,\eta_l -\eta) \to 0$, as $l\to\infty$. Thus, we obtain that $(\EE^{\rho},\D(	\EE^{\rho} ))$ is closed and since it clearly extends $(\EE ,\D_0)$, the latter is closable.

It remains to prove the lemma, for which we keep following \cite{AR_90_closability,AR_91}.

\begin{proof}[Proof of Lemma \ref{lemm_clos_gradn}]
	The first observation is that \eqref{eqn37n} is a closable operator iff the bilinear form
	\[
	\wt{\EE} (\eta,\eta) = \int_{\mathcal X}\|
	\diff \eta \|_{\mathcal{H}_{-1}^*}^2 \,{\rho}
	\d\wt{\mu},
	\] 
	initially defined on $\D_0$, 
	admits a closed extension on $L^2(\wt{\mu})$. Secondly, by the comments above \eqref{eqn_SVD} there exists an orthonormal basis $( \vp_k)_{k\in \N}$  of $\mathcal{H}_{-1}$ such that $(\sigma_k\vp_k)_{k\in \N}$ is orthonormal in  $\mathcal{H}_1$. Using the latter, we may measure
	\begin{equation}
	\|p \|_{\mathcal{H}_{-1}^*}^2  \,=\,\sum_{k\in \N} \bigl( p_\cdot \vp_k\bigr)^2.
	\end{equation}
	Applying the above to $p=\diff \eta$, we may represent $\diff \eta_\cdot \vp_k $ as the directional derivative
	\[
	\diff|_{h} \eta_\cdot \vp_k \,=\, \lim_{t\to 0} \frac{\eta(h+t\vp_k) -\eta(h)}{t}.
	\]
	Moreover, since the sum of closable forms is closable again, see \cite[Theorem 1.2]{AR_90_closability}, it suffices to show in the following  that
	\begin{equation}\label{eqn2}
		\wt{\EE}_k(\eta,\eta) \,=\, \int_{\X} \bigl(	\diff\eta_\cdot \vp_k\bigr)^2 \,{ \rho}\d \widetilde{\mu}    ,
		\qquad \eta\in \D_0,
	\end{equation}
	admits a closed extension on $L^2(\wt{\mu})$ for any $k\in \N$.
	
	As a first step, let us note that $
	\wt{\EE}_k$ respects $\tilde\mu$-equivalence classes. 
	The latter is relevant in  our case, since $\widetilde{\mu}$ does not necessarily have full support on $\X$ and it could in principle happen that two elements of $\D_0$, which don't coincide as functions on $\X$, lie in the same  $L^2(\widetilde{\mu})$-equivalence class. Then one needs to make sure that $\widetilde{\EE}_k$ assigns anyways a unique value. Suppose that $\eta = \zeta $, $\widetilde{\mu}$-a.e., for $\eta,\zeta \in \D_0$. Since both functions are continuous in the $\mathcal{X}$-topology, they coincide on $\{V<\infty\}\cap\supp(\mu)$. But since an infinitesimal shift in the direction  $ \vp_k$ respects this class by Assumption \ref{Ass_abstract}~\eqref{A2} and the Cameron--Martin theorem, we find that $\diff |_{h } \eta_\cdot   \vp_k = \diff |_{h } \zeta_{\cdot } \vp_k  $ for  $h\in\{V<\infty\}\cap\supp(\mu)$, and thus $\wt{\mu}$-a.e.\ $h$.
	
	Next, to prove that the well-defined $(\wt{\EE}_k,\D_0)$ is closable,  we check the following {Hamza-type condition} from \cite[Theorem 3.2]{AR_90_closability} and \cite[Corollary 2.5]{AR_91}, see  the comments below \eqref{eqn18263n}: 
	
	\emph{Let $\X_k $ be a complementing subspace  of $\spa \{ \vp_k  \}$ in $\X$, and let $\nu_k$ be a Markovian kernel such that
		\begin{equation}
			\int_{\X} \eta(h) \,\d\widetilde{\mu}(h) \,=\, \int_{\X_k} \int_{\R} \eta( h +r \vp_k )  \, \nu_k(h,\d r)\,\d{\pi_k}_{\#}\widetilde{\mu}(h),
		\end{equation}
		for all  measurable $\eta \colon\X \to [0,\infty) $, where $\pi_k\colon \X\to\X_k$ is the projection with $ \pi_k ( \vp_k  )=0$. Assume that for ${\pi_k}_{\#}\tilde{\mu}$-a.e.\ $h\in \X_k$, we have $\nu_k(h ,\d r)= \nu_k( h,r) \d r$ for a measurable function $\nu_k(h,\cdot)$ satisfying for $\d r$-a.e.\ $r\in \R$ with $\nu_k(h,r)>0$ that $\int_{r-\epsilon}^{r+\epsilon} 1/\nu_k(h,s) \d s <\infty$ for some $\epsilon>0$. If additionally 
		for ${\pi_k}_\#\widetilde{\mu} $-a.e.\ $h\in \X_k$ and $\d r$-a.e.\ $r\in \R$ with $\nu_k(h,r)>0$ and ${\rho}(h+r  \vp_k)>0$,  
		\begin{equation}\label{eqn38n}\mathrm{ess}\inf\bigl\{{\rho}(h +s  \vp_k  )\,\big|\, r -\epsilon \le s \le r +\epsilon\bigr\}\,>\,0,
		\end{equation}
		for some $\epsilon>0$, then $ (\wt{\EE}_k,\D_0)$ is closable on $L^2(\widetilde{\mu})$.} 
	
	To choose $\X_k$ appropriately, we leverage the property that $(h,\vp_k)_{\HH_1} = \sigma_k^{-2}(h,\vp_k)_{\HH_{-1}}$ for all $h\in \HH_1$.  Thereby, the latter functional is continuous on $\X$ and we can choose 
	\begin{equation}
		\X_k \,=\,  \bigl\{ h\in \X \big| (h,\vp_k)_{\mathcal{H}_1} = 0 \bigr\},
	\end{equation}
	in which case we have for the reference measure $\mu$ the Gaussian disintegration 
	\begin{align}
		\int_{\X} \eta(h) \,\d\mu\,&=\, \frac{1}{\sqrt{2\pi\sigma_k^2}}\int_{\X_k} \int_{\R} \eta( h \,+\,r  \vp_k ) \exp(-r^2/(2\sigma_k^2))  \,\d r \,\d{\pi_k}_{\#} {\mu}(h).
	\end{align}
	Applying the above with $\eta(h)$ replaced by $\eta(h) \exp(-V(h))$, we find that 
	\begin{align}\begin{split}\label{eqn39n}
			\int_{\X} \eta (h) \,\d \widetilde{\mu}(h) \,&=\,
			\int_{\X} \eta(h) \exp(-V(h)) \,\d\mu(h)\\&=\,\frac{1}{ \sqrt{2\pi\sigma_k^2}}\int_{\X_k} \int_{\R}\eta  ( h \,+\,r  \vp_k )
			\exp(-V(h+r\vp_k))  \exp(-r^2/(2\sigma_k^2))\,\d r\,\d{\pi_k}_{\#} {\mu}(h).
		\end{split}
	\end{align}
	Taking $\eta  = \mathbf{1}_{\{\pi_k h \in \Lambda_k \}}$ for a measurable $\Lambda_k \subset \X_k$, we deduce that 
	\begin{align}
		{\pi_k}_{\#} \widetilde{\mu}(\Lambda_k)
		\,&=\, \frac{1}{ \sqrt{2\pi\sigma_k^2}}\int_{\Lambda_k} \int_{\R}	\exp(-V(h+r\vp_k))   \exp(-r^2/(2\sigma_k^2))\,\d r\,\d{\pi_k}_{\#} {\mu}(h)\\&
		=\, \int_{\Lambda_k} \zeta_k(h)\,\d{\pi_k}_{\#} {\mu}(h),
	\end{align}
	in other words $ \d{\pi_k}_{\#} \widetilde{\mu} \,=\, \zeta_k \d{\pi_k}_{\#} {\mu} $,
	where 
	\begin{equation}
		\zeta_k (h) \,=\, \frac{1}{ \sqrt{2\pi\sigma_k^2}} \int_{\R}\exp(-V(h+r\vp_k)) \exp(-r^2/(2\sigma_k^2))\,\d r.
	\end{equation}
	Taking now $\eta$ in \eqref{eqn39n} arbitrary, we observe that
	\begin{align}
		\int_{\X} \eta(h) \,\d\widetilde{\mu}  \,=\, \frac{1}{ \sqrt{2\pi\sigma_k^2}}\int_{\X_k} \int_{\R}\eta( h \,+\,r  \vp_k ) \frac{\exp(-V(h+r\vp_k)) \exp(-r^2/(2\sigma_k^2)) }{\zeta_k (h)} \,\d r\,\d{\pi_k}_{\#}\widetilde{\mu}.
	\end{align}
	In consequence, we have identified
	\begin{align}		
		\nu_k(h ,r) \,=\,  \frac{\exp(-V(h+r\vp_k)) \exp(-r^2/(2\sigma_k^2)) }{\sqrt{2\pi\sigma_k^2}\,\zeta_k (h)}\mathbf{1}_{\{\zeta_k(h)>0\}},
	\end{align}
	for $\d{\pi_k}_{\#} \widetilde{\mu}$-a.e.\ $h\in \X_k$. 
	
	Since by Assumption \ref{Ass_abstract}~\eqref{A2}, for any $h \in \X_k$, the set of all $r$ such that $\exp(-V(h+r\vp_k))>0$ is a convex set, and since $V$ is continuous on $\{V<\infty\}$, we see that $\nu_k (h,\cdot)$ satisfies Hamza's condition, i.e., the first part of the above condition. To check also the second part, we assume that  $\nu_k(h,r)>0$, and therefore  $h+r\vp_k \in \{V<\infty\} $. As the latter set is open by  Assumption \ref{Ass_abstract}~\eqref{A2}, this also holds if we replace $r$ by $s\in [r-\epsilon,r+\epsilon]$. Since $\rho$ is positive and continuous on  $\{V<\infty\}$ by Assumption \ref{Ass_abstract}~\eqref{A3}, the claimed \eqref{eqn38n} follows.
\end{proof}
\subsection{Properties of $(\EE,\D(\EE))$}
 We denote the closure of  $(\EE,\D_0)$ by $(\EE,\D(\EE))$.  
Here, we collect additional properties of $(\EE,\D(\EE))$ and  start with the Markovian property, which  follows by an approximation argument.

\begin{proposition}\label{lemm_is_DFn}The closed bilinear form $(\EE ,\D(\EE))$ on $L^2(\widetilde{\mu})$ is a Dirichlet form.
\end{proposition}
\begin{proof} Firstly, we convince ourselves that for all $\eta \in \D_0$, it holds that $\zeta  = (\eta\vee 0) \wedge 1 \in \D(\EE)$ with $\EE(\zeta,\zeta )\le \EE(\eta,\eta)$. To this end, let $\phi_\epsilon$ be a smooth increasing function with Lipschitz constant $1$, such that $\phi_\epsilon(r) = r$ for $r\in [0,1]$, which is constant outside of $ [-\epsilon,1+\epsilon]$. Then $\zeta_\epsilon=\phi_\epsilon(\eta) \in \D_0$ and we have that $\zeta_\epsilon \to \zeta$ in $L^2(\wt{\mu})$ by the dominated convergence theorem.
	Meanwhile,  we have
	$
	\diff \zeta_\epsilon = \phi_\epsilon'(\eta) \diff \eta
	$, 
	so that 
	\begin{align}\EE(\zeta _\epsilon -\zeta_\delta,\zeta_\epsilon-\zeta_\delta) \,\le \,  \int_{\X} \mathbf{1}_{(-\epsilon,0) \cup (1,1+\epsilon) } (\eta(h)) g^h (\diff|_h \eta , \diff|_h \eta) \,\d\wt{\mu}(h),
	\end{align}
	for $0 <\delta <\epsilon $. Thereby, $\zeta_\epsilon$ is an $\EE$-Cauchy sequence, which implies that $\zeta_\epsilon\to \zeta$ in $\D(\EE)$. In particular, we have  that $\EE(\zeta,\zeta ) =\lim_{\epsilon\searrow0} \EE(\zeta_\epsilon,\zeta_\epsilon )\le   \EE(\eta,\eta)$ as desired. 
	
	Now take any $\eta \in \D(\EE)$. Then there exists $\D_0 \ni \eta_k \to \eta $ in $\D(\EE)$, as $k\to\infty$. Defining moreover $\zeta$ and $\zeta_k$ accordingly, we find that $\zeta_k \to \zeta $ in $L^2(\wt{\mu})$. By the boundedness of $\zeta_k$ in $\D(\EE)$, the latter convergence holds also weakly in said space. But then, using lower semicontinuity of the norm with respect to weak convergence, we deduce  
	\begin{align}
		\EE(\zeta,\zeta )\,\le\, \liminf_{k\to \infty}  \EE (\zeta_k,\zeta_k )\,\le\, \lim_{k\to \infty}  \EE (\eta_k,\eta_k )\,=\, \EE(\eta,\eta).
	\end{align}
	We conclude that $(\EE,\D(\EE))$ satisfies the Markovian property, i.e., it is a Dirichlet form. 
\end{proof}
To proceed, we observe some consequences of the proof of Lemma \ref{lemm_clos_gradn} in the following remark. 
\begin{remark}
	The approach from \cite{AR_90_closability,AR_91} to obtain that $\diff\colon \D_0 \subset L^2(\wt{\mu}) \to L^2(\rho\d\wt{\mu};\HH_{-1}^*) $ is closable reveals something about the resulting closure. Using the notation from  the proof of Lemma \ref{lemm_clos_gradn}, we have 
	\begin{equation}\label{eqn574n}
		\D(\diff^\rho) \,\subset \,
		\biggl\{
		\eta \in \bigcap_{k\in \N} \D(	\widetilde{\EE}_k)\,\Big|\,
		\sum_{k\in \N } \widetilde{\EE}_k(\eta,\eta )\,<\,\infty
		\biggr\},
	\end{equation}
	and with $p^k$ from \eqref{eqn_SVD}
	\begin{equation}\label{eqn575n}
		\diff^{\rho}\eta \,=\, \sum_{k\in \N}
		(\diff^\rho\eta_\cdot  \vp_k)  p^k \,\in \, L^2 (\rho\d \wt{\mu};\mathcal{H}_{-1}^*),
	\end{equation}
	where the individual summands of the latter series are well-defined for $\eta\in \D(	\widetilde{\EE}_k)$, respectively. Moreover,  by its construction \cite[Corollary 2.5]{AR_91}, the latter domain $\D(	\widetilde{\EE}_k)$ is contained in 
	\begin{align}
		\biggl\{
		\eta \in L^2(\widetilde{\mu}) \,\biggl|\,& \text{for $\d{\pi_k}_{\#} \widetilde{\mu}  $-a.e.\ $h\in \X_k$, $\eta_h (r) =  \eta(h+r \vp_k)$ has an absolutely continuous $\d r$-version}
		\\& \text{on $ \biggl\{r\in \R : \, \int_{r-\epsilon}^{r+\epsilon} 1/\nu_k{(h,s)}\,\d s<\infty \,\text{for some $\epsilon>0$} \biggr\}$ with derivative $\frac{\d \eta_h}{ \d r}$ }\\&
		\text{and } \int_{\X} \Bigl| \frac{\d \eta_{\pi_k (h)}} {\d r}  \bigl((h,\vp_k)_{\HH_{-1}}\bigr) \Bigr|^2 \,{\rho}(h)\,\d  \widetilde{\mu}(h)\,<\,\infty
		\biggr\},\label{eqn83n}
	\end{align}
	and we have $\diff|_h^\rho\eta_{\cdot}  \vp_k\,=\, \frac{\d \eta_{\pi_k (h)}}{ \d r} \bigl((h,\vp_k)_{\HH_{-1}}\bigr) $, for $\eta \in \D(	\widetilde{\EE}_k)$. As a consequence,   the {product rule} 
	\begin{equation}
		\diff^\rho(\eta\zeta)_\cdot \vp_k\,=\, \eta \diff^\rho \zeta_{\cdot}  \vp_k \,+\, 
		\zeta \diff^\rho  \eta_{\cdot}  \vp_k,\qquad \eta,\zeta \in L^\infty(\widetilde{\mu})\cap \D(	\widetilde{\EE}_k),
	\end{equation}
	holds. Moreover, the {local property} 
	\[
	\diff^\rho\eta_{\cdot}  \vp_k \,=\, 0 \qquad \text{$\widetilde{\mu}$-a.e.\ outside of $\supp (\eta)$,}
	\]
	holds for $\eta \in \D(	\widetilde{\EE}_k)$.
	By \eqref{eqn574n} and \eqref{eqn575n}, these properties transfer to $(\diff^\rho,\D(\diff^\rho))$.
\end{remark}

\begin{proposition}\label{prop:CarreduChampn}The  Dirichlet form $(\EE,\D(\EE))$ is 
	local, conservative  and satisfies
	\begin{equation}\label{eqn44n}
		2 \EE(\eta, \eta \zeta) \,-\, \EE(\eta^2,\zeta )\,=\, 2\int_{\X} \zeta g^h (\diff|_h^\rho \eta, \diff|_h^\rho \eta )\,\d\wt{\mu}(h),
	\end{equation}
	for all $\eta,\zeta \in L^\infty(\wt{\mu})\cap \D(\EE)$.
\end{proposition}
\begin{proof}We keep following \cite[Theorem 3.1]{AR_91} to obtain the stated properties of $(\EE ,\D(\EE))$: That $(\EE,\D(\EE))$ is conservative, i.e., that the associated semigroup keeps $\mathbf{1}_{\X}$ invariant, follows from $\mathbf{1}_{\X}\in \D(\EE)$ and $\EE(\mathbf{1}_{\X}, \mathbf{1}_{\X})=0$. Indeed, this implies that $\mathbf{1}_{\X}$ lies in the kernel of the generator of $\EE$, yielding the desired conservative property. Regarding the other two properties, we recall that \eqref{eqn40n} is an extension of $(\EE,\D(\EE))$. Thereby, the local property of $\diff^\rho$ detailed in the preceding remark yields the local property of $\EE$. The product rule for $\diff^{\rho}$, stated in  the same remark, yields then that 
	\begin{align}2 \EE(\eta, \eta \zeta) \,&=\, 2  \int_{\X} 
	g^h\bigl( \diff|_h^\rho \eta ,
		\eta \diff|_h^\rho \zeta + \zeta \diff|_h^\rho \eta 
		\bigr)  \,\d\wt{ \mu}(h)
		\\
		\EE(\eta^2,\zeta )\,&=\, 
		{2}  \int_{\X} \eta g^h\bigl( \diff|_h^\rho \eta , \diff|_h^\rho \zeta \bigr) \,\d\wt{\mu}(h),
	\end{align}
	for $\eta,\zeta \in L^\infty(\wt{\mu})\cap \D(\EE)$.
	Subtraction  results in \eqref{eqn44n}.
\end{proof}
While not part of the statement of Theorem \ref{thm_abstract}, the  assertion \eqref{eqn44n} is useful in the following section since it implies that  $(\EE,\D(\EE))$ admits 
\begin{equation}\label{SFO}
\Gamma(\eta,\zeta ) (h) \,=\, 2g^h \bigl(\diff|_h^\rho \eta, \diff|_h^\rho \zeta\bigr)
\end{equation}
as square field operator, cf.\ \eqref{eqn45}.  

The remaining collected properties of $(\EE,\D(\EE))$ guarantee the conditions for the existence of an associated diffusion process  from Theorem \ref{thm_abstract}, besides the required  quasi-regularity. The properties \eqref{IB} and \eqref{IC} of a quasi-regular Dirichlet form, cf.\ Definition \ref{def_QR}, are on the one hand immediate by our choice of pre-domain $\D_0$. That there exists an $\EE$-nest of compact subsets of $\HH_{-1}$ follows, on the other hand, from \cite[Proposition 3.1]{rock_schmu_tightness}. This completes the proof of Theorem \ref{thm_abstract}.

	\section{Existence and regularity of the stochastic thin-film diffusion}\label{S:DF_in_inf_dimensinos}
	\subsection{Assumption and result of this section}
	In this section, we are concerned with the Dirichlet form \eqref{e:Dirichlet_form_+}, i.e.,
	\begin{equation}\label{eqn19}
		\EE (\eta,\zeta ) \,=\,\int_{\MM^+} g^h ( \diff|_h \eta ,\diff|_h \zeta ) \, \d  \mu^+(h)\,=\, 
		\int_{\MM^+} \int_{\T} m(h) \partial_x  \diff|_h \eta \partial_x \diff|_h \zeta  \,\d x\,  \d  \mu^+(h)
		.
	\end{equation} 
In order to apply the abstract Theorem  \ref{thm_abstract} to deduce the existence of an associated diffusion process, we impose a condition  tailored towards Assumption \ref{Ass_abstract} on the mobility function $m$, which we allow here to be non-local. 
	\begin{assumption}\label{ass:deg_m}
		Let $m\colon \{h\in C(\overline{\T})| h\ge 0 \} \to \{h\in C(\overline{\T})| h\ge 0 \}$ be measurable such that there exists $n\in [0,\infty)$ with
		\begin{align}\label{eq440}
		\sup_h 	\| m(h)\|_{C(\overline{\T})}\Big/
			(1 \vee 	\| h\|_{C(\overline{\T})})^{n}<\infty.
		\end{align}Moreover, we demand that there exists a continuous and increasing $\rho \colon (0,\infty) \to (0,\infty)$ such that 
		\begin{align}\label{eq440_lower}
			\inf_{x\in \T} m(h)(x) \,\ge \, \rho\Bigl( \inf_{x\in \T}h(x) \Bigr),
		\end{align}
	where we set $\rho(0)=0$.
	\end{assumption}
To work with a centered Gaussian measure, we translate the film height by $-\mathbf{1}_\T$. This gives  the following setup, where we recall the convention that dotted spaces refer to the subspace of mean-free functions or distributions.
	\begin{itemize}
		\item $\HH_1 = \dot{H}^{1}(\T)$, $\HH_{-1} = (\dot{H}^{1}(\T))^*$, $\mathcal{X} = \dot{C}(\overline{\T})$. 
		\item $V(h) = \log(\lambda) + \infty \mathbf{1}_{(-\infty,-1]}(\inf_{x\in \T} h (x))$, where 
		\begin{equation}\label{e:lambda}\lambda = \mu \bigl(
		\bigl\{ h \in \MM \,\big| \, h \,>\,  -\mathbf{1}_{\T} \bigr\}\bigr) \,>\, 0.
	\end{equation}	
\item Further, we define  for sufficiently large $C<\infty$  and $n$ and $\rho$ as in Assumption \ref{ass:deg_m}
\begin{equation}
	\sigma(h) = C
	(1 \vee 	\| 1+h\|_{C(\T)})^{n} ,\qquad 
	\rho(h) =  \rho\Bigl( 1+\inf_{x\in \T}h(x) \Bigr)
	,\qquad
	g^h( p,p) \,=\, \int_\T m(1+h) (\partial_x p)^2 .
\end{equation}
	\end{itemize}
Then  Assumption \ref{Ass_abstract} is satisfied and Theorem \ref{thm_abstract} applies. We point out that we exploited \eqref{ent_rep} to define $V$ in line with Assumption \ref{Ass_abstract}~\eqref{A2}. Then, by translating back by $\mathbf{1}_\T$, the results transfer  to the original state space $\MM^+$, which we equip in the following with the topology of $({H}^{1}(\T))^*$.

We  argue below that the resulting diffusion process has $C_{t,x}^{1/8-, 1/2-}$-regularity when started from equilibrium, see  \eqref{e:Holder} for the definition of this space. This  improves in particular its continuity in $({H}^{1}(\T))^*$ implied by the quasi-regularity in that space.  For this we apply the Lyons--Zheng decomposition, cf.\ Theorem \ref{LZ}, to the pointwise approximations $\psi_\epsilon*h(x)$ for a standard mollifier $(\psi_\epsilon)_{\epsilon>0}$. Together with the equivalence of moments for Gaussian measures and an application of the Kolmogorov--Chentsov theorem, this results in the space-time regularity assertion. The following summarizes the results.

	\begin{theorem}
		\label{thm_existence_deg_case}
	Under Assumption \ref{ass:deg_m}, the pre-Dirichlet form \eqref{eqn19} initially defined on
	\begin{align}\label{D0}
			\biggl\{
			\Phi\biggl(\int_{\T}  h p^1,\dots , \int_{\T}  h p^k \biggr) \,\bigg|\,
			p^i\in {H}^1(\T),\,\Phi \in C_b^1(\R^k)
			\biggr\}
	\end{align}  is closable on $L^2(\mu^+)$. Its closure $(\EE,\D(\EE))$ admits an associated $\MM^+$-valued diffusion process $((h_t)_{t\ge 0}, (P_z)_{z\in \MM^+})$, which satisfies
	\begin{equation}\label{statement_reg}
		h\in C_{t,x}^{1/8-, 1/2-}, \qquad P_{\mu^+}\text{-almost surely}.
	\end{equation}	
\end{theorem}
\subsection{Regularity of the associated diffusion}\label{SS:reg}
	It remains to show the additional regularity of the diffusion process provided  by Theorem \ref{thm_abstract}. As a simple first step, we recall the standard regularity result for the Brownian bridge.
	\begin{lemma}\label{lemma_reg}
		It holds that
		\begin{align}\label{e:Bspp-bound_MM+_cont}
			\int_{\MM^+} \|h\|_{C^\gamma(\T)}^p \,\d \mu^+ \,<\,\infty,
		\end{align}
		for each $\gamma\in (0,1/2)$ and $p\in (0,\infty)$.
	\end{lemma}
	\begin{proof}
		By construction of $\mu^+$, \eqref{e:Bspp-bound_MM+_cont} is a consequence of 
		\begin{align}\label{eqn72345}
			\int_{\MM} \|h\|_{C^\gamma(\T)}^p \,\d \mu \,<\,\infty,
		\end{align}
		which, by the Kolmogorov--Chentsov theorem, follows from 
		\begin{align}
			\biggl(	\int_{\MM} |h(x) -h(y)|^p \,\d\mu \biggr)^{1/p}\,\lesssim_p |x-y|^{1/2},\qquad 
			\biggl(	\int_{\MM} |h(x)|^p \,\d\mu \biggr)^{1/p}\,\lesssim_p 1,
		\end{align}
		for all $x,y\in \T$.
		For the former, e.g.,  we may use that by 
		\begin{equation}\label{Eq223}
			\int_{\MM}
			\langle h, p \rangle^2
			\,\d\mu(h) \,=\, \|p\|_{(\dot{H}^{1}(\T))^*}^{2}
		\end{equation} 
	and equivalence of Gaussian moments 
		\begin{align}\label{eqn36}
			\biggl(	\int_{\MM} |h(x) -h(y)|^p \,\d\mu \biggr)^{1/p} \,\eqsim_p\, \biggl(
			\int_{\MM} \langle h, \delta_x -\delta_y\rangle^2   \,\d\mu\biggr)^{1/2} \,=\, \|\delta_x -\delta_y\|_{(\dot{H}^{1}(\T))^*} \,\lesssim\, |x-y|^{1/2},
		\end{align}
		where the last bound follows from $H^1(\T)\hookrightarrow C^{1/2}(\T)$.
	\end{proof}
	\begin{proposition}\label{prop_reg_est}For the  diffusion process $((h_t)_{t\ge 0}, (P_z)_{z\in \MM^+})$ associated to $(\EE,\D(\EE))$, we have  
		\begin{equation}\label{Eq264}
			E_{\mu^+}^{1/p} \Bigl[
			\|h\|_{C^{\gamma/4,\gamma}((0,T)\times \T)}^p
			\Bigr] \,<\,\infty,
		\end{equation}
		for each $\gamma \in (0,1/2)$, $p\in (0,\infty)$ and $T<\infty$.
	\end{proposition}
	\begin{proof}
		Similarly to the proof of Lemma \ref{lemma_reg}, the assertion follows if we can verify that
		\begin{align}\label{eqn94271}
			E_{\mu^+} ^{1/p}
			|h_{t}(x)  - h_s(y)|^p
			\,\lesssim_{(p,T)}\, |t-s|^{1/8}+|x-y|^{1/2},\qquad 
			E_{\mu^+} ^{1/p}
			|h_{t}(x) |^p
			\,\lesssim_p\, 1,
		\end{align} 
		for all $0\le s\le t \le T$ and $x,y\in \T$. Focusing on the former estimate, we decompose
		\begin{align}\label{eqn94272}
			E_{\mu^+} ^{1/p}
			|h_{t}(x)  - h_s(y)|^p \,\lesssim_p\, E_{\mu^+} ^{1/p}
			|h_{t}(x)  - h_s(x)|^p \,+\, |x-y|^{1/2},
			\end{align}
		using \eqref{eqn36}. 
		
		Letting now $(\psi_\epsilon)_\epsilon$ be  a standard mollifier on $\R$ with $\supp \psi_\epsilon = [-\epsilon,\epsilon]$, we further split the temporal increment into
		\begin{align}&
			E_{\mu^+} ^{1/p}
			|h_{t}(x)  - h_s(x)|^p \\&\quad \le\,
			E_{\mu^+} ^{1/p}
			|h_{t}(x)  - \psi_\epsilon*h_t(x)|^p \,+\,
			E_{\mu^+} ^{1/p}
			|\psi_\epsilon*h_{t}(x)  - \psi_\epsilon*h_s(x)|^p
			\,+\, 
			E_{\mu^+} ^{1/p}
			|\psi_\epsilon*h_{s}(x)  - h_s(x)|^p.
		\end{align}
	If the support of $\psi_\epsilon(x-\cdot)$ extends beyond $\T$, we define the convolution in terms of the Neumann  extension of $h$, characterized by $h(2k+\cdot ) = h(2k-\cdot ) =h$ for $k\in \Z$. This may be expressed through the formulae
	\begin{align}&\label{eqn327}
	\psi_\epsilon*h(x) \,=\,\int_{\R} \psi_\epsilon(x-z) h(z) \,\d z\,=\,  \int_{\T} \overline{\psi_\epsilon}(x,z) h(z)\,\d z ,\\&  \overline{\psi_\epsilon}(x,z) \,=\,\sum_{k\in \Z} \psi_\epsilon(x-(2k +z))+ \psi_\epsilon(x-(2k -z)).
	\label{eqn8324}
	\end{align}
	Since $\overline{\psi_\epsilon}(x,z)=0$ unless $|x-z|\le \epsilon$, we find similarly to \eqref{eqn36} that
	\begin{align*}&
		E_{\mu^+} ^{1/p}
		|h_{t}(x)  - \psi_\epsilon*h_t(x)|^p \,=\, \biggl(
		\int_{\MM^+}  
		|h(x) - \psi_\epsilon*h (x)|^p
		\d\mu^+(h)
		\biggr)^{1/p}\\& \quad \lesssim_p \, 
		\biggl(
		\int_{\MM}  
		\langle h, \delta_x - \overline{\psi_\epsilon}(x,\cdot) \rangle^2
		\d\mu(h)
		\biggr)^{1/2}	\,	=\,
		\|\delta_x -  \overline{\psi_\epsilon}(x,\cdot )\|_{(\dot{H}^1(\T))^*}\,\lesssim\, \epsilon^{1/2},
	\end{align*}
	and the same estimate holds for $E_{\mu^+} ^{1/p}
	|\psi_\epsilon*h_{s}(x)  - h_s(x)|^p$.
	
	For the  increment of $\psi_\epsilon*h(x)$, we apply the Lyons--Zheng decomposition from Theorem \ref{LZ} to the additive functional $\eta(h_t) = \psi_\epsilon*h_t(x)$, which implies that
		\[
	\psi_\epsilon*h_t(x)\,-\, 
	\psi_\epsilon*h_0(x) \,=\, \frac{1}{2}M_t \,-\, \frac{1}{2} (\overline{M}_T\, - \, \overline{M}_{T-t} ),
	\]
	for a forward and backward martingale  $M_t$ and $\overline{M}_t$, respectively. Using the representation \eqref{eqn327}, we may evaluate the square field operator \eqref{SFO} as 
	\begin{align}
		\Gamma (\eta,\eta) \,=\,2\int_{\T} m(h)( \partial_z\overline{\psi_\epsilon}(x,\cdot) )^2\,\d z,
	\end{align}
	which appears in the quadratic variations
	\begin{align}&
		\langle M\rangle_t \,=\, 2 \int_0^t\int_{\T} m(h_\tau)( \partial_z\overline{\psi_\epsilon}(x,\cdot) )^2\,\d z \,\d \tau , \qquad 	\langle \overline{M}\rangle_t \,=\,
		2\int_0^t \int_{\T} m(h_{T-\tau})( \partial_z\overline{\psi_\epsilon}(x,\cdot) )^2\,\d z\, \d \tau.
	\end{align}
	This yields for $p\in [2,\infty)$ by convexity and the Burkholder--Davis--Gundy inequality
	\begin{align}&E_{\mu^+} ^{1/p}
		|\psi_\epsilon*h_{t}(x)  - \psi_\epsilon*h_s(x)|^p\,\le\, \frac{1}{2}
		E_{\mu^+}^{1/p}\bigl[
		|M_t - M_s
		|^p\bigr] \,+\, \frac{1}{2} E_{\mu^+}^{1/p}\bigl[
		|\overline{M}_{T-t} -\overline{ M}_{T-s}
		|^p\bigr], \\&\quad
		\lesssim_p \,
		 E_{\mu^+}^{1/p}\biggl[
		 \biggl( \int_s^t 
		 \int_{\T} m(h_\tau)( \partial_z\overline{\psi_\epsilon}(x,\cdot) )^2\,\d z
		 \,\d \tau\biggr)^{p/2}\biggr]\,+\, 
		 E_{\mu^+}^{1/p}\biggl[
		 \biggl(\int_s^t 
		 \int_{\T} m(h_{T-\tau})( \partial_z\overline{\psi_\epsilon}(x,\cdot) )^2\,\d z
		 \,\d \tau\biggr)^{p/2}\biggr]
		 \\&\quad\lesssim \, |t-s|^{1/2}\biggl( \int_{\MM^+}  \biggl(\int_{\T} m(h)( \partial_z\overline{\psi_\epsilon}(x,\cdot) )^2\,\d z\biggr)^{p/2} \,\d\mu^+(h)\biggr)^{1/p},
	\end{align}
	since $h_\tau \sim \mu^+$ for all $\tau\in [0,T]$. We may further bound the right-hand side by \eqref{eq440} and Lemma \ref{lemma_reg}, resulting in the estimate
	\begin{align}
		E_{\mu^+} ^{1/p}
		|\psi_\epsilon*h_{t}(x)  - \psi_\epsilon*h_s(x)|^p\,\lesssim_{p}\,|t-s|^{1/2}\| \overline{\psi_\epsilon}(x,\cdot)\|_{\dot{H}^1(\T)}.
	\end{align}
	Using that at most $2\lceil\epsilon\rceil$ summands in \eqref{eqn8324} are non-zero, we obtain then
	\[\|
	 \overline{\psi_\epsilon}(x,\cdot)\|_{\dot{H}^1(\T)}\,\le\, 2\lceil\epsilon\rceil\| 
	 {\psi_\epsilon}\|_{\dot{H}^1(\R)}\,\lesssim\,\lceil\epsilon\rceil\epsilon^{-3/2}
	\]
	by the triangle inequality.
	We combine the preceding estimates to 
	\begin{align}	E_{\mu^+} ^{1/p}
		|h_{t}(x)  - h_s(x)|^p\,\lesssim_p\,\epsilon^{1/2}\,+\,\lceil\epsilon\rceil\epsilon^{-3/2} |t-s|^{1/2}.
	\end{align}
	Choosing $\epsilon= |t-s|^{1/4}$, so that in particular $\lceil\epsilon\rceil \lesssim_T 1 $, we obtain
	\[
	E_{\mu^+} ^{1/p}
	|h_{t}(x)  - h_s(x)|^p\,\lesssim_{(p,T)}\, |t-s|^{1/8}.
	\]
	Together with \eqref{eqn94272}, we arrive at \eqref{eqn94271} and it remains to use the Kolmogorov--Chentsov theorem to conclude.
	\end{proof}

\section{Uniform tightness of diffusions on discretized state space}\label{S:discrete}
\subsection{Result of this section}
In this section, we obtain analogous results to the ones of Section \ref{S:DF_in_inf_dimensinos} in a spatially discretized setting, uniformly in the discretization parameter. More precisely, we consider the Dirichlet form introduced in \cite{GGKO21} and described in Subsection \ref{SS_tightness_dircscrete}: On the state space
\begin{equation}\label{eqn418_newnew}
	\MM^+_N \,=\, \bigl\{ h\in \MM^+ \,\big|\, \text{$h$ is linear on $[x_i,x_{i+1} ]$}\bigr\},
\end{equation}  
for $N+1$ equidistant  points  $0=x_0< \dots < x_N =1$, we consider 
\begin{equation}\label{e:Dirichlet_form_+_discr_new}
	\widetilde{\EE}_{N} (\eta, \eta )\,=\, 
	\int_{\MM_N^+} \wt{g}^h ( \diff|_h \eta ,\diff|_h \eta ) \, \d  \mu_{N}^+(h),
\end{equation} 
for $\eta \in C^{1}(\MM_N^+)$.  
The  metric and measure appearing were defined in \eqref{E9} and \eqref{defi_mu_N+}, respectively.
The  result of this section is the following.
	\begin{theorem}[Existence and tightness of diffusion processes]\label{thm_existence_deg__fin_case}
	Under Assumption \ref{ass:deg_m}, $(	\widetilde{\EE}_{N} ,C^{1}(\MM_N^+))$ is closable on $L^2(\mu_{N}^+)$ and  its closure $(\widetilde{\EE}_{N},\D(\widetilde{\EE}_{N}) )$ admits an associated $\MM_N^+$-valued  diffusion process $((h_{t} )_{t\ge 0}, (P_{N,z})_{z\in \MM_N^+})$. The laws  $( P_{\mu_{N}^+} \circ {h}^{-1})_{N\ge 2}$ are uniformly in $N$ tight on $
	C_{t,x}^{1/8-, 1/2-}$.
\end{theorem}
As in Section \ref{S:DF_in_inf_dimensinos}, the qualitative properties of the (pre-)Dirichlet form follow from Theorem \ref{thm_abstract}. To show  uniform tightness of the associated diffusion processes, we follow the proof of Proposition \ref{prop_reg_est} leveraging  additionally uniform properties of the involved  metrics and measures. 

\subsection{Estimates on the discrete metric tensor} 

To  show tightness, we need to bound the metric tensor \eqref{E9} of  differential forms  
\begin{align}\label{eqn924}
q_\cdot\dot{h} = \int_{\T}\dot{h} q ,
\end{align}
uniformly in $N$. For this we express the above as 
\begin{align}
	q_\cdot\dot{h} \,=\, 
	\frac{1}{N} \biggl(\frac{\dot{h}(0)\tilde{q}(0)  +\dot{h}(1)\tilde{q}(1) }{2}+\sum_{i=1}^{N-1}
	\dot{h}(x_i) \tilde{q}(x_i)
	\biggr),
\end{align}
where $\wt{q}$ is the convolution of (the Neumann extension of) $q$  with the hat function given by $\hat{\rho}(0) = N$ and $\hat{\rho}(\pm 1/N) = 0$. Based on this, we may use the Legendre formula  to deduce
\begin{align}
	\frac{1}{2}	\tilde{g}^{h} (q,q ) \,&=\, \sup_{\dot {{h}} \in \mathcal{T}_h\MM^+_N} \biggl\{	\frac{1}{N} \biggl(\frac{\dot{h}(0)\tilde{q}(0)  +\dot{h}(1)\tilde{q}(1) }{2}+\sum_{i=1}^{N-1}
	\dot{h}(x_i) \tilde{q}(x_i)
	\biggr)  \,-\, \frac{1}{2} \tilde{g}_{h}( \dot{{h} },\dot{{h} }) \biggr\}
	\\& = \, \sup_{\dot {{h}} \in \mathcal{T}_h\MM^+_N} \biggl\{   \int_{\T} m(h) \px  {p} \partial_x \tilde{q} \,-\, \frac{1}{2}  \int_{\T} m(h) |\px  {p}|^2  \biggr\},
\end{align}
where  ${p}$ is related to $\dot{h}$ via \eqref{E9}. By Young's inequality, the above yields the estimate
\begin{align}\label{E89}
\tilde{g}^{h} (q,q )\,\le\, 
\int_{\T} m(h) |\partial_x \tilde{q}|^2.
\end{align}
This estimate replaces the identity \eqref{E39872} in the following argument.

\subsection{Estimates on the equilibrium measures on discrete state space}
\label{SS:measure}
	The second ingredient for tightness is a collection of uniform bounds on the measures $\mu_N^+$. To this end, we relate them via
	\begin{equation}\label{e-discr:def_mu-1}
	\d \mu_N^+(h+\mathbf{1}_\T) \,=\, \exp( - V_N(h) ) \d\mu_N(h)
	\end{equation}
	for 
	\begin{align}\label{e-discr:def_mu}
		V_N(h) = \log(\lambda_N) + \infty \mathbf{1}_{(-\infty,-1]}\Bigl(\inf_{x\in \T} h (x)\Bigr)
	,\qquad \lambda_N \,=\, 	\mu_N\bigl(\bigl\{ h \in \MM_N \,\big| \, h \,>\,  -\mathbf{1}_{\T} \bigr\}\bigr),
	\end{align} analogous to the continuous setting
	to a centered Gaussian measure $\mu_N \propto \exp( -  E(h)) \d h$ on $\MM_N$. To uniformly control the latter, we connect the above to the continuous measure $\mu$ using the (recentered) nodal interpolation operator 
$Q_N \colon C(\T)\to\MM_N$ defined by
\begin{equation}\label{Eq101}
	Q_N  h \,=\, \sum_{i=0}^N \biggl(
	h(x_i) \,-\, \frac{1}{N}\biggl(\frac{h(0)+h(1)}{2}+ \sum_{k=1}^{N-1} h(x_k)
	\biggr)\biggr)
	\hat{\vp}_i,
\end{equation}
where $\hat{\vp}_i$ is the piecewise linear function satisfying $\hat{\vp}_i(x_j)=\delta_{ij}$. 
	We claim that
	\begin{equation}\label{eqn1001}
		\mu_{N}\,=\,  \mu \circ Q_N^{-1},
	\end{equation}
 which is a consequence of the following lemma.
\begin{lemma}\label{lemma_is_ortho}
	The nodal interpolation  operator $Q_N\colon \dot{H}^1(\T) \to \MM_N$ is an orthoprojection.
\end{lemma}
\begin{proof}$Q_N$ defined through \eqref{Eq101} is clearly a projection onto $\MM_N$ so that it suffices to show that
	\[
	( Q_N h_1 ,  h_2 )_{\dot{H}^1(\T)}  \,=\, ( h_1 , h_2)_{\dot{H}^1(\T)}, \qquad h_1\in \dot{H}^1(\T), \, h_2\in \MM_N.
	\]
	But the above holds true, since the functional $(\cdot, h_2)_{\dot{H}^1(\T)}$ only depends on nodal values and is invariant under spatially constant shifts such as the one in \eqref{Eq101}. 
\end{proof}
Since we anyways need it in Section \ref{S:convergence_explicit}, when we  study convergence of Dirichlet forms,  we also recall an elementary lemma on the approximation of H\"older functions. Its proof  is postponed to  Appendix \ref{sec_holder}.
	\begin{lemma}\label{Lemma_numerical_approx-n}
		Let $h$ be mean-free and $0<s\le \tilde{s}<1$. Then  it holds  $
			\|Q_N h - h\|_{C^s(\T)} \le 4 
			N^{-(\tilde{s}-s)}\| h\|_{C^{\tilde{s}}(\T)}$.
	\end{lemma}
	
	Combining the above, we obtain a uniform version of Lemma \ref{lemma_reg}.

	\begin{lemma}\label{p1}
			It holds that
		\begin{align}\label{e:Bspp-bound_MM+}
		\sup_{N\ge 2}	\int_{\MM_N^+} \|h\|_{C^\gamma(\T)}^p \,\d \mu_N^+ \,<\,\infty,
		\end{align}
		for each $\gamma\in (0,1/2)$ and $p\in (0,\infty)$. 
	\end{lemma}
\begin{proof}
	Using  \eqref{e-discr:def_mu-1}--\eqref{e-discr:def_mu} and \eqref{eqn1001}, we may bound 
	\begin{align}
		\int_{\MM_N^+} \|h\|_{C^\gamma(\T)}^p \,\d \mu_N^+ \,\lesssim\, \frac{1}{\lambda_N}\biggl(
		1 + \int_{\MM} \|Q_Nh\|_{C^\gamma(\T)}^p \,\d \mu
		\biggr).
	\end{align}
	By means of Lemma \ref{Lemma_numerical_approx-n}, the above is, up to a constant,   further estimated  by 
	\begin{equation}\label{upper_bound_1}
	\frac{1}{\lambda_N}\biggl(
	1 + \int_{\MM} \|h\|_{C^\gamma(\T)}^p \,\d \mu
	\biggr),\end{equation}
	and by Lemma \ref{lemma_reg} it remains to control the constants $\lambda_N$. The same ingredients yield however
	\[\lambda_N = \mu \bigl(\bigl\{ h \in \MM \,\big| \, Q_N h \,>\,  -\mathbf{1}_{\T} \bigr\}\bigr)\]
	as well as
	\begin{align}\label{eqn1004}\liminf_{N\to\infty}
		\mathbf{1}_{ \{
			Q_N h >  -\mathbf{1}_{\T}
			\} } \,\ge\, \mathbf{1}_{ \{
			h >  -\mathbf{1}_{\T}
			\} }, \qquad \mu\text{-a.e.}
	\end{align}
	We obtain then by Fatou's lemma 
	\begin{align}
		\liminf_{N\to\infty} 	\lambda_{N}\,=\, 
		\liminf_{N\to\infty} \int_{\MM} 	\mathbf{1}_{ \{
			Q_N h >  -\mathbf{1}_{\T}
			\} } \,\d\mu\,\ge\, \int_{\MM} 	\mathbf{1}_{ \{
			h >  -\mathbf{1}_{\T}
			\} }\,\d\mu(h) \,=\, \lambda,
	\end{align}
	where $\lambda>0$ was defined in \eqref{e:lambda}. This controls also the prefactor from \eqref{upper_bound_1} uniformly in $N$, finishing the proof. 
\end{proof}

	\subsection{Uniform  estimates on the associated  diffusion processes}
	\label{SS:discr-reg}
	Equipped with the results of the preceding subsections, 	we  produce a uniform version of Proposition \ref{prop_reg_est}. 
	\begin{proposition}\label{proposition_unif_reg}For the  diffusion processes $((h_{t} )_{t\ge 0}, (P_{N,z})_{z\in \MM_N^+})$ associated to $(\wt\EE_N,\D(\wt\EE_N))$, we have  
		\begin{equation}\label{Eq264-unif}
		\sup_{N\ge 2}	E_{\mu_N^+} \Bigl[
			\|h\|_{C^{\gamma/4,\gamma}((0,T)\times \T)}^p
			\Bigr]^{1/p} \,<\,\infty,
		\end{equation}
		for each $\gamma \in (0,1/2)$, $p\in (0,\infty)$ and $T<\infty$.
	\end{proposition}
	\begin{proof}Since the proof is mostly analogous to the one of Proposition \ref{prop_reg_est}, we comment mainly on the necessary modification.
		We use here that it suffices to show that
		\begin{align}
			E_{\mu_N^+} ^{1/p}
			|h_{t}(x)  - h_s(y)|^p
			\,\lesssim_{(\delta, p,T)}\, |t-s|^{1/8-\delta}+|x-y|^{1/2-\delta},
		\end{align} 
		for all $0\le s\le t \le T$, $x,y\in \T$ and $\delta>0$. By  Lemma \ref{p1}, we can further reduce to showing 
		\begin{align}\label{eqn6}E_{\mu_N^+} ^{1/p}
			|h_{t}(x)  - h_s(x)|^p \,\lesssim_{(p,T)}\, |t-s|^{1/8-\delta}.
		\end{align}
		Taking $\psi_\epsilon$ 
		as  a standard mollifier  with $\supp \psi_\epsilon = [-\epsilon,\epsilon]$, we  split 
		\begin{align}&
			E_{\mu_N^+} ^{1/p}
			|h_{t}(x)  - h_s(x)|^p \\&\quad \le\,
			E_{\mu_N^+} ^{1/p}
			|h_{t}(x)  - \psi_\epsilon*h_t(x)|^p \,+\,
			E_{\mu_N^+} ^{1/p}
			|\psi_\epsilon*h_{t}(x)  - \psi_\epsilon*h_s(x)|^p
			\,+\, 
			E_{\mu_N^+} ^{1/p}
			|\psi_\epsilon*h_{s}(x)  - h_s(x)|^p,
		\end{align}
		where the convolution with respect to the even extension can be rephrased using 
		 $\overline{\psi_\epsilon}$ as in \eqref{eqn327}--\eqref{eqn8324}.
		 The first and last terms can be estimated using Lemma \ref{p1}, while for the second term we use the Lyons--Zheng decomposition, resulting in 
		 \begin{align}&E_{\mu_N^+} ^{1/p}
		 	|\psi_\epsilon*h_{t}(x)  - \psi_\epsilon*h_s(x)|^p\,\lesssim \, |t-s|^{1/2}\biggl( \int_{\MM_N^+}  \biggl(\int_{\T} m(h)( \partial_z\hat{\rho}*\overline{\psi_\epsilon}(x,\cdot) )^2\,\d z\biggr)^{p/2} \,\d\mu_N^+(h)\biggr)^{1/p},
		 \end{align}
		 based on the bound \eqref{E89} for \eqref{eqn924}. We recall that $\hat{\rho}$
		 is the 
		 hat function characterized by $\hat{\rho}(0) = N$ and $\hat{\rho}(\pm 1/N) = 0$. Since the convolution is defined with respect to the even extension of $\overline{\psi_\epsilon}$, 
		 we may use  \eqref{eq440}, Lemma \ref{p1} and Young's inequality for convolution to obtain the further bound
		  \begin{align}&E_{\mu_N^+} ^{1/p}
		 	|\psi_\epsilon*h_{t}(x)  - \psi_\epsilon*h_s(x)|^p\, \lesssim_{p}\,|t-s|^{1/2}\| \overline{\psi_\epsilon}(x,\cdot)\|_{\dot{H}^1(\T)}.
		 \end{align}
		 From here one can proceed as in Proposition \ref{prop_reg_est}.
	\end{proof}

	\section{Positivity  of the discretized diffusion and identification of its martingale problem}
	\subsection{Assumptions and results of this section}
	As mentioned in Subsection \ref{SS_tightness_dircscrete}, one of the contributions of \cite{GGKO21} is that solutions to the problem that 
	\eqref{Eq191} is a martingale for any 
		$\eta \in C^\infty(\MM_N^+)$ preserve positivity for  mobility functions $m$ with sufficient decay at $h=0$.    In particular, the  threshold $n\ge 3$ for $m(r) =r^n$ 
		was identified. 
		To see that there is some room for improvement, we consider the situation that  $N=1$ and $m(r)=r^n$, in which case the generator acts on $\eta (h)$, where $h=h(0)\in [0,2]$ is an affine transformation of the $\alpha$-coordinate, via  
		\begin{align}\begin{split}\label{Eq_gen_1d}4 &\Bigl( \frac{1}{\wt{g}(h)}  \partial_{hh} \eta(h)\,-\,\frac{\wt{g}'(h)}{ \wt{g}^2(h)}\partial_h\eta(h) \,+\, \frac{2(2-2h)}{\wt{g}(h)}\partial_h\eta(h)\Bigr),\\& \text{where}\qquad  \frac{1}{\wt{g}(h)}  =   \begin{cases}
				\frac{2-2h}{\log(2-h) - \log(h)  }, & n=1,\\
				(n-1)\frac{2h -2}{(2-h)^{1-n} - h^{1-n}  }, &  n\ne  1,
			\end{cases}
		\end{split}
		\end{align}
	cf.\ \eqref{E11}. 
		The Feller test, which we perform in Subsection \ref{SS_Feller} below, shows  that the associated Markov process preserves positivity even for $n\ge 2$; see Table \ref{table1}.  Therefore, we revisit in this section the  positivity of solutions to the martingale problem \eqref{Eq191} and its  relation  to the Dirichlet form \eqref{e:Dirichlet_form_+_discr}. We find the following.

	\begin{table}
		\begin{tabular}{|c|c|}\hline
			$0 \le n <2$ & \text{regular (accessible)}\\ \hline
			$2\le n< 3$ & \text{entrance (inaccessible)}\\ \hline
			$n\ge 3$ & \text{natural (inaccessible)} \\\hline
		\end{tabular}
		\vspace{.2cm}
		\caption{Feller's boundary classification for $N=1$.}\label{table1}
	\end{table}

	\begin{theorem}\label{thm_nonnegativity}
		Assume that $m$ is continuously differentiable and satisfies \eqref{ass_n}.
		\begin{enumerate}[(i)]
			\item\label{L2} If $n\ge 2$ and  $h$ is a solution to the martingale problem \eqref{Eq191} 
			with $\P(\forall i :  h_0(x_i) >0) = 1$, then  $\P(\forall i, t\ge 0 : h_t(x_i) >0) =1$. 	 
			\item \label{L3} If $n \ge 1$ and $m$ is continuously differentiable and increasing,
			 then the diffusion process $h$ associated to $(\wt{\EE}_N,\D(\wt{\EE}_N))$ solves the martingale problem \eqref{Eq191} when started from $\mu_N^+$.
		\end{enumerate}
	\end{theorem}

	Part \eqref{L2} relies on controlling the entropy, cf.\ Subsection \ref{SS:GR-scheme}. While this is also the idea of \cite[Theorem 8.4]{GGKO21}, where higher moments of the entropy are controlled, we include the range  $n\in [2,3)$ by relying on  moments less than one only. In particular, we do not control the quadratic variation of emerging martingales, but instead eliminate them using a stopping time argument. Comparison with Table \ref{table1} shows that this result is optimal. 
	Part \eqref{L3} may be interpreted as a preservation of non-negativity result: There exists a solution to the martingale problem \eqref{Eq191} which is distributed according to $\mu_N^+$ for all times. In particular, no Skorokhod  term is necessary to prevent $h$ from becoming negative for the specified mobility functions, which is in contrast to, say, a Brownian motion. Its proof relies on an integration by parts formula for $\wt{\EE}_N$ together with  Dirichlet form  arguments.

	\subsection{Details on the Feller test}\label{SS_Feller}
	Here we assume that   $N=1$ and  $m(r)=r^n$. 
	 The solution to the martingale problem associated to the generator \eqref{Eq_gen_1d} solves  the stochastic differential equation 
		\begin{equation}\label{Eq_1d_SDE}
			\d h_t \,=\, 4\biggl( \frac{2(2-2h_t)}{\wt{g}(h_t)} \,-\,\frac{\wt{g}'(h_t)}{ \wt{g}^2(h_t)} \biggr)\,\d t\,+\, \sqrt{\frac{8}{ \wt{g}(h_t)}} \, \d w_t.
		\end{equation}
		Feller proposed in   \cite{feller_annals}  the transformation $h\mapsto s(h)$, where  $s$ satisfies 
		\[
		s'(h) \,=\, 
		\exp \biggl(
		-\int_1^h 
		2 (2-2r)\,-\, \frac{\wt{g}'(r)}{\wt{g}(r)}
		\,\d r
		\biggr),
		\] 
		which results by It\^o's formula in the driftless equation 
		\begin{equation}\label{Eq_sv}
			\d s(h_t) \,=\, s'(h_t) \sqrt{ \frac{8}{\wt{g}(h_t)}} \,\d w_t.
		\end{equation}
		The factorization
		\[
		s'(h) \,=\, 
		\exp \biggl(
		-\int_1^h 
		2 (2-2r) \,\d r \biggr) \exp \biggl(\int_1^h \frac{\wt{g}'(r)}{\wt{g}(r)}
		\,\d r
		\biggr)
		\]
		shows that the behavior of $s$ near the boundary is determined by the second term on the right-hand side. Assuming for simplicity that $n>1$, we have
		\begin{align}\label{Eq221}
			\wt{g}(h) \sim h^{1-n} \qquad \text{and thus also} \qquad s'(h) \sim h^{1-n},
		\end{align}
		for $h\sim 0$, by \eqref{Eq_gen_1d}. 
			Since the latter is not integrable for $n\ge 2$ and $s(h_t)$ satisfies the driftless equation \eqref{Eq_sv}, the boundary is \emph{inaccessible}, see \cite[Chapter 33]{kallenberg1997foundations}. 
			In the case $n<2$, the state space of $s(h_t)$ is a finite interval and its speed measure is 
			\[
			\frac{\wt{g} }{8(s')^2} \circ s^{-1}(x) \, \d x.
			\]
			Accordingly, the behavior of  $s(h_t)$ near $s(0)$ can be  deduced from the integral 
			\[
			\int_{s(0)}^{s(1)} 	\frac{ \wt{g} }{8 (s')^2} \circ s^{-1}(x) \,\d x \,=\, \frac{1}{8}\int_0^{1}
			\frac{\wt{g}(h) }{s'(h)}
			\,\d h.
			\]
			By  \eqref{Eq221} the latter is always finite, so that   \cite[Theorem 33.12]{kallenberg1997foundations} yields then that the boundaries of the new state space are \emph{accessible} 
			for $n<2$.
			We remark that this classification is consistent  with  the  possible choices of domain for the Dirichlet form \eqref{e:Dirichlet_form_+_discr}: Following \cite[Section 2.2.3]{Chen_Fukushima}, the above implies that the set of  $\eta \in C_c^1((0,2))$ is dense in $	\D (\wt{\EE}_1)$ for $n\ge 2$ while its closure is complemented by a two-dimensional subspace for $n<2$.
			
			Beyond the accessibility of the boundary points, Feller's test distinguishes between the finiteness of 
			\begin{align}
				-
				\int_0^1 \frac{\wt{g}(h)s(h)}{ s'(h)} \,\d h,
			\end{align}
			for which the critical threshold is $n=3$ by \eqref{Eq221}, see also \cite[Theorem 8.1.1]{EK}. The result is summarized in Table \ref{table1}, and may be stated in words as follows: For $n<2$ the Markov process with generator \eqref{Eq_gen_1d} may reach the boundary from the interior. For $n\ge 2$ it does not, while for $n\in [2,3)$ it may at least be started from there. 
	\subsection{The case $n\ge 2$}
		In this subsection, we verify part \eqref{L2} of Theorem \ref{thm_nonnegativity}, for which we recall the discrete entropy functional 
		\[
		S(h) = \frac{1}{N}\biggl( \frac{s(h(0)) + s(h(1))}{2}+\sum_{i=1}^{N-1}  s(h(x_i)) \biggr),\qquad s(h) = \int_1^h \int_1^r \frac{1}{m(r')}\,\d r'\,\d r,
		\]
		from Subsection \ref{SS:GR-scheme}. Due to the assumption that \eqref{ass_n} holds with $n\ge 2$, we observe that  
		\begin{align}
			S(h)<\infty \qquad \Rightarrow \qquad \forall i\,:\, h(x_i)>0,
		\end{align}
		so that a localization in $\omega$ of the following estimate yields the claim.
	\begin{proposition}\label{lem:improved_entropy_est}
		Under the assumptions of Theorem \ref{thm_nonnegativity}~\eqref{L2}, if $h$ is a solution to the martingale problem \eqref{Eq191}, we have   
		\begin{align}&
					\E \biggl[
					\sup_{t\le T}\biggl(  S(h_t)+
					\frac{1}{N}\int_0^t2 {((A^\top A h_s)^0)^2
						+2((A^\top A h_s)^N	)^2
					}+\sum_{i=1}^{N-1} \bigl((A^\top A h_s)^i \bigr)^2 \,\d s
					\biggr)^p\biggr] \\&\quad \le\,  \frac{1}{1-p}\biggl(
					\E [
					S(h_{0})
					] \,+\, {2N^2(N+1)T }
					\biggr)^p,
			\end{align}
		for all $p\in (0,1)$ and $T<\infty$.
	\end{proposition}
	\begin{proof}
		For $R<\infty$, we introduce the stopping time 
			\[
		\tau_R \,=\, \inf\{ t\ge 0 \,| \, S(h_t) \ge  R\},
		\]
		so that until $\tau_R$, we may take  $\eta = S$ in \eqref{Eq191}. 
		Thereby, using \eqref{eqn43ubv}, we find that
			\begin{align}\begin{split}\label{eqn66}&
						S(h_{t\wedge \tau_R})
						+
						\frac{1}{N}\int_0^{t\wedge \tau_R}2 ((A^\top A h_s)^0)^2
							+2((A^\top A h_s)^N	)^2
					 +\sum_{i=1}^{N-1} \bigl((A^\top A h_s)^i \bigr)^2 \,\d s
						\\&\quad =\, S(h_{0}) \,+\,  M_{t\wedge\tau_R}^R
						\,+\, \int_0^{t\wedge \tau_R}
						\partial_{\alpha'} (\wt{g}^{\alpha \alpha'} \partial_\alpha S)(h_s)\,\d s
						,\end{split}
			\end{align}
		for some martingale $M^R$. 	
		 As in \cite[eqs. (8.12)--(8.13)]{GGKO21} we use once more  \eqref{E12} together with \eqref{eqn_honig} to rewrite the last integral in light of  
		 \[\partial_{\alpha'} (\wt{g}^{\alpha \alpha'} \partial_\alpha S)\,=\, 
		  (1+\mathbf{1}_{\{0,N\}}(i)) \partial_i (A^\top A h)^i \,=\, 2N^2(N+1)
		 \]
		 as simply $2N^2(N+1)(t\wedge \tau_R)$. Using  martingality, we obtain after taking the expectation 
		 		\[
		\E \bigl[
		L({\sigma\wedge T \wedge \tau_R  })
		\bigr] \,\le \, \E [
		S(h_{0})
		] \,+\, {2N^2(N+1)T },
		\]
		where $L$ denotes the process on the left-hand side of \eqref{eqn66},
		for any  stopping time $\sigma$. Choosing $\sigma = \inf\{ t\ge 0 | L(t) \ge  \lambda \}$ for $\lambda\ge 0$,  we find 
		\begin{align}
			\P\Bigl(
			\sup_{t\le T\wedge \tau_R} L(t) \,\ge  \, \lambda
			\Bigr) \,\le\, \P\bigl(L({\sigma\wedge T \wedge \tau_R  }) \,\ge\, \lambda
			\bigr)\,\le \, \lambda^{-1}\bigl(
			\E [
			S(h_{0})
			] \,+\, {2N^2(N+1)T }
			\bigr) ,
		\end{align}
		by means of the Chebyshev--Markov inequality. This implies that
		\begin{align}
			\E \Bigl[
			\sup_{t\le T\wedge \tau_R} L^p(t)
			\Bigr] \, =\, \int_0^\infty  	\P\Bigl(
			\sup_{t\le T\wedge \tau_R} L(t) \,> \, \lambda^{1/p}
			\Bigr) \,\d \lambda \,\le \, r \,+\,\frac{ p}{1-p} \bigl(\E [
			S(h_{0})
			] \,+\, {2N^2(N+1)T }
			\bigr) r^{1-1/p},
		\end{align}
		for any $r\in (0,\infty)$.
		Optimizing over the latter yields 
		\[
		\E \Bigl[
		\sup_{t\le T\wedge \tau_R} L^p(t)
		\Bigr]\,\le\, \frac{1}{1-p}\bigl(\E [
		S(h_{0})
		] \,+\, {2N^2(N+1)T }
		\bigr)^p,
		\]
		and it remains to eliminate the stopping time $\tau_R$. 
		For this we use the preceding estimate to deduce 
		\[
		\P ( \tau_R < T ) \,\le \,  \P\Bigl(
		\sup_{t\le T\wedge \tau_R} L(t) \,\ge  \, R
		\Bigr) 
		\le\, \frac{R^{-p}}{(1-p)}\bigl(\E [
		S(h_{0})
		] \,+\, {2N^2(N+1)T }
		\bigr)^p\,\to\, 0,
		\]
		as $R\to \infty$, so that the claim follows by Fatou's lemma.
	\end{proof}

	\subsection{The case $n\ge 1$}
	Here we provide the proof of part \eqref{L3} of Theorem \ref{thm_nonnegativity}. Our argument is based on  the following lemma.
	
	\begin{lemma} \label{lemma_IBP_fd}
		Under the assumptions of Theorem \ref{thm_nonnegativity}~\eqref{L3}, we have for all $\eta\in C^2(\MM_N^+)$ and $\zeta \in C^1(\MM_N^+)$ 
			\begin{align}\label{eqn234}
				\wt{\EE}_{N}(\eta ,\zeta) \,=\,  -  \int_{\MM_N^+} \zeta 
				\Bigl(\tilde g^{\alpha \alpha '} \partial_{\alpha \alpha'} \eta \,+\,  (\partial_\alpha\tilde  g^{\alpha \alpha'} )\partial_{\alpha'}\eta \,-\, \tilde g^{\alpha\alpha '} (\partial_\alpha E) \partial_{\alpha'} \eta
				\Bigr) 
				\, \d \mu_{N}^+,
			\end{align}
		where the term in the brackets lies in $L^1(\mu_N^+)$. 
	\end{lemma}
	Before proving the lemma, we show how it entails the claim by standard Dirichlet form arguments.
		\begin{proof}[Proof of Theorem \ref{thm_nonnegativity}~\eqref{L3}]
		 Applying  Fukushima's decomposition, Theorem \ref{thm_fukushima},  to $\eta \in C^\infty(\MM_N^+)$, we find that 
		 \begin{align}
		 	\eta(h_t) \,-\, \eta(h_0) \,=\,  M_t \,+\, N_t,
		 \end{align}
		 for a martingale additive functional of finite energy  $(M_t)_{t\ge 0}$ and a continuous additive functional of zero energy $(N_t)_{t\ge 0}$. When the diffusion is started from $\mu_N^+$, $M_t$ is moreover a true martingale, and it remains to identify $N_t$ as the temporal integral in \eqref{Eq191}. To this end, we use \cite[Corollary 5.4.1]{FOT} together with  Lemma \ref{lemma_IBP_fd} to deduce  that $N_t = N_t^+ -  N_t^- $ for positive continuous additive functionals  $(N^{\pm}_t)_{t\ge 0}$ with Revuz measures  
		\[
		\zeta^{\pm} \d \mu_{N}^+,\qquad\text{for}\qquad \zeta^{\pm}\,=\,
	\Bigl(\tilde g^{\alpha \alpha '} \partial_{\alpha \alpha'} \eta \,+\,  (\partial_\alpha \tilde g^{\alpha \alpha'} )\partial_{\alpha'}\eta \,-\, \tilde g^{\alpha\alpha '} (\partial_\alpha E) \partial_{\alpha'} \eta
	\Bigr) ^{\pm},
		\]
		respectively, see \cite[Section 5.1]{FOT} for this terminology. Now  let 
		$ \zeta_k^\pm \in L^2(\mu_{N}^+)$ be non-negative such that $\zeta_k^\pm \nearrow \zeta^\pm$ in $L^1(\mu_{N}^+)$, so that on the one hand  
		\[
		E_{\mu_N^+ }\biggl[
		\biggl| N^{\pm}_t    - \int_0^t 
		\zeta_k^\pm ( h_s ) \,\d s\biggr|
		\biggr]\,\le\, (1+t) \|\zeta^\pm - \zeta^\pm_k \|_{L^1(\mu_{N}^+) }\,\to \, 0,
		\]
		as $k\to\infty$, uniformly on bounded time intervals
		by \cite[Lemma 5.1.9]{FOT}. On the other hand, using invariance of the measure $\mu_{N}^+$ for $h$, we have that
		\[
		E_{\mu_N^+ }\biggl[\sup_{t\le T}
		\biggl|  \int_0^t \zeta^\pm  (h_s) - 
		\zeta_k^\pm ( h_s ) \,\d s \biggr|
		\biggr] \,\le\, T \|\zeta^\pm - \zeta^\pm_k \|_{L^1(\mu_{N}^+) } \,\to\, 0,
		\] identifying $N_t$ as the desired expression. 
	\end{proof}
	
	\begin{proof}[Proof of Lemma \ref{lemma_IBP_fd}]
		We need to justify the  integration by parts formula 
		\eqref{eqn234}. Since mass is fixed as $1$, we have $h^i \le 2N$ for all $i$ at fixed discretization length $1/N$, so that  \eqref{ass_n} holds also with $n=1$. Then we see that the diagonal $\wt{g}^{\alpha \alpha'}$ from \eqref{E11} satisfies the bound 
		\begin{align}\label{eqn433}
	\tilde 	g^{\alpha \alpha}(h) \,\lesssim\, \biggl(\fint_{I_\alpha} \frac{1}{h}\,\d x \biggr)^{-1} 
		\,=\,  	\frac{h^{\alpha+} - h^{\alpha - }}{\log(h^{\alpha+}) -\log (h^{\alpha-})},
		\end{align} 
		which tends to $0$ if either $h^{\alpha+}$ or $h^{\alpha-}$ go to $0$. As the normal vector at 	\[  \bigl\{h\in 
		\partial\MM_N^+\, \big|\, h^i =0 \bigr\},
		\]
		is perpendicular to $\bar{\vp}_{\alpha}$ unless $i  = \alpha \pm $, boundary contributions vanish therefore and 
		it remains to argue that the integrand on the right-hand side of \eqref{eqn234} is integrable.  In light of the bound \eqref{eqn433}, the latter depends only on whether  $\partial_\alpha \tilde g^{\alpha \alpha} 
		\in L^1(\mu_N^+)$. Since $m$ is increasing, the  implicitly defined function  $ G(h^{\alpha-},h^{\alpha+})=\tilde g^{\alpha \alpha} $ is increasing in each argument. This implies that 
		\begin{align}
			|\partial_{\alpha}\tilde g^{\alpha \alpha}| \lesssim_N \partial_1 G(h^{\alpha-},h^{\alpha+}) + \partial_2 G(h^{\alpha-},h^{\alpha+})
		\end{align}
	and the right-hand side is $\mu_N^+$-integrable, since 
	\begin{align}
			\int_0^{2N} \int_0^{2N}  \partial_1 G(y_1 ,y_2) + \partial_2 G(y_1 ,y_2) \,  \d y_1\d y_2\, \lesssim_N\,  G(2N,2N) \,<\,\infty.
	\end{align}
	\end{proof}

	\section{Mosco convergence for non-degenerate metrics}\label{Sec_mosco_abstr}
	\subsection{Setup and result}Throughout this section, we impose the following. 
		\begin{assumption}\label{ass_Mosco_abstr}Let $\HH_{\pm 1}$ and $g$ be as in Assumption \ref{Ass_abstract} with $V=0$. Moreover, for each $N\ge 2$, we assume the same for finite dimensional  spaces $  \HH_N   \subset \HH_1$ and respective  inner products $g_N$. We impose further that the orthogonal projections $P_N \colon \HH_1 \to \HH_N$ converge strongly to the identity on $\HH_1$.
		\end{assumption}
		We assume that $\mu$ and $\mu_N$ are the Gaussian measures with Cameron--Martin spaces $\HH_1$ and $\HH_N$, respectively.  Introducing also the set of test functions 
					\begin{align}\label{eqn_CC_0}
						\CC_0 \,=\, \biggl\{
					\Phi\bigl( (\cdot, p^1)_{\mathcal{H}_{1}\times \mathcal{H}_{1}^*} , \dots,(\cdot, p^k)_{\mathcal{H}_{1}\times \mathcal{H}_{1}^*}  \bigr) ,\, 
					p^i \in \mathcal{H}_{1}^*,\,  \Phi \in C^1_b(\R^k)\biggr\} 
		\end{align} 
		we have that $\CC_0 \subset L^2(\mu)$ is dense, because it contains $\D_0$,  cf.\ Theorem \ref{thm_abstract}. Moreover, since under $\mu_N$ the vector 
		\[
		(\cdot, p^1)_{\mathcal{H}_{1}\times \mathcal{H}_{1}^*} , \dots,(\cdot, p^k)_{\mathcal{H}_{1}\times \mathcal{H}_{1}^*}
		\]
		is jointly Gaussian with covariances $(P_N^*p^i ,P_N^* p^j)_{\HH_1^*}\to (p^i , p^j)_{\HH_1^*} $, we see that   $\|\zeta\|_{L^2(\mu_N)} \to \|\zeta\|_{L^2(\mu)}$ for all $\zeta\in \CC_0$. We remark that we identify here and below $\HH_N^*$ with those elements of $\HH_{1}^*$ that vanish on the $\HH_{1}$-orthocomplement of $\HH_{N}$,  so that the adjoint projection is $P_N^*\colon \HH_{1}^* \to \HH_N^*, p\mapsto (P_N \cdot, p)_{\mathcal{H}_{1}\times \mathcal{H}_{1}^*} $. In particular, we may identify $L^2(\mu_N)$ as a sequence of Hilbert spaces converging to $L^2(\mu)$ in the sense of Appendix \ref{App:Mosco}. 
		
		The aim of  this section is to give a sufficient condition for the Mosco convergence of the closures of 
		\begin{align}
			\EE_N(\eta,\eta )\,=\, \int_{\X} g_N^h(\diff|_h \eta, \diff|_h \eta  ) \,\d \mu_N (h),
			\qquad
			\text{to}\qquad 
			\EE (\eta,\eta )\,=\, \int_{\X} g^h(\diff|_h \eta, \diff|_h \eta  ) \,\d \mu (h),
		\end{align}
		initially defined for $\eta \in \D_0$,
		on $L^2(\mu_N)$ and $L^2(\mu)$, respectively.
		 For an intermediate step, we also need the weak Dirichlet form 
			\begin{align}\frac{1}{2}\EEw (\eta,\eta ) \,=\, \sup \biggl\{
				 \int_{\X} \eta \div_\mu(X) \,\d \mu - \frac{1}2 \int_{\X}  
				g_\cdot(X,X)\,
				\d\mu \,\bigg|\, 
				X = \sum_{i=1}^k \zeta^i \dot{h}_i,\, \zeta^i \in \CC_0,\, \dot{h}_i \in \HH_1   
				\biggr\},
			\end{align}
		defined on $\D(\EEw)$, which consists of all $\eta\in L^2(\mu)$ for which the above is finite.
		Here,  $g_h$ is the dual inner product of $g^h$, defined on $\HH_{-1}$ and 
		\begin{align}\label{eqn_divmu}
			\div_{\mu}(X) \,=\, \sum_{i=1}^k\diff {\zeta^i}_{\cdot} \dot{h}_i \,-\, \zeta^i(\cdot, \dot{h}_i)_{\HH_1},
		\end{align}
	for vector fields
	\begin{align}\label{eqn_XX} 
		X = \sum_{i=1}^k\zeta^i \dot{h}_i,\qquad \zeta^i \in \CC_0, \, \dot{h}_i \in \HH_1.
		\end{align}
	We notice that the latter defines an element of $L^2(\mu)$ independent of the representation of $X$, since 
	\begin{align}\int_\X \eta \div_\mu(X) \,\d \mu \,+\, \int_\X \diff \eta_{\cdot } X \,\d\mu \,=\,\sum_{i=1}^k\biggl(  
		\int_\X \diff(\eta \zeta^i )_\cdot \dot{h}_i \,\d\mu - \int_{\X} \eta \zeta^i (\cdot ,\dot{h}_i)_{\HH_1} \,\d\mu  \biggr)
		 \,=\, 0,
		\label{GI_eqn}
	\end{align}
	 for each $\eta \in \CC_0$ by the Cameron--Martin theorem, cf.\ \eqref{e:mu_div}. 
	
	\begin{remark}\label{rem_weak_to_strong}It always holds that $\EEw \le \EE $, where as usual we set $\EE(\eta,\eta)=\infty$ for $\eta\notin \D(\EE)$. Indeed, integrating the pointwise inequality 
		\[
		\diff \eta|_{h\cdot} X_h \,\le\,\frac{1}{2} g^h(\diff|_h\eta ,\diff|_h\eta)\, +\, \frac{1}{2} g_h (X_h ,X_h)
		\]
		with respect to $\mu$ and using \eqref{GI_eqn} yields the desired  bound  on $\D_0$. It extends then to $\D(\EE)$ by density.
		\end{remark}
	Using the above, we identify an upper and a lower bound for the sequence of Dirichlet forms $\EE_N$ for suitably converging  metric tensors. We remark that while the projection $P_N$ is defined pointwise only on ${\HH}_1$, it can be defined  $\mu$-a.e.\ on $\X$ by the Paley--Wiener isomorphism.
		\begin{theorem}\label{thm_mosco_abstract}
			 Under Assumption \ref{ass_Mosco_abstr} we have the following:
			\begin{enumerate}[(i)]
				\item\label{Mosco_I1} If  for all $p \in \HH_{-1}^*$, for $\mu$-a.e.\ $h\in \X$
				\begin{align}\label{eqn_ass_convergence}\limsup_{N\to\infty} g_N^{ P_Nh}(P_N^* p,P_N^*  p) \,\le\, g^h(p,p),
				\end{align}
				and $  g_N^{P_N \cdot }(P_N^*  p, P_N^* p) $ is uniformly $ \mu $-integrable, then 
				\begin{align}\label{eqn_Mosco_upper}
					\limsup_{N\to\infty} \EE_N (\eta ,\eta ) \,\le\, \EE(\eta, \eta ),
				\end{align}
				for all $\eta \in \D_0$. 
				\item\label{Mosco_I2} If for all $\dot{h} \in \HH_1$, for $\mu$-a.e.\ $h \in \X$, 
				\begin{align}
					\limsup_{N\to\infty} g_{N, P_Nh} (P_N\dot{h} , P_N\dot{h}) \,\le\, g_h(\dot{h},\dot{h})
				\end{align}
				and $g_{N, P_N\cdot } (P_N\dot{h} , P_N\dot{h})$  is uniformly $ \mu $-integrable, then 
				\begin{align}\label{lower_bound_eqn}
					\liminf_{N\to\infty} \EE_N(\eta_N,\eta_N ) \,\ge\, \EEw(\eta,\eta),
				\end{align}
				for all $L^2(\mu_N) \ni \eta_N$ converging weakly to $  \eta \in L^2(\mu)$.
			\end{enumerate}
		\end{theorem}
		To deduce Mosco convergence from Theorem \ref{thm_mosco_abstract}, it remains to identify the upper and lower bound with each other. We achieve this for non-degenerate and uniformly bounded metric tensors, essentially by comparing to the flat case.
		\begin{proposition}\label{Prop_weak_is_strong}
			Under Assumption \ref{ass_Mosco_abstr}, if there are $0<c\le C<\infty$ such that
			\begin{equation}\label{comparable_to_flat}
				c|p|_{\HH_{-1}^*}^2  \, \le \,
				g^h ( p,p) \,\le \, 
				C|p|_{\HH_{-1}^*}^2   ,
			\end{equation}
			for all $h\in \X$, then $(\EE,\D(\EE)) =(\EEw, \D(\EEw))$. In particular, if the conditions of Theorem \ref{thm_mosco_abstract} \eqref{Mosco_I1} and \eqref{Mosco_I2} are satisfied,  $\EE_N $ converges to $\EE$ in Mosco sense.
		\end{proposition}

	\subsection{Proof of the upper bound}
	We verify the first part of Theorem \ref{thm_mosco_abstract}. 
	\begin{proof}[Proof of Theorem \ref{thm_mosco_abstract}~\eqref{Mosco_I1}] As $\mu_N=  \mu \circ P_N^{-1} $, we may write for $\eta\in \D_0$ and $\delta>0$
		\begin{align}&\limsup_{N\to \infty}	\EE_N(\eta,\eta )  \,=\,\limsup_{N\to\infty} \int_{\X}g_N^{P_Nh}(\diff|_{P_N h} \eta ,\diff|_{P_Nh}  \eta )
			 \,\d\mu(h)\\&\quad \le C_\delta \lim_{N\to\infty } 
			 \int_{\X}g_N^{P_Nh}(\diff|_{P_N h} \eta -   \diff|_{h} \eta, \diff|_{P_N h} \eta -   \diff|_{h} \eta )
			 \,\d\mu(h)\\&\qquad +(1+\delta) \limsup_{N\to\infty}  \int_{\X}g_N^{P_Nh}( \diff|_{h} \eta,   \diff|_{h} \eta )
			 \,\d\mu(h)\\&\quad =:\, C_\delta\lim_{N\to\infty}I_1(N)\,+\, (1+\delta)\limsup_{N\to\infty}  I_2(N),
		\end{align} 
	where we  do not distinguish between $P_N^*p $ and $p$, since the former is just the restriction of the latter. 
	
	Regarding $I_1$, we assume that $\eta$ is of the form \eqref{eqn_eta}, so that one can expand
	\begin{align}\label{eqn2736}
		I_1(N) \,=\,\sum_{i,j=1}^k \int_{\X} \Bigl( \partial_i\Phi( l( P_Nh )) - \partial_i\Phi( l( h )) \Bigr)\Bigl( \partial_j\Phi( l( P_Nh )) - \partial_j\Phi( l( h )) \Bigr) g_N^{P_N h }(p^i,p^j) \,\d \mu(h),
	\end{align}
	if one abbreviates 
	\[
	l \,=\, \bigl( (\cdot, p^1)_{\mathcal{H}_{-1}\times \mathcal{H}_{-1}^*} , \dots,(\cdot, p^k)_{\mathcal{H}_{-1}\times \mathcal{H}_{-1}^*}  \bigr). 
	\]
	Then, for any subsequence, by the Paley--Wiener isomorphism, we have $l(P_N\cdot) \to l$ in $L^2(\mu)$ and therefore, up to passing to a further subsequence, also $\mu$-almost everywhere. Since by assumption the integrand of \eqref{eqn2736} is additionally uniformly $\mu$-integrable, it follows from the continuity of $\partial_i\Phi$ that $I_1(N)\to 0$ along this subsequence. Since this argument may be repeated for any subsequence of $I_1(N)$, we find that also the original sequence tends to $0$. 
	
	To deal with $I_2$, we may by assumption find for each $\epsilon>0$ a $K<\infty$ such that
	\[I_2(N) \le\int_{\X} K\wedge g_N^{P_Nh}( \diff|_{h} \eta,   \diff|_{h} \eta ) 
	\,\d\mu(h) \,+\,\epsilon,
	\]
	for all $N$. Using \eqref{eqn_ass_convergence}, we deduce then by means of (the reverse) Fatou's lemma that
	\[
	\limsup_{N\to \infty} I_2(N)  \le\int_{\X} K\wedge g^{h}( \diff|_{h} \eta,   \diff|_{h} \eta ) 
	\,\d\mu(h) \,+\,\epsilon\,\le\, \EE(\eta ,\eta) +\epsilon.
	\]
	
	All in all, we have shown that 
	\[
	\limsup_{N\to \infty}	\EE_N(\eta,\eta ) \le (1+\delta)\bigl(   \EE(\eta ,\eta) +\epsilon \bigr),
	\]
	and since $\delta,\epsilon>0$ are arbitrary, \eqref{eqn_Mosco_upper} follows. 
\end{proof}
	\subsection{Proof of the lower bound}
	In the following, we also need $\div_{\mu_N}(X)$, which we define as \eqref{eqn_divmu} for $\HH_N$-valued vector fields. 
Then, the $\mu_N$-divergences converge strongly.
	\begin{lemma}For $X$ as in \eqref{eqn_XX}, we have $L^2(\mu_N)\ni \div_{\mu_N}(P_N X) \to \div_\mu(X) \in L^2(\mu)$.\label{Lemma_div_conv}
	\end{lemma}
\begin{proof}
	We need to check that for all $\eta \in \CC_0$, we have 
	\begin{align}\begin{split}\label{eqn82164321866213}
			&
		(\eta , \div_{\mu_N}(P_N X))_{L^2(\mu_N)} \to (\eta, \div_\mu (X))_{L^2(\mu)},
		\\
		&\| \div_{\mu_N}(P_N X) \|_{L^2(\mu_N)}\to \| \div_{\mu}( X) \|_{L^2(\mu)}, 
		\end{split}
\end{align}
as $N\to\infty$. 
For this we use that for $X$ as in \eqref{eqn_XX}, we may evaluate 
\begin{align}
	\div_\mu (X)  \,=\,   \sum_{i=1}^k\diff {\zeta^i}_{\cdot} \dot{h}_i \,-\, \zeta^i(\cdot, \dot{h}_i)_{\HH_1} ,
\end{align}
and the latter summands further as
\begin{align}&
	\diff \zeta|_{\cdot} X \,-\, \zeta (\cdot, X)_{\HH_1}  \, =\,\sum_{i=1}^k \partial_i \Phi\bigl( (\cdot, p^1)_{\mathcal{H}_{1}\times \mathcal{H}_{1}^*} , \dots,(\cdot, p^k)_{\mathcal{H}_{1}\times \mathcal{H}_{1}^*}  \bigr) (X,p^i)_{\mathcal{H}_{1}\times \mathcal{H}_{1}^*}- \zeta (\cdot, X)_{\HH_1} ,
\end{align}
if 
\[
	\zeta \,=\,\Phi\bigl( (\cdot, p^1)_{\mathcal{H}_{1}\times \mathcal{H}_{1}^*} , \dots,(\cdot, p^k)_{\mathcal{H}_{1}\times \mathcal{H}_{1}^*}  \bigr).
\] 
The summands making up $\div_{\mu_N}(P_N X)$ differ only slightly:  they are of the form 
\[\sum_{i=1}^k\partial_i \Phi\bigl( (\cdot, p^1)_{\mathcal{H}_{1}\times \mathcal{H}_{1}^*} , \dots,(\cdot, p^k)_{\mathcal{H}_{1}\times \mathcal{H}_{1}^*}  \bigr) (P_NX,p^i)_{\mathcal{H}_{1}\times \mathcal{H}_{1}^*}- \zeta (\cdot,P_N X)_{\HH_1}.\]
It remains to use that $(\cdot, X)_{\HH_1} = (\cdot, P_NX)_{\HH_1}$ in $L^2(\mu_N)$ and that $P_NX \to X$ in $\mathcal{H}_{1}$ to deduce \eqref{eqn82164321866213} following  the line of arguments below \eqref{eqn_CC_0}.
\end{proof}
\begin{proof}[Proof of Theorem \ref{thm_mosco_abstract}~\eqref{Mosco_I2}]
Let	$L^2(\mu_N) \ni \eta_N \rightharpoonup \eta \in L^2(\mu)$ and $X$ as in \eqref{eqn_XX}. Then, by applying Remark \ref{rem_weak_to_strong} in the finite dimensional setting, we can bound 
\begin{align}
	\EE_N (\eta_N,\eta_N) \ge  2
	\int_{\X} \eta_N \div_{\mu_N}(P_N X) \,\d\mu_N \,-\, \int_{\X} g_{N,\cdot} (P_NX,P_NX)\,\d\mu_N,
\end{align}
so that in light of Lemma \ref{Lemma_div_conv}
\begin{align}
	\liminf_{N\to\infty} \EE_N (\eta_N,\eta_N) \ge  2
	\int_{\X} \eta \div_{\mu}(X) \,\d\mu \,-\, \limsup_{N\to\infty}\int_{\X} g_{N,\cdot} (P_N X,P_N X)\,\d\mu_N.
\end{align}
Analogously to the proof of part \eqref{Mosco_I1}, we deduce from the assumptions  that 
\[
\limsup_{N\to\infty}\int_{\X} g_{N,\cdot} (P_N X,P_N X)\,\d\mu_N \,\le\, \int_{\X} g_{\cdot} ( X, X)\,\d\mu.
\]
Taking the supremum over $X$ results in \eqref{lower_bound_eqn}.
\end{proof}
	\subsection{Weak equals strong for non-degenerate metrics} To prove Proposition \ref{Prop_weak_is_strong}, we introduce some extra notation: We denote the Hermite polynomials by \begin{equation}\label{Eq62}
		H_a(x) =  \frac{(-1)^a}{a!} \exp(x^2/2) \frac{\d^a}{\d x^a} \exp(-x^2/2), \quad a\in \N_0.
	\end{equation}
Then, if $( p^k/\sigma_k)_{k\in \N}$ is the orthonormal basis of $\HH_{1}^*$, such that $ p^k  $ is orthonormal in $ \HH_{-1}^*$ constructed in \eqref{eqn_SVD}, we obtain by 
\begin{equation}\label{Eq60}
\mathbf{H_a} \,=\, \sqrt{\mathbf{a}!}\prod_{k\in \N} H_{a_k}((\cdot,  p^k/\sigma_k )_{{\HH}_1 \times \HH_1^*} ),
\end{equation}
indexed by finitely supported multiindices $\mathbf{a} \in \N_0^{\N}$,  an orthonormal basis of $L^2(\mu)$, see \cite[Proposition 1.1.1]{Nualart_related}. We first consider the flat case.
	\begin{lemma}\label{lemma_flat}
			Under  Assumption \ref{ass_Mosco_abstr}, if $
			g^\cdot ( p,p) =
			|p|_{\HH_{-1}^*}^2   $, then $(\EE,\D(\EE)) =(\EEw, \D(\EEw))$.
	\end{lemma}
\begin{proof}
	In light of Remark \ref{rem_weak_to_strong}, it suffices to show that $\EEw \ge \EE$. For this, we recall 
	for the Hermite polynomials from \eqref{Eq62} that $H_1'(x) =H_0(x)$ as well as 
	\begin{align} (a+1)H_{a+1}(x) = xH_a(x)- H_a'(x) \label{eqn_rec_relation}
	\end{align}
 and thus it follows by induction 
\begin{align}\begin{split}\label{Eq68-n}
		H_{a+1}'(x) \,=\,& \frac{1}{a+1} \bigl(H_a(x)\,+ \,x H_a '(x) \,-\, H_{a}''(x)\bigr)
		\\=\, & \frac{1}{a+1} \bigl(
		H_a(x)\,+ \,x H_{a-1} (x) \,-\, H_{a-1}'(x) \bigr) \,=\, H_a(x).
	\end{split}
\end{align}
Consequently, we have 
\begin{align}\label{Eq16_newn} 
	\diff \mathbf{H_a}\,=\,
	\sqrt{ \mathbf{a}!}\sum_{k\in \N: a_k\ne 0} H_{a_k-1}((\cdot,  p^k/\sigma_k )_{{\HH}_1 \times \HH_1^*} ) \biggl(
	\prod_{j\in \N\setminus \{ k\}} H_{a_j}((\cdot,  p^j/\sigma_j )_{{\HH}_1 \times \HH_1^*} )
	\biggr) { p^k}/{\sigma_k}. 
\end{align}
	 Approximating $\mathbf{H_a}$ by functions from $\D_0$, we find thereby that $\mathbf{H_a} \in \D(\EE)$ with 
	 \begin{align}
	 	\EE(\mathbf{H_a},\mathbf{H_b})\,=\, \int_{\X} g^h(\diff \mathbf{H_a},\diff \mathbf{H_b}) \,\d\mu\,=\, \delta_{\mathbf{ab}}  \sum_{k\in \N}
	 	a_k /\sigma_k ^2,
	 \end{align} 
 and consequently 
 \begin{align}\label{eqn73}\eta \,=\, 
 	\sum_{\mathbf{a}}\lambda_{\mathbf{a}} \mathbf{H_a} \in \D(\EE) \qquad \iff \qquad \EE(\eta,\eta )\,=\, 	\sum_{\mathbf{a}} \lambda_{\mathbf{a}}^2
 	\biggl(\sum_{k\in\N}
 	a_k /\sigma_k ^2\biggr)\,<\,\infty.
 \end{align}

Now let $\eta = 
\sum_{\mathbf{a}}\lambda_{\mathbf{a}} \mathbf{H_a}  \in \D(\EEw)$ for which our aim is to verify condition \eqref{eqn73}. 
To this end, we observe that 
\begin{align}\label{eqn63}X\mapsto 
\int_\X \eta \div_\mu(X) \,\d\mu 
\end{align}
extends to a continuous linear functional on $L^2(\mu;\HH_{-1})$ with norm $ \EEw^{1/2}(\eta,\eta ) $. 
 From the dual basis $(\vp_k)_{k\in \N}$   to $(p^k)_{k\in \N}\subset (\HH_{-1})^*$ in $\HH_{-1}$, cf.\ \eqref{eqn_SVD}, we may construct moreover the orthonormal basis
\begin{align}(
	\mathbf{H_b} \vp_k )_{\mathbf{b},k}
\end{align}
of $L^2(\mu; \HH_{-1})$. Based on the latter, we may measure 
\begin{align}\label{eqn275}
	 \EEw(\eta,\eta ) \,=\, 
	 \sum_{ \mathbf{b},k}  \biggl(\int_\X \eta \div_\mu( \mathbf{H_b} \vp_k ) \,\d\mu\biggr)^2  .
\end{align}
We evaluate the appearing $\mu$-divergence as
\begin{align}
	\div_{\mu}(\mathbf{H_b} \vp_k  ) \,=&\, \diff \mathbf{H_b}_\cdot \vp_k \,-\, \mathbf{H_b} (\cdot, \vp_k )_{\HH_1}\\
 =&\, 
\frac{ \sqrt{\mathbf{b!}}}{\sigma_k} H_{b_k-1} ((\cdot,  p^k/\sigma_k )_{{\HH}_1 \times \HH_1^*} ) \biggl(
\prod_{j\in \N\setminus \{ k\}} H_{b_j}((\cdot,  p^j/\sigma_j )_{{\HH}_1 \times \HH_1^*} )
\biggr)\\& \,- \, \frac{ \sqrt{\mathbf{b!}}}{\sigma_k}  (\cdot,  p^k/ \sigma_k)_{\HH_1\times \HH_1^*} 
\prod_{j\in \N }
 H_{b_j}((\cdot,  p^j/\sigma_j )_{{\HH}_1 \times \HH_1^*} ),
\end{align}
using \eqref{Eq16_newn} and that $(\cdot,\vp_k)_{{\HH}_1} = \sigma_k^{-1} (\cdot,  p^k/ \sigma_k)_{\HH_1\times \HH_1^*}$. By the recurrence relation \eqref{eqn_rec_relation} of Hermite polynomials and \eqref{Eq68-n}, the above means nothing but 
\begin{align}\div_{\mu}(\mathbf{H_b} \vp_k )  \,=\, - \frac{\sqrt{b_{k}+1}}{\sigma_k}\mathbf{H_{b+e_k}} .
\end{align}
Inserting this in \eqref{eqn275} results in 
\begin{align}	\sum_{\mathbf{a}} \lambda_{\mathbf{a}}^2\biggl(
	\sum_{k\in \N} a_k/\sigma_k^2
	\biggr)
	\,=\,
	\sum_{ k\in \N } \sigma_k^{-2} \biggl(\sum_{\mathbf{b}}(b_k+1) \lambda_{\mathbf{b+e_k}}^2\biggr) \,=\, \EEw(\eta,\eta ) .
\end{align}
By \eqref{eqn73} we conclude that $\eta\in \D(\EE)$ with $\EE(\eta,\eta)=\EEw(\eta,\eta)$, completing the proof.
\end{proof}
	\begin{proof}[Proof of Proposition \ref{Prop_weak_is_strong}]
		This time we operate under the assumption \eqref{comparable_to_flat} instead of the equality assumed in Lemma \ref{lemma_flat}. From Remark \ref{rem_weak_to_strong} we have already that $\EEw \le \EE$. Moreover, since passing from the flat to a uniformly equivalent metric results in equivalent  weak and strong  Dirichlet forms, we  have $\D(\EE) = \D(\EEw)$ by Lemma \ref{lemma_flat}. It remains to show that $\EEw(\eta,\eta) = \EE(\eta,\eta)$ for $\eta \in \D_0$. For this, we use that  the Gaussian integration by parts formula \eqref{GI_eqn} implies
		\begin{align}
			\label{eqn824}
			\frac{1}{2}\EEw (\eta,\eta ) \,\ge  \,
			\int_{\X} \diff \eta_\cdot X\, \d \mu - \frac{1}2 \int_{\X}  
			g_\cdot(X,X)\,
			\d\mu ,
		\end{align}
	for all $X$ of the form \eqref{eqn_XX}. Since $\diff \eta \in L^2(\mu; \HH_{-1}^*)$, its $h$-wise dual $ \grad \eta (h) = g^h(\cdot, \diff|_h \eta) $ lies in $ L^2(\mu; \HH_{-1} )$. Thus, there exist  $X_k \to \grad \eta $ in $L^2(\mu; \HH_{-1} )$ and by inserting these in \eqref{eqn824}, we find  
	\begin{align}&\int_{\X} \diff \eta_\cdot X_k\, \d \mu \,\to\, \int_{\X} \diff \eta_\cdot \grad\eta\, \d \mu \,=\, \EE(\eta,\eta),
		\\&
		 \int_{\X}  
		g_\cdot(X_k,X_k)\,
		\d\mu \,\to\,  \int_{\X}  
		g_\cdot(\grad \eta,\grad\eta)\,
		\d\mu \,=\, \EE(\eta,\eta),
	\end{align}
	using once more the equivalence \eqref{comparable_to_flat}.
	\end{proof}

	\section{Convergence to non-degenerate  stochastic thin-film type diffusions}
	\subsection{Assumption and results of this section}\label{S:convergence_explicit}
	Here,  we prove the convergence of Markov processes in finite dimensional state spaces to a  stochastic thin-film type diffusion in a simplified setting: Firstly, we remove the non-negativity constraint $h\ge 0$ in order to work in the Gaussian setting imposed in Assumption \ref{ass_Mosco_abstr}. Since the fixed mass is thereby arbitrary by translation, we may in the following demand that $\int h = 0$ instead  of $1$, i.e., we work over the linear space $\MM$ from \eqref{e:M_example_flat}. Secondly, to identify the weak and strong Dirichlet form by means of Proposition \ref{Prop_weak_is_strong}, we assume also the following.
	\begin{assumption}\label{assumption_mosco_concrete}
				Let $m\colon  C(\overline{\T})\to  C(\overline{\T})$ be continuous such that 
				\begin{equation}\label{eqn_m_unif}
					0<c\le m(h) \le C<\infty,
				\end{equation}
				for all $h\in C(\overline{\T})$.
	\end{assumption}
	Although we now restrict ourselves to non-degenerate mobility functions, we again allow them to be non-local, as in Sections \ref{S:DF_in_inf_dimensinos}--\ref{S:discrete}.
	The additional continuity assumption is to ensure suitable convergences of discretized metric tensors. 
	
	 The conclusions of Theorems \ref{thm_existence_deg_case}--\ref{thm_existence_deg__fin_case} clearly remain valid when passing to this simplified setting and transfer also to the pre-Dirichlet form \eqref{E2}, i.e.,
	\begin{equation}\label{eqn:spectral_DF_section}
			\check{\EE}_N (\eta,\eta ) \,=\,\int_{{\check{\MM}_N}} \check{g}^h ( \diff|_h \eta  ,\diff|_h \eta ) \, \d  \check{\mu}_N(h)
	\end{equation}
defined over the vector space $\check{\MM}_N$ spanned by the first $N$ cosine modes. We recall that the appearing metric is defined in \eqref{eqn_spectral_Galerkin_metric} and refer for details on adaptions to this case to Appendix \ref{App:spectral}. Then, since the nodal interpolation operators $Q_N$ from \eqref{Eq101} as well as the cosine projections $\Pi_{\le N}$, cf.\ \eqref{eqn73343}, converge strongly to the identity in $\dot{H}^1(\T)$,  Assumption \ref{ass_Mosco_abstr} is in both cases  satisfied. For the following we recall that the measures $\mu_N$ and $\mu$ are versions of $\mu_N^+$ and $\mu^+$ on the linear space $\MM$ and  therefore centered Gaussians just as $\check{\mu}_N$. 
	\begin{theorem}\label{thm_convergence}
		Under Assumption \ref{assumption_mosco_concrete}, the Dirichlet forms $\wt{\EE}_N$ and $\check{\EE}_N$ defined on $L^2(\mu_N)$ and $L^2(\check{\mu}_N)$, respectively, converge in Mosco sense to the Dirichlet form $\EE$ on $L^2(\mu)$.  Their associated diffusions converge in law on $C^{1/8-,1/2-}_{t,x}$ when started from equilibrium.
	\end{theorem}
Below, we verify the remaining  convergence criteria from Theorem \ref{thm_mosco_abstract} for the Gr\"un--Rumpf type  metric tensor, leading to the Mosco convergence for this approximation. We also comment on how the  weak convergence of the associated Markov processes can be deduced. Details on adaptions for  the spectral Galerkin approximations are again deferred to Appendix \ref{App:spectral}.  
\subsection{Verification of criteria  for the upper bound of $\wt{\EE}_N$}
Let $p\in \dot{H}^1(\T)$. Then we have the estimate 
\begin{align}
	\tilde{g}^{Q_N h} (p,p )\,\le\, 
	\int_{\T} m(Q_N h) | \partial_x \hat{\rho}^N *  {p}|^2
\end{align} 
from \eqref{E89}, where $\hat{\rho}^N$ is the piecewise linear approximate identity with $\hat{\rho}^N(0) = N$ and $\hat{\rho}^N(\pm 1/N) = 0$, and the convolution is with respect to the even extension of $p$. In particular, we have that $\partial_x \hat{\rho}^N *  {p} \to \partial_x p$ in $L^2(O)$. Since $Q_Nh \to h$ in $C(\overline{O})$ for  $\mu$-a.e.\ $h \in C(\overline{O})$ by Lemma \ref{lemma_reg} and Lemma \ref{Lemma_numerical_approx-n}, we deduce that $m(Q_Nh )\to m(h)$, $\mu$-a.e., using Assumption \ref{assumption_mosco_concrete}. Combining these observations yields the desired bound
\[
\limsup_{N\to\infty} 	\tilde{g}^{Q_N h} (p,p ) \,\le\, 	\int_{\T} m( h) | \partial_x p|^2\,=\, g^h(p,p),
\]
for $\mu$-a.e.\ $h$, i.e., the first part of Theorem \ref{thm_mosco_abstract}~\eqref{Mosco_I1}. 
Using the upper bound from \eqref{eqn_m_unif} we deduce by the same ingredients also that
\[
	\tilde{g}^{Q_N h} (p,p )\,\le\, 
\int_{\T} m(Q_N h) | \partial_x \hat{\rho}^N *  {p}|^2
\,\le\, C \|\partial_x p \|_{L^2(\T)}^2,
\]
which implies the required uniform integrability. 
\subsection{Verification of criteria for the lower bound  of $\wt{\EE}_N$} Let $\dot{h}\in \dot{H}^1(O)$ and, for now,  $h\in \dot{C} (\overline{O})$. By \eqref{E9}, there exists for each $N$ a ${p}_N \in \dot{H}^1(\T)$ such that 
\begin{align}&
\frac{1}{2}\tilde 	g_{Q_Nh}(Q_N \dot{h}, Q_N \dot{h} ) \\&\quad=\,\frac{1}{N}\biggl( \frac{{p}_N (0)Q_N \dot{h}(0)+ {p}_N (1)Q_N\dot{h}(1) }{2}\,+\, \sum_{i=1}^{N-1} {p}_N (x_i) Q_N\dot{h}(x_i) 
	\biggr) \,-\, \frac{1}{2}\int_{\T} m(Q_N h) |\partial_x  {p}_N  |^2. \label{eqn11}
\end{align}
Since the sequence of Radon measures 
\begin{equation}\label{eqn7}
{\frac{1}{N}\biggl(\frac{Q_N \dot {h}(0)\delta_{0}+Q_N\dot {h}(1)\delta_{1}}{2}+ \sum_{i=1}^{N-1} Q_N\dot {h}(x_i)\delta_{x_i}\biggr)}
\end{equation}
is uniformly bounded and 
\[
\frac{1}{N}
\biggl( \frac{q (0)Q_N\dot{h}(0)+ q (1)Q_N\dot{h}(1) }{2}\,+\, \sum_{i=1}^{N-1}q (x_i) Q_N\dot{h}(x_i) 
\biggr)\,\to\, \int_\T q \dot{h}
\]
for all $q\in C(\overline{\T})$, we deduce that \eqref{eqn7} tends to $\dot{h}$ weakly as measures and therefore in $(\dot{H}^{1}(O))^*$, as $N\to\infty$. Moreover, by elliptic regularity theory, $ {p}_N$ is
bounded in $\dot{H}^1(O)$, and, up to passing to a subsequence, we may assume that 
\[
\partial_x {p}_N \rightharpoonup \partial_x {p},\qquad N\to\infty
\]
for some ${p}\in \dot{H}^1(O)$
in $L^2(O)$.  Since $m(Q_Nh ) \to m( h)$ in $C(\overline{\T})$, we conclude that also $\sqrt{m}(Q_Nh)\partial_x {p}_N \rightharpoonup \sqrt{m}(h) \partial_x {p}$ in $L^2(O)$. Combining all these convergences, we may deduce from \eqref{eqn11} that
\begin{align}&
\limsup_{N\to\infty}\frac{1}{2}	\tilde g_{Q_Nh}(Q_N \dot{h}, Q_N \dot{h} ) \,=\,\int_\T {p} \dot{h}\,-\,\frac{1}{2}\liminf_{N\to\infty} \int_\T \bigl| \sqrt{m}(Q_Nh) \partial_x {p}_N \bigr|^2  \\&\quad \le\, \int_\T {p} \dot{h}\,-\, \frac{1}{2} \int_\T \bigl| \sqrt{m}(h) \partial_x {p} \bigr|^2 .
\end{align}
The latter is bounded from above by $\frac{1}{2}g_h(\dot{h},\dot{h})$ and since this can be done starting from any subsequence in $N$, it holds along the original one. Because $\mu$ is supported on $\dot{C}(\overline{\T})$, the first part of Theorem \ref{thm_mosco_abstract}~\eqref{Mosco_I2} follows. 

If $c$ is the lower bound from \eqref{eqn_m_unif}, the preceding observations yield moreover
\begin{align}&
\tilde g_{Q_Nh}(Q_N \dot{h}, Q_N \dot{h} ) \,= \,
\int_{\T} m(Q_N h) |\partial_x {p}_N  |^2 \\&\quad\lesssim_c \,
\biggl\|  {\frac{1}{N}\biggl(\frac{Q_N\dot {h}(0)\delta_{0}+Q_N\dot {h}(1)\delta_{1}}{2}+ \sum_{i=1}^{N-1} Q_N\dot {h}(x_i)\delta_{x_i}\biggr)} \biggr\|_{(\dot{H}^{1}(O))^*}^2\,
\lesssim \,  
\|\dot{h}\|_{C(\T)}^2, 
\end{align}
so that uniform integrability follows too.
\subsection{Convergence of the associated diffusions}
The above, in combination with Proposition \ref{Prop_weak_is_strong},  yields Mosco convergence of the Dirichlet forms $\wt{\EE}_N$ to $\EE$ on the varying Hilbert spaces $L^2(\mu_N)$ tending to $L^2(\mu)$. To complete the proof of Theorem \ref{thm_convergence}, it remains to argue that also the associated diffusions converge,  when started from equilibrium. To this end, we deduce from Theorem \ref{thm_existence_deg__fin_case} that their laws are tight on $C^{1/8- ,1/2-}_{t,x}$ and therefore converge along a subsequence. In the following, we characterize their limit point uniquely as the law of the diffusion associated to $(\EE,\D(\EE))$, so that the desired convergence follows  along the original sequence by a subsequence-subsequence argument. For this purpose, let $(T_{N,t})_{t>0}$ and $(T_t)_{t>0}$ be the respective associated semigroups. By Theorem \ref{Mosco_theorem_app}, we deduce from the Mosco convergence of $\wt{\EE}_N$ to $\EE$, that $T_{N,t} \eta_N $ converges to $T_{t} \eta$ for every $ \eta_N \in L^2(\mu_N)$  converging strongly to $\eta\in L^2(\mu)$ and $t>0$. In particular, for any  $\eta_0\in \CC_0$, we have that $T_{N,t} \eta_0 \to T_{t} \eta_0 $, strongly, as $N\to \infty$. In other words, there exists $\CC_0 \ni \zeta_k \to T_t \eta_0$ in $L^2(\mu)$ such that 
\[
\lim_{k\to \infty} \limsup_{N\to\infty} 
\|  \zeta_k - T_{N,t} \eta_0 \|_{L^2(\mu_N)} \,=\, 0.
\]
Consequently, for any  $\eta_1 \in \CC_0$, we deduce by boundedness that also 
\[
\lim_{k\to \infty} \limsup_{N\to\infty} 
\| \eta_1 \zeta_k -  \eta_1 T_{N,t} \eta_0 \|_{L^2(\mu_N)} \,=\, 0.
\]
Since  by the same reason  $\eta_1 \zeta_k \to \eta_1 T_t\eta_0$ in $L^2(\mu)$, we conclude that $ \eta_1 T_{N,t} \eta_0 \in L^2(\mu_N)$ converges strongly to $\eta_1 T_t\eta_0\in L^2(\mu )$. An iteration of this argument yields that 
\begin{equation}\label{conv_of_sems}
L^2(\mu_N) \,\ni\, \eta_k T_{N,t_{k-1}} \bigl(
\eta_{k-1} T_{N, t_{k-2}} ( ... T_{N,t_0} \eta_0)
\bigr) \,\to\,  \eta_k T_{t_{k-1}} \bigl(
\eta_{k-1} T_{ t_{k-2}} ( ... T_{t_0} \eta_0)
\bigr) \,\in\,
L^2(\mu) , 
\end{equation}
strongly, for any $ \eta_{0},\dots, \eta_k \in \CC_0$ and $t_0,\dots , t_{k-1} >0$. 

To connect the above to the associated diffusions, we introduce time instances $ s_l = \sum_{i=1}^l  t_{k-i}$, for which we may write by the tower property of conditional expectation
\begin{align}&
E_\mu 
\biggl[\prod_{l=0}^{k} \eta_{k-l} (h_{s_l}) \biggr] \,=\, E_{\mu } \biggl[
\eta_k (h_{0})
E_{\mu } 
\biggl[\prod_{l=1}^{k} \eta_{k-l} (h_{s_l}) 
\bigg| \mathfrak{F}_{0}
\biggr]
\biggr]\,=\, E_{\mu } \biggl[
\eta_k (h_{0})E_{\mu } \biggl[
 \dots E_{\mu }
\biggl[ \eta_{0} (h_{s_k}) 
\bigg| \mathfrak{F}_{s_{k-1}}
\biggr]\dots \bigg| \mathfrak{F}_{0}
\biggr]
\biggr]\\&\quad=\, 
E_{\mu } \biggl[
\eta_k (h_{0})E_{\mu } \biggl[
\dots 
T_{t_0} \eta_0(h_{s_{k-1}})
\dots \bigg| \mathfrak{F}_{0}
\biggr]
\biggr]\,=\, 
E_{\mu} \bigl[
\eta_k T_{t_{k-1}} \bigl(
\eta_{k-1} T_{ t_{k-2}} ( ... T_{t_0} \eta_0)
\bigr)(h_0)
\bigr].
\end{align}
But this is nothing but the integral of the right-hand side of \eqref{conv_of_sems} against the equilibrium measure, and by repeating the argument for the left, we conclude that
\begin{align}\label{Eq1400000}
	E_{N,\mu_N}\biggl[\prod_{l=0}^{k} \eta_{k-l} (h_{s_l}) \biggr] \,\to \,
	E_\mu 
	\biggl[\prod_{l=0}^{k} \eta_{k-l} (h_{s_l}) \biggr],
\end{align}
as $N\to \infty$. Since  the marginals of the limiting distribution coincide thereby with the finite-dimensional marginals of the diffusion associated to $(\EE,\D(\EE))$, this yields the claim.

	\section{Non-degenerate  stochastic thin-film type diffusions  as energy solutions}
	\subsection{Assumptions and results of this section} It is the aim of this section to give an alternative description of the diffusion process associated to the Dirichlet form $(\EE,\D(\EE))$ in terms of energy solutions to a stochastic thin-film type SPDE 
		\begin{equation}\label{EqSTFE}
		\partial_t u\,=\, -\partial_x(m(u)\partial_x^3 u ) \,+\,\sqrt{2 } \partial_x (m^{1/2}(u)\xi) .
	\end{equation}
	To this end, we work throughout in the non-degenerate situation investigated in the previous section, i.e., on the linear space $\MM$ equipped with the Gaussian measure $\mu$ and under Assumption \ref{assumption_mosco_concrete}. Moreover, we impose additionally that
	\begin{equation}\label{Eq220}
		m\in C^1_b(\R),
	\end{equation}
	and in particular that $m$ is a local function of $h$. 
	 To define energy solutions, we introduce also the $p$-variation norm 
	\[
	\|f\|_{\pv, T}\,=\,\biggl( \sup_{0 =t_0 <\dots<t_k =T} \sum_{l=1}^k |f(t_l) -f(t_{l-1})|^p\biggr)^{1/p},
	\]of a function $f$,
	which bounds in particular the supremum in time whenever $f(0)=0$. We let $\D(\EE)^*$ be the dual space of $(\D(\EE),(\cdot,\cdot)_{L^2(\mu)} +\EE)$ under the pairing in $L^2(\mu)$, see Subsection \ref{ss:dual_space_info} for more information. 
	
	\begin{definition}\label{defi_energy_sol_iii'}
		Let $\frac{\d \nu}{ \d \mu}\in L^2(\mu)$ be a probability density. Then, a $(\dot{H}^1(O))^*$-continuous stochastic process $h$  is called an {energy solution} to \eqref{EqSTFE}  with initial distribution  
		$\nu$ if $h(0)\sim \nu$ and the following  conditions are satisfied.
		\begin{enumerate}[(i)]
			\item \label{Item_inc_New} \emph{(Incompressibility)}
			We have the estimate 
			$
			\E [|\eta(h_t)|] \,\lesssim \, \|\eta \|_{L^2(\mu)}
			$
			for all $\eta \in L^2(\mu)$ and $t\ge 0$.
			\item \label{Item_energy_est_New} \emph{(Energy estimate)} For any $p>2$ we have
			\begin{align}\label{Eq142_New}
				\E\biggl[\biggl\|
				\int_0^{\cdot} \eta (h_s)\,\d s
				\biggr\|_{\pv, T}
				\biggr] \,\lesssim_p\, (T^{1/2}+T)\|\eta \|_{\D(\EE)^*},
			\end{align}
			for all $\eta \in L^2(\mu)$.
			\item \label{Item_MP_New} \emph{(Martingale problem)}
			For all $p \in \dot{H}^3_N(\T)$
			\begin{align}\begin{split}\label{Eq158_New}
					&
					\int_\T  h_t  p \,-\, \int_\T h_0 p \\&\quad-\, \lim_{N\to \infty} \int_0^t  \int_{\T} m( \Pi_{\le N} h_s)\px^3\Pi_{\le N}h_s 
					 \px\Pi_{\le N} p\,+\, \frac{\partial_x (D_N(2\pi  \cdot ))}{4}   m'(\Pi_{\le N}  h_s)  \partial_x \Pi_{\le N }p  \,\d x \,\d s   
				\end{split}
			\end{align}
			is a martingale with respect to the filtration generated by $h$ and has  quadratic variation
			\[2
			\int_0^t \int_\T 
			m(h_s) |\px p|^2 \,\d x \,\d s 
			.
			\]
		\end{enumerate}
	\end{definition} 
We remark that the conditions  \eqref{Item_inc_New} and  \eqref{Item_energy_est_New} guarantee that the limit in \eqref{Eq158_New} is
attained in $L^1(\Omega; C([0,T]))$, cf.\ the discussion below Lemma \ref{Lemma_conv_gens}. The following is the main result of this section.
		\begin{theorem}\label{thm_characterization}
			Let Assumption \ref{assumption_mosco_concrete} and \eqref{Eq220} hold, and $\frac{\d \nu}{ \d \mu}\in L^2(\mu)$ be a probability density. Then, the diffusion associated to $(\EE,\D(\EE))$ is an energy solution when started from $\nu$. Conversely, any energy solution started from $\nu$ coincides in law with said diffusion. 
	\end{theorem}
	This is proved with the help of the auxiliary Dirichlet forms
	\begin{align}\label{eqn_spectral_whole_space}
		\check{\EE}_N(\eta, \eta) \,=\, \int_{\MM} 
		\int_O m(\Pi_{\le N} h)\bigl| \partial_x \Pi_{\le N}\diff|_h\eta \bigr|^2 \,\d x
		\,\d\mu(h)\,+\, \int_{\MM} 
		\int_O \bigl| \partial_x  \Pi_{> N} \diff|_h\eta \bigr|^2 \,\d x
		\,\d\mu(h),
	\end{align}
where $\Pi_{>N} p = p-\Pi_{\le N }p $. These are designed in a way that allows for a computation of their generators $\check{\L}_N$ in contrast to the generator of $\EE$, cf.\ Subsection \ref{SS:generator_is_diff}. We also remark that the above is a slight abuse of notation, since we defined $\check{\EE}_N$ originally as the form \eqref{E2} on $L^2(\check{\mu}_N)$, but that the two forms agree on observables which only depend on the first $N$ cosine modes of $h$ by \eqref{eqn_explicit_formula_for_metric_galerkin}. In particular, since the forms $\check{\EE}_N$ and ${\EE}$ are both comparable to the flat form
\begin{align}
	\EEf(\eta, \eta )\,=\,  \int_{\MM} 
		\int_O \bigl| \partial_x \diff|_h\eta \bigr|^2 \,\d x
		\,\d\mu(h),
\end{align}
we can compare their generators in $\D(\EE)^* = \D(\EEw)^* = \D(\EEf)^*$. We show moreover in Subsection \ref{approx_of_generators} that $\check{\L}_N\eta \to\L\eta$ in $\D(\EE)^*$, for all $\eta \in \D(\EE)$.  This, together with an explicit identification of $\check{\L}_N$, readily allows us to show that the diffusion associated to $(\EE,\D(\EE))$ constitutes an energy solution to \eqref{EqSTFE} in Subsection \ref{SS:diff_is_es}.

The strategy for the reverse direction is more involved. In essence, one has to verify that any energy solution to \eqref{EqSTFE} solves a generalized martingale problem similar to Definition \ref{defi_energy_sol_iii'}~\eqref{Item_MP_New}, but with cylinder functions as observables. This requires in particular the justification of It\^o's formula, despite the last term in \eqref{Eq158_New} not being of bounded variation. For this we follow in Subsection \ref{more_gen_MP} the ideas from \cite[Section 4.4]{GubPer_generator} and identify the limiting integral using Young integration and regularity of $h$. This allows us to characterize the marginals of an energy solution as the right-hand side of \eqref{conv_of_sems}, i.e., the marginals of the diffusion associated to $(\EE, \D(\EE))$, in Subsection \ref{ss:uniqueness}.
	\subsection{Description of the dual space $\D(\EE)^*$}
	\label{ss:dual_space_info}
	We start by describing the dual space $\D(\EE)^*$, for which we take $(\vp_k)_{k\in \N}$, $(p^k)_{k\in \N}$ and $(\sigma_k)_{k\in \N}$ as constructed in and above \eqref{eqn_SVD}. In this particular case, we can calculate that 
	\begin{align}\label{eqn112}
		\vp_k(x) \,=\,\sqrt{2}\pi k \cos(\pi k x) ,\qquad
		p^k \,=\, \frac{\sqrt{2}}{\pi k} \cos(\pi k x) ,\qquad 
		\sigma_k \,=\, (\pi k )^{-2}.
	\end{align}
	Then, if we let $(\mathbf{H_a})_\mathbf{a}$ be the orthonormal basis of $L^2(\mu)$ defined in \eqref{Eq60}, it follows from \eqref{eqn73} that $\D(\EE)$ is complete when equipped with
	\begin{equation}\label{Eq140}
		(\eta,\eta)_{L^2(\mu)} \,+\, 
		\EEf(\eta,\eta) \,=\, \sum_{\mathbf{a}}
		\lambda_\mathbf{a}^2\biggl(1+\sum_{k\in \N}
	a_k/\sigma_k^2 \biggr)
		,\qquad  \eta\,=\,  \sum_{\mathbf{a}}\lambda_{\mathbf{a}} \mathbf{H_a}.
	\end{equation}
	As the embedding $\D(\EE)\hookrightarrow L^2(\mu)$ is continuous and dense, we have $L^2(\mu) \hookrightarrow \D(\EE)^*$. The latter space can  be described as the relaxed eigenfunction expansions
	\begin{equation}\label{Eq281}
		{\D}(\EE)^*\,=\, 
		\biggl\{
		\zeta \,=\, \sum_{\mathbf{a}}\nu_{\mathbf{a}} \mathbf{H_a}\bigg|
		\sum_{\mathbf{a}}\nu_\mathbf{a}^2\biggl(1+\sum_{k\in \N}
		(\pi k )^4 a_k \biggr)^{-1}<\infty
		\biggr\}
	\end{equation}
	if we extend the $L^2(\mu)$-inner product by
	\[
	\llangle \eta, \zeta \rrangle\,=\, \sum_{\mathbf{a}} \lambda_\mathbf{a} \nu_\mathbf{a}.
	\]
	With this identification we have also that
		\begin{equation}\label{Eq1433}
	 \bigl\{\eta \in \D(\EE)\,\big|\,
		(\eta, \mathbf{1}_\MM)_{L^2(\mu)}  = 0
		\bigr\}^* \,\cong \,  \bigl\{\zeta \in \D(\EE)^*\,\big|\,
		\llangle \mathbf{1}_\MM,\zeta \rrangle   = 0
		\bigr\}.
	\end{equation}
Due to the spectral gap of $\EEf$, cf.\ \eqref{Eq140}, the mean-zero subspace of $\D(\EE)$ appearing on the left-hand side of \eqref{Eq1433} can be equivalently equipped with $\EEf$ itself, and by equivalence also with $\EE$ or $\check{\EE}_N$, for any $N\in \N$. In particular, the weak forms of the respective  generators, e.g.,  
\begin{align}
	\L_{\text{flat}} \eta \,=\, -\EEf(\eta , \cdot ) \,\in \, \D(\EE)^*,\qquad \eta \in \D(\EE),
\end{align}
 induce isomorphisms 
\begin{equation}\label{Eq143}
\L_{\text{flat}} \colon  \bigl\{\eta \in \D(\EE)\,\big|\,
(\eta, \mathbf{1}_\MM)_{L^2(\mu)}  =  0
\bigr\}\,\to \,  \bigl\{\zeta \in \D(\EE)^*\,\big|\,
\llangle \mathbf{1}_\MM,\zeta \rrangle   = 0
\bigr\},
\end{equation}
and analogously for $\L$ and $\check{\L}_N$.

	\subsection{Approximate generators and convergence}\label{approx_of_generators}
	We first start with a formula for the generator $\check{\L}_N$, for which we recall the space \eqref{eqn_H3N}, 
	and that 
	\[
		\dot{H}^3_N(\T) \hookrightarrow \dot{H}^1(O) \hookrightarrow 
		(\dot{H}^1(O))^*
	\]
	is a Gelfand triple when $	\dot{H}^3_N(\T) $ is equipped with the inner product 
	\[
	( p,q)_{\dot{H}^3_N(\T)} \,=\, \int_{\T} \partial_x^3 p \partial_x^3 q.
	\] 
	In particular, since $\Pi_{\le N}$ is orthogonal in $\dot{H}^1(O)$ and $
	(\dot{H}^1(O))^* $, it is also  an orthogonal projection  in $	\dot{H}^3_N(\T) $. 
	We also recall the Dirichlet kernel $D_N$ defined in \eqref{eqn_def_Dirichlet kernel}.

	  \begin{lemma}\label{Lemma_generator_approx_problem} Let
		\begin{equation}\label{Eq11_New}
			\eta(h) \,=\,  
			\Phi\biggl(\int_{\T}  h q^1,\dots , \int_{\T}  h q^k \biggr),
		\end{equation}
		for $q^i\in H_N^3(\T)$  and $\Phi\in C_b^2(\R^k)$. Then $\eta \in  \D(\check{\L}_{N})$ with
		\begin{align} \begin{split}\label{Eq74}
				\check{\L}_{N} \eta \,=\,&	 \sum_{i=1}^k
				\partial_i	\Phi\biggl(\int_{\T}  \cdot  q^1,\dots , \int_{\T}  \cdot  q^k \biggr)
				\bigl( \cdot, 
				\partial_x^2 \Pi_{>N}q^i
				\bigr)_{\dot{H}^1(O)  }
				\\&  +\,\sum_{i,j=1}^k \partial_{ij} \Phi\biggl(\int_{\T}  \cdot  q^1,\dots , \int_{\T}  \cdot  q^k \biggr) \int_\T  \partial_x \Pi_{>N} q^i \partial_x \Pi_{>N}  q^j
				\\& +\,
				\sum_{i=1}^k
				\partial_i\Phi\biggl(\int_{\T}  \cdot  q^1,\dots , \int_{\T}  \cdot  q^k \biggr)
				\bigl( \cdot,  \Pi_{\le N} \partial_x  (m( \Pi_{\le N} \cdot)  \partial_x  \Pi_{\le N} q^i)
				\bigr)_{\dot{H}^1(O)  }
				\\&
				+\, \frac{1}{4}
				\sum_{i=1}^k
				\partial_i\Phi\biggl(\int_{\T}  \cdot  q^1,\dots , \int_{\T}  \cdot  q^k \biggr) 
				\int_O \partial_x ( D_N(2\pi  \cdot ))    m'(\Pi_{\le N}  \cdot)  \partial_x \Pi_{\le N } q^i 
				\\& +\, 
				\sum_{i,j=1}^k 
				\partial_{ij} 
				\Phi\biggl(\int_{\T}  \cdot  q^1,\dots , \int_{\T}  \cdot  q^k \biggr)  \int_\T  
				m(\Pi_{\le N} \cdot)  \partial_x \Pi_{\le N} q^i \partial_x \Pi_{\le  N}  q^j .
			\end{split}
		\end{align}
	\end{lemma}
	 The proof of Lemma \ref{Lemma_generator_approx_problem} amounts to a direct justification of the integration by parts formula 
	\begin{align}
		(\check{\L}_N\eta, \zeta )_{L^2(\mu)} \,=\, -\check{\EE}_N (\eta, \zeta), \qquad \zeta \in \D_0,
	\end{align}
	for $\eta$ as in \eqref{Eq11_New} 
	and is deferred to Appendix \ref{App:proof_IBP}.
	
	We also remark that, by the equivalence to the flat setting and the construction of the Dirichlet form closure from Subsection \ref{SS:closable},  we know that $\D(\EE) = \D(\diff)$ is the domain of the smallest closed extension of the operator $\diff \colon \D_0 \subset L^2(\mu) \to L^2(\mu; \dot{H}^1(\T))$. In particular, the differential of each $\eta \in \D(\EE)$ may be realized as some $\diff \eta \in L^2(\mu; \dot{H}^1(\T))$ and we can write the Dirichlet forms also for such $\eta $ as
	\begin{align}
		\EE (\eta ,\eta )\,=\, \int_{\MM} \int_O
		m(h)\bigl| \partial_x \diff|_h\eta \bigr|^2 \,\d x
		\,\d\mu(h)
	\end{align}  and \eqref{eqn_spectral_whole_space}. 
	\begin{lemma}\label{Lemma_conv_gens}
		We have $\check{\L}_N \eta \to \L \eta $ in $\D(\EE)^*$ for all $\eta \in \D(\EE)$.
	\end{lemma}
\begin{proof}
Let $\eta \in \D(\EE)$ and $\diff \eta$  the realization of its differential in $L^2(\mu; \dot{H}^1(\T))$. We need to show that 
\[\sup_{\substack{\zeta \in \D_0 \\ \|\zeta \|_{\D(\EE)} \le 1 }}
\check{\EE}_N  ( \eta ,\zeta) - \EE(\eta ,\zeta)  \,\to\, 0,
\]
as $N\to\infty$. For this we write
\begin{align}&
	\check{\EE}_N  ( \eta ,\zeta) - \EE(\eta ,\zeta) \\&\quad=\,  \int_{\MM}  
 \Bigl( \partial_x \bigl(
	m(h)  \partial_x \diff|_h\eta\bigr) \,-\,  	\Pi_{\le N}  \partial_x \bigl(	m(\Pi_{\le N}h) \partial_x \Pi_{\le N} \diff|_h\eta \bigr) ,    \diff|_h \zeta \Bigr)_{(\dot{H}^1(O ))^* \times \dot{H}^1(O)} 
	\,\d\mu(h)\\&\qquad -\, 
	\int_{\MM}  \Bigl( 	\Pi_{> N}  \partial_x^2 \Pi_{> N} \diff|_h\eta  ,    \diff|_h \zeta \Bigr)_{(\dot{H}^1(O ))^* \times \dot{H}^1(O)} 
	\,\d\mu(h)
	\\&\quad \lesssim \, \biggl(
	\int_{\MM}  
	\Bigl\| \partial_x \bigl(
	m(h)  \partial_x \diff|_h\eta\bigr) \,-\,  	\Pi_{\le N}  \partial_x \bigl(	m(\Pi_{\le N}h) \partial_x \Pi_{\le N} \diff|_h\eta \bigr) \Bigr\|_{(\dot{H}^1(O ))^* }^2 
	\,\d\mu(h)\biggr)^{1/2} \\&\qquad+\, \biggl( 
	\int_{\MM}  
	\Bigl\| 	\Pi_{> N}  \partial_x^2 \Pi_{> N} \diff|_h\eta  \Bigr\|_{(\dot{H}^1(O ))^* }^2 
	\,\d\mu(h)\biggr)^{1/2}\,=\, I_1(N) \,+\, I_2(N),
\end{align}
for $\|\zeta \|_{\D(\EE)}\le 1$,
 by the Cauchy--Schwarz inequality. For $I_2$, we use that the operators $ \Pi_{>N } \partial_x^2 \Pi_{>N}\colon \dot{H}^1(\T) \to (\dot{H}^1(\T) )^* $ are uniformly in $N$ bounded and converge pointwise to $0$, which implies that  $I_2(N) \to 0$ by dominated convergence. Also the operators $ 	\Pi_{\le N}  \partial_x (	m(\Pi_{\le N}h) \partial_x \Pi_{\le N} \cdot )\colon \dot{H}^1(\T) \to (\dot{H}^1(\T) )^* $ are uniformly bounded and
 we have
 \begin{align}&\|
 	\Pi_{\le N}  \partial_x (	m(\Pi_{\le N}h) \partial_x \Pi_{\le N} p )
 	-   \partial_x  (
 	m(h)  \partial_x p  )\
 	\|_{(\dot{H}^1(O))^*}
 	\\&\quad \le 
 	\|
 		m(\Pi_{\le N}h) \partial_x \Pi_{\le N} p \,-\, 	m(h)  \partial_x p 
 	\|_{L^2(O)}\,+\, \| \Pi_{> N}\partial_x  (
 	m(h)  \partial_x p  ) \|_{(\dot{H}^1(O))^*}
 	\\&\quad \le 
 	\bigl\|
 \bigl(	m(\Pi_{\le N}h) \,-\, 	m(h) \bigr) \partial_x p 
 	\bigr\|_{L^2(O)}\,+\, \bigl\| 	m(\Pi_{\le N}h) \bigl(\partial_x \Pi_{\le N} p \,-\, 	 \partial_x p \bigr) \bigr\|_{L^2(O)}
 	\\&\qquad +\, \bigl\| \Pi_{> N}\partial_x  (
 	m(h)  \partial_x p  ) \bigr\|_{(\dot{H}^1(O))^*}\,\to\, 0,
 \end{align}
 for all $p\in \dot{H}^1(O)$ and $\mu$-a.e.\ $h\in C(\overline{\T})$. Therefore, also  $I_1(N) \to 0$ by the dominated convergence theorem, and the claim follows.	
\end{proof}
	The previous lemma is important in the interpretation of condition \eqref{Item_MP_New} from Definition \ref{defi_energy_sol_iii'}.
	Indeed, Lemma \ref{Lemma_generator_approx_problem} precisely describes $\check{\L}_{N}\eta $ for $\eta$ of the form \eqref{Eq11_New} as \eqref{Eq74}. Moreover, since $\check{\L}_N$ is closed, it generalizes straightforwardly to $\eta$ for which $\Phi$ and its derivatives have  polynomial growth. Now we can use that
	\begin{align}\label{Eq161}
		\partial_i\Phi (\dots) 	\bigl( \cdot, 
		\partial_x^2 \Pi_{>N}q^i
		\bigr)_{\dot{H}^1(O)  }  \,&\to\, 0,\\ \label{Eq162} 
		\partial_{ij} \Phi (\dots ) \int_\T  \partial_x \Pi_{>N} q^i \partial_x \Pi_{>N}  q^j\, &\to\, 0,\\ \label{Eq163} 
		\partial_{ij}\Phi (\dots)\int_\T  
		m(\Pi_{\le N} \cdot)  \partial_x \Pi_{\le N} q^i \partial_x \Pi_{\le  N}  q^j
		\,&
		\to\, 
		\partial_{ij}\Phi (\dots) \int_\T  
		m( \cdot)  \partial_x  q^i \partial_x  q^j, 
	\end{align} 
	in $L^2(\mu)$, where  the last limit follows by dominated convergence. By Lemma \ref{Lemma_conv_gens}
	\begin{equation}\label{Eq307}\lim_{N\to\infty}
		\sum_{i=1}^k 
		\partial_i\Phi (\dots) \biggl( \bigl( \cdot,  \Pi_{\le N} \partial_x  (m( \Pi_{\le N} \cdot)  \partial_x  \Pi_{\le N} q^i)
		\bigr)_{\dot{H}^1(O)  } 
		+\, \frac{1}{4}
		\int_O  \partial_x (D_N(2\pi  \cdot ))   m'(\Pi_{\le N} \cdot)  \partial_x \Pi_{\le N } q^i \biggr)
	\end{equation}
	exists {therefore} in $\D(\EE)^*$ and hence also the limit in \eqref{Eq158_New} in $L^1(\Omega; C([0,T]))$ due to Definition \ref{defi_energy_sol_iii'}~\eqref{Item_energy_est_New}.

	\subsection{Diffusions are energy solutions}\label{SS:diff_is_es}
	Here we prove that the diffusion process associated to $(\EE,\D(\EE))$ constitutes an energy solution to \eqref{EqSTFE}. In addition to the generator convergence proved above, we rely on the Fukushima and Lyons--Zheng decompositions from Dirichlet form theory as well as L\'epingle's inequality, which asserts that for all $p>2$ and $q \in [1,\infty)$
	\begin{align}\label{lepingle}
		\E^{1/q} [ \| M\|_{{\pv,T}}^q ] \eqsim_{(p,q)} \E^{1/q} \Bigl[\sup_{t\le T} |M_t|^q\Bigr] ,
	\end{align}  
	for martingales with $M_0 = 0$, cf.\ \cite{lepingle}. 
	\begin{proposition}\label{Prop_DF_yields_ES}
		Let $\frac{\d \nu}{ \d \mu}\in L^2(\mu)$ be a probability density and   $((h_t)_{t\ge 0}, (P_z)_{z\in \MM})$  be the diffusion associated to
		$(\EE,\D(\EE))$. Then, under $P_\nu$, $h$ is an energy solution to \eqref{EqSTFE}  starting from $\nu$ in the sense of Definition \ref{defi_energy_sol_iii'}.
	\end{proposition} 
	\begin{proof}
		We set $\eta_0 = \frac{\d \nu}{ \d \mu}$ and estimate for $\zeta \in L^2(\mu)$
		\[
		E_\nu[|\zeta(h_t)|] \,=\, \int_{\MM} E_{\cdot}[|\zeta (h_t)|] \eta_0 \,\d \mu 
		\,\le\, \|T_t|\zeta|\|_{L^2(\mu)} \|\eta_0\|_{L^2(\mu)}
		\]
		using H\"older's inequality. 
		Since $T_t$ is contractive, we have already verified part \eqref{Item_inc_New} of Definition \ref{defi_energy_sol_iii'}. 
		
		For $\eta\in \D(\EE)$ we insert the decomposition from Theorem \ref{LZ} into the one of Theorem \ref{thm_fukushima} to obtain that 
		\begin{equation}\label{Eq144}
			N^{[\eta]}_t \,=\, -\frac{1}{2}M_t^{[\eta]} \,-\,  \frac{1}{2}( \overline{M}_{T}^{[\eta]} - \overline{M}_{T-t}^{[\eta]}  ) 
		\end{equation}
		for the zero-energy part of Fukushima's decomposition $
		\eta(h_t) - \eta(h_0) \,=\, N_t^{[\eta]}\,+\, M_t^{[\eta]}$. By \eqref{SFO}, the  forward martingale $M^{[\eta]}_t$ and backward martingale $\overline{M}_{t}^{[\eta]}$ have quadratic variation processes 
		\begin{align}&
	2	\int_0^t g^{h_s}(  \diff \eta(h_s) ,\diff \eta  (h_s) ) \, \d s
		\quad \text{and}\quad
2	\int_0^t
	g^{h_{T-s}}(  \diff \eta(h_{T-s}) ,\diff \eta  (h_{T-s}) )  \, \d s,
\end{align}
respectively. By means of H\"older's inequality, L\'epingle's inequality \eqref{lepingle}, and the Burkholder--Davis--Gundy inequality, 
 we therefore deduce  that
\begin{align}\begin{split}
\label{Eq145}&
E_\nu\Bigl[\| N_t^{[\eta]}\|_{\pv,T}\Bigr] \,\le\, E_{\mu}^{1/2}\biggl[\frac{1}{2}\| M_t^{[\eta]}\|_{\pv,T}^2 \,+\, \frac{1}{2} \| \overline{M}_{T}^{[\eta]} - \overline{M}_{T-t}^{[\eta]}  \|_{\pv,T}^2 \biggr]  \|\eta_0\|_{L^2(\mu)}\\& \quad \lesssim_p
\biggl(
\int_0^T
\int_{\MM} 
g^{h}(  \diff \eta(h) ,\diff \eta  (h) ) \,\d\mu(h)
\,\d t
\biggr)^{1/2} \|\eta_0\|_{L^2(\mu)} \,\lesssim_{\nu} \, T^{1/2} \EE(\eta,\eta)^{1/2},
\end{split}
\end{align}
for $p>2$. As follows from the comments below \eqref{Eq1433},  we can write any $\zeta \in L^2(\mu)$ as $\zeta = \L \eta + \lambda$ for some $\lambda\in \R$ and $\eta\in \D(\L)$, such that $\|\zeta\|_{\D(\EE)^*} \eqsim |\lambda |+ \|\eta\|_{\D(\EE)}$. With this choice we have 
\[
\int_0^t \zeta(h_s)\,\d s \,=\, N_t^{[\eta]} \,+\,\lambda t,
\]
 and using $\|\lambda t\|_{\pv,T} =|\lambda |T$,  Definition \ref{defi_energy_sol_iii'}~\eqref{Item_energy_est_New} follows by \eqref{Eq145}.

Using once more that $\L$ is an isomorphism between the mean-free observables in $\D(\EE)$ and $\D(\EE)^*$, we deduce from Lemma \ref{Lemma_conv_gens} that $\eta_N = \L^{-1}\check{\L}_{N} \eta\to \L^{-1}{\L} \eta = \eta$ in $\D(\EE)$ for mean-free $F$. This implies that $N^{[\eta_N]} \to N^{[\eta]}$ in $L^1(P_\nu; C([0,T]))$ by \eqref{Eq145}, i.e.,
\begin{align}
N^{[\eta]}_t\,=\,\lim_{N\to\infty} \int_0^t \check{\L}_{N} \eta(h_s) \,\d s ,
\end{align}
in said space. 
By Lemma \ref{Lemma_generator_approx_problem} and its generalization to polynomial $\Phi$, see the comments below Lemma \ref{Lemma_conv_gens}, we 
can represent the latter for $\eta(h) =\int_O h p$  and 
 $p\in \dot{H}^3_N(\T)$   furthermore as
\begin{align}&
\lim_{N\to\infty}
\int_0^t 
	\bigl( h_s, 
\partial_x^2 \Pi_{>N}p
\bigr)_{\dot{H}^1(O)  } 
\\&\,+\, \lim_{N\to\infty}\int_0^t \int_{\T} m( \Pi_{\le N} h_s)\px^3\Pi_{\le N}h_s 
\px\Pi_{\le N} p\,+\, \frac{\partial_x (D_N(2\pi  \cdot ))}{4}    m'(\Pi_{\le N}  h_s)  \partial_x \Pi_{\le N }p .
\end{align} Based on the convergence \eqref{Eq161}, we find that
\begin{align}&
\lim_{N\to\infty} E_\nu \biggl[\sup_{t\le T} \biggl| 
\int_0^t 
\bigl( h_s, 
\partial_x^2 \Pi_{>N}p
\bigr)_{\dot{H}^1(O)  } \,\d s\biggr|  \biggr] \,\le\, \|\eta_0\|_{L^2(\mu)} 
\lim_{N\to\infty}  \int_0^T  E_\mu^{1/2}\biggl[ \biggl| 
\bigl( h_s, 
\partial_x^2 \Pi_{>N}p
\bigr)_{\dot{H}^1(O)  } \biggr|^2  \biggr]\,\d s
\\&\quad \lesssim_T\,  \|\eta_0\|_{L^2(\mu)} 
\lim_{N\to\infty}\bigl\| (\cdot , \partial_x^2 \Pi_{>N}p  )_{\dot{H}^1(O)  }\bigr\|_{L^2(\mu)} \,=\, 0,
\end{align}
so that 
we arrive at Definition \ref{defi_energy_sol_iii'}~\eqref{Item_MP_New} by inserting  the above in Fukushima's decomposition of $\eta(h)$, cf.\ Theorem \ref{thm_fukushima}. 
\end{proof}
\subsection{Energy solutions solve a generalized martingale problem for $\L$}\label{more_gen_MP}
Next, we execute the most delicate part of the proof of Theorem \ref{thm_characterization}, namely that energy solutions  satisfy a generalized martingale problem for cylinder functions, analogous to  Definition \ref{defi_energy_sol_iii'}~\eqref{Item_MP_New}. 
\begin{proposition}\label{Prop_ES_to_MP}Let $\frac{\d \nu}{ \d \mu}\in L^2(\mu)$ be a probability density, let  $h$ be an energy solution  to \eqref{EqSTFE}  starting from $\nu$ in the sense of Definition \ref{defi_energy_sol_iii'} and let $\eta$ be as in \eqref{Eq11_New} for $q^i\in H_N^3(\T)$  and $\Phi\in C_b^2(\R^k)$. 
	Then
	\begin{align}\begin{split}\label{Eq158_New2}
			&
			\eta(h_t)\,-\, 	\eta(h_0 )\\&\quad  -\, \lim_{N\to \infty} \Biggl(\int_0^t 
			\sum_{i=1}^k 
			\partial_i\Phi \biggl(\int_{\T}  h_s q^1,\dots , \int_{\T}  h_s q^k \biggr) \\&\qquad \qquad\times \int_\T m( \Pi_{\le N} h_s)\px^3\Pi_{\le N}h_s 
			\px\Pi_{\le N} q^i\,+\, \frac{\partial_x (D_N(2\pi  \cdot ))}{4}   m'(\Pi_{\le N}  h_s)  \partial_x \Pi_{\le N }q^i\Biggr)
			\\& \quad -\, 
			\int_0^t
			\sum_{i,j=1}^k
			\partial_{ij}\Phi \biggl(\int_{\T}  h_s q^1,\dots , \int_{\T}  h_s q^k \biggr) \int_\T 
			m(h_s) \px q^i \partial_x q^j 
		\end{split}
	\end{align}
	is a martingale with respect to the filtration generated by $h$.
\end{proposition}
\begin{proof}
The existence of the limit in \eqref{Eq158_New2} follows from the convergence \eqref{Eq307} in $\D(\EE)^*$ together with Definition \ref{defi_energy_sol_iii'}~\eqref{Item_energy_est_New}. It remains to verify martingality of the resulting process, which we achieve in several steps by justifying an application of It\^o's formula to \eqref{Eq158_New}. Throughout we set
\begin{equation}\label{Eq267}
A^{[q]} \,=\, 
\lim_{N\to\infty} \int_0^\cdot \int_\T m( \Pi_{\le N} h_s)\px^3\Pi_{\le N}h_s 
\px\Pi_{\le N} q\,+\, \frac{\partial_x (D_N(2\pi  \cdot ))}{4}    m'(\Pi_{\le N}  h_s)  \partial_x \Pi_{\le N }q ,
\end{equation}
for $q\in H^3_N(\T)$, which is  attained in $L^1(\Omega; C([0,T]))$ and also $L^1(\Omega; \pv([0,T]))$ for $p>2$ by the preceding comments. We also recall the remark made below \eqref{E35}, namely that the second term in the above integrand doesn't arise in a spatially periodic situation. Therefore, we refer to this term in the following  as the boundary contribution. 

\emph{Step 1. (Quantified version of \eqref{Eq267})}
 We start by deriving a rate for the convergence \eqref{Eq267}. For this, we combine Definition \ref{defi_energy_sol_iii'}~\eqref{Item_energy_est_New}, Lemma \ref{Lemma_generator_approx_problem}, Lemma \ref{Lemma_conv_gens} and \eqref{Eq161} to deduce 
 \begin{align}&
	\E\biggl[\biggl\|
\int_0^\cdot \int_\T m( \Pi_{\le N} h_s)\px^3\Pi_{\le N}h_s 
\px\Pi_{\le N} q\,+\, \frac{\partial_x (D_N(2\pi  \cdot ))}{4}    m'(\Pi_{\le N}  h_s)  \partial_x \Pi_{\le N }q  \,-\, A^{[q]} 
	\biggr\|_{\pv, T}
	\biggr] \\&
	\quad =\,  \lim_{M\to\infty} 
	\E\biggl[
	\biggl\|
	\int_0^\cdot \int_\T m( \Pi_{\le N} h_s)\px^3\Pi_{\le N}h_s 
	\px\Pi_{\le N} q\,+\, \frac{\partial_x (D_N(2\pi  \cdot ))}{4}    m'(\Pi_{\le N}  h_s)  \partial_x \Pi_{\le N }q \\& \qquad 
	- \int_0^\cdot \int_\T m( \Pi_{\le M} h_s)\px^3\Pi_{\le M}h_s 
	\px\Pi_{\le M} q\,+\, \frac{\partial_x (D_M(2\pi  \cdot ))}{4}    m'(\Pi_{\le M} h_s )  \partial_x \Pi_{\le M }q
	\biggr\|_{\pv,T}
	\biggr]
	\\&\quad \lesssim_{p,T}\,  \lim_{M\to\infty} \biggl\|
 \int_\T m( \Pi_{\le N} \cdot )(\px^3\Pi_{\le N}\cdot ) 
	\px\Pi_{\le N} q\,+\, \frac{\partial_x (D_N(2\pi  \cdot ))}{4}    m'(\Pi_{\le N}  \cdot )  \partial_x \Pi_{\le N }q \\& \qquad 
	- \int_\T m( \Pi_{\le M} \cdot )(\px^3\Pi_{\le M}\cdot )
	\px\Pi_{\le M} q\,+\, \frac{\partial_x (D_M(2\pi  \cdot ))}{4}    m'(\Pi_{\le M} \cdot )  \partial_x \Pi_{\le M }q
	\biggr\|_{\D(\EE)^*}
	\\&\quad =\,  \biggl\|
	\int_\T m( \Pi_{\le N} \cdot )(\px^3\Pi_{\le N}\cdot ) 
	\px\Pi_{\le N} q\,+\, \frac{\partial_x (D_N(2\pi  \cdot ))}{4}    m'(\Pi_{\le N}  \cdot)  \partial_x \Pi_{\le N }q 
	\,-\, \L \eta
	\biggr\|_{\D(\EE)^*},
\end{align}
for  $\eta(h )=\int_\T h q $. To this end, we consider for $\zeta \in \D_0$ with $\|\zeta\|_{\D(\EE)} \le 1$, the tested
\begin{align}&
	\int_{\MM} 
	\biggl(
	\int_\T m( \Pi_{\le N} \cdot )(\px^3\Pi_{\le N}\cdot ) 
	\px\Pi_{\le N} q\,+\, \frac{\partial_x (D_N(2\pi  \cdot ))}{4}    m'(\Pi_{\le N}  \cdot)  \partial_x \Pi_{\le N }q 
	\,-\, \L \eta
	\biggr)\zeta
	\,\d\mu(h)
	\\& \quad =\, \EE(\eta, \zeta) \,+\, \int_{\MM}  
	\int_O  \Pi_{\le N}\partial_x( m(\Pi_{\le N} h)  \partial_x \Pi_{\le N}
q)
	\diff|_h\zeta  \,\d x
	\,\d\mu(h)
	\\&\quad \lesssim \, 
	\biggl(
	\int_{\MM}  
	\Bigl\|	\Pi_{\le N}  \partial_x \bigl(	m(\Pi_{\le N}h) \partial_x \Pi_{\le N} q \bigr)\,-\,  \partial_x \bigl(
	m(h)  \partial_x q\bigr)  \Bigr\|_{(\dot{H}^1(O ))^* }^2 
	\,\d\mu(h)\biggr)^{1/2}
	,
\end{align}
where we  reverted the calculations regarding $I_1$ in \eqref{Eq111} from Appendix \ref{App:proof_IBP} in the  equality. Using the triangle inequality, we may estimate the latter by 
\begin{align}&
	\biggl(
	\int_{\MM}  
	\Bigl\|	\Pi_{\le N}  \partial_x \bigl(	m(\Pi_{\le N}h) \partial_x \Pi_{\le N} q \bigr)\,-\, 	\Pi_{\le N} \partial_x \bigl(
	m(h)  \partial_x 	\Pi_{\le N} q\bigr)  \Bigr\|_{(\dot{H}^1(O ))^* }^2 
	\,\d\mu(h)\biggr)^{1/2}
	\\&\quad + 
	\biggl(
	\int_{\MM}  
	\Bigl\|		\Pi_{\le N} \partial_x \bigl(
	m(h)  \partial_x 	\Pi_{\le N} q\bigr)  \,-\,  	\Pi_{\le N} \partial_x \bigl(
	m(h)  \partial_x  q\bigr) \Bigr\|_{(\dot{H}^1(O ))^* }^2 
	\,\d\mu(h)\biggr)^{1/2} 
	\\&\quad + 
	\biggl(
	\int_{\MM}  
	\Bigl\|	\Pi_{\le N} \partial_x \bigl(
	m(h)  \partial_x  q\bigr) \,-\, \partial_x \bigl(
	m(h)  \partial_x  q\bigr)  \Bigr\|_{(\dot{H}^1(O ))^* }^2 
	\,\d\mu(h)\biggr)^{1/2} \,=\, J_1 \,+\, J_2\,+\, J_3,
\end{align}
for which we proceed separately. 

Regarding the first term, we use the boundedness of $\Pi_{\le N}\partial_x$  from $L^2(O)$ to $(\dot{H}^1(O))^*$, H\"older's inequality, and the Lipschitz continuity of $m$   by \eqref{Eq220} to estimate
\begin{align}J_1 \,&\lesssim \, 	\biggl(
	\int_{\MM}  
	\|	 	m(\Pi_{\le N}h)  - m(h) \|_{L^2(O)}^2 \| \partial_x \Pi_{\le N} q\|_{C(\overline{O})}^2 
	\,\d\mu(h)\biggr)^{1/2} \\& \lesssim_m\,  \biggl(
		\int_{\MM}  
	\|	 	\Pi_{\le N}h  - h \|_{L^2(O)}^2 
	\,\d\mu(h)\biggr)^{1/2}\| \partial_x \Pi_{\le N} q\|_{C(\overline{O})}.
\end{align} 
By the regularity \eqref{eqn72345} and the Sobolev embedding $H^3 (O) \hookrightarrow C^1({O})$ we may further bound this by 
\begin{align}C_\epsilon N^{-(1/2-\epsilon)} \| q\|_{H^3_N(O)},
	\label{eqnq}
\end{align}
for any $\epsilon>0$. 
For $J_2$, we arrive at the same upper bound based on 
\begin{align}
	J_2  \,\lesssim \, 	\biggl(
	\int_{\MM}  
	\|	 	 m(h) \|_{C(\overline O)}^2 \| \partial_x \Pi_{\le N} q - \partial_x q \|_{L^2(O)}^2 
	\,\d\mu(h)\biggr)^{1/2}.
\end{align}
Lastly, denoting the sine projection operator by $\Pi^\text{sine}_{\le N}$, we also have 
\begin{align}J_3 \,&\lesssim \, 	\biggl(
	\int_{\MM}  
	\|	 \Pi^\text{sine}_{\le N}  m(h) \partial_x q -  m(h) \partial_x q \|_{L^2(O)}^2
	\,\d\mu(h)\biggr)^{1/2}
	\\& \lesssim \, N^{-(1/2-\epsilon)} 	\biggl(
	\int_{\MM}  
	\|	 m(h) \partial_x q   \|_{C^{1/2-\epsilon}(O)}^2
	\,\d\mu(h)\biggr)^{1/2} ,
\end{align}
which is estimated by \eqref{eqnq}  using additionally the algebra property of H\"older spaces. All in all, we have derived that
\begin{align}\begin{split}\label{Eq271}
		&
	\E\biggl[\biggl\|
	\int_0^\cdot \int_\T m( \Pi_{\le N} h_s)\px^3\Pi_{\le N}h_s 
	\px\Pi_{\le N} q\,+\, \frac{\partial_x (D_N(2\pi  \cdot ))}{4}    m'(\Pi_{\le N}  h_s)  \partial_x \Pi_{\le N }q \,-\, A^{[q]} 
	\biggr\|_{\pv, T}
	\biggr]  \\&\quad \lesssim_{(\epsilon,m,p,q,T)}\, N^{-(1/2 -\epsilon)}  .
	\end{split}
\end{align}

\emph{Step 2. (Quantified explosion in H\"older spaces)} Next, based on the estimate from  part \eqref{Item_inc_New} of Definition \ref{defi_energy_sol_iii'} we quantify the explosion of the  right-hand side of \eqref{Eq267} in more regular spaces. From the aforementioned property of energy solutions, we deduce firstly that also
\begin{equation}\label{Eq276}
\E [|\zeta(h_t)|^r] \,\lesssim \, \|\zeta \|_{L^{2r}(\mu)}^r,\qquad r\ge 1, \, t\ge 0.
\end{equation} Regarding the first integrand in \eqref{Eq267}, this yields 
\begin{align}
\E\biggl[\biggl|
\int_s^t \int_\T m( \Pi_{\le N} h_\tau)\px^3\Pi_{\le N}h_\tau 
\px\Pi_{\le N} q  \biggr|^r
\biggr]
\, &\lesssim_{(m,q,r)} \,|t-s|^{r-1}\int_s^t \E \bigl[ \|\px^3\Pi_{\le N}h_\tau \|_{L^2(\T)}^r \bigr] \,\d \tau
\\&
 \lesssim_r   |t-s|^r  \biggl( \int_{\MM} \| \px^3 \Pi_{\le N}h\|_{L^2(\T)}^2 \,\d\mu(h) \biggr)^{r/2} ,
\end{align}
using additionally H\"older's inequality and equivalence of moments for Gaussian random variables. Applying \eqref{Eq223} we can calculate the latter quantity to be
\begin{align} \int_{\MM} \biggl\| \sum_{1\le k \le N} \partial_x^3 p^k \int_{\T} h \vp_k  \biggr\|_{L^2(O)}^2 \,\d\mu(h)  \,=\,  \sum_{1\le k \le N} (\pi k)^4 \|\vp_k\|_{(\dot H^1(O))^*}^2 \,\eqsim \, N^5,
\end{align}
so that
\begin{align}\label{Eq272}
\E\biggl[\biggl|
\int_s^t \int_\T m( \Pi_{\le N} h_\tau)\px^3\Pi_{\le N}h_\tau 
\px\Pi_{\le N} q  \biggr|^r
\biggr] \,\lesssim_{(m,q,r)} \, N^{5r/2}|t-s|^r.
\end{align}

For the second integrand in \eqref{Eq267}, we obtain after integrating by parts and using \eqref{Eq220} that
\begin{align}\E\biggl[\biggl|\int_s^t \int_\T \frac{ \partial_x (D_N(2\pi  \cdot ))}{4}   m'(\Pi_{\le N}  h_\tau)  \partial_x \Pi_{\le N }q   \biggr|^r
	\biggr]\,\lesssim_{(m,q,r)}\, \bigl( |t-s| \| \partial_x D_N(2\pi \cdot)\|_{L^1(O)}\bigr)^r
\end{align}
 Due to the $N\log N$-divergence of the term involving the Dirichlet kernel, see \cite[Lemma 2.5]{sheng1990extension}, this can, up to a constant, be estimated by the right-hand side of \eqref{Eq272} as well. Combining yields 
by the Kolmogorov--Chentsov theorem  that
\begin{align}\begin{split}\label{Eq301}&
\E\biggl[\biggl\|\int_0^\cdot 
\int_\T m( \Pi_{\le N} h_s)\px^3\Pi_{\le N}h_s 
\px\Pi_{\le N} q\,+\, \frac{\partial_x (D_N(2\pi  \cdot ))}{4}    m'(\Pi_{\le N}  h_s)  \partial_x \Pi_{\le N }q \biggr\|_{C^{1/p'}([0,T])}^r
\biggr]^{1/r}\\&\quad \lesssim_{(m,p',q,r,T)}\, N^{5/2},\end{split}
\end{align}
for any $r,p'\in (1,\infty)$.

\emph{Step 3. (Interpolation)} 
Next, we consider real interpolation between the spaces
\[
L^{1}(\Omega; \pv([0,T]))
\quad
\text{ and }\quad L^{r}(\Omega; C^{1/p'}([0,T])),
\]
for which the $K$-functional
is defined as 
\[
K(\sigma, A^{[q]}) \,=\, \inf_{A^{[q]} = B+C}\bigl(\|B\|_
{L^{1}(\Omega; \pv([0,T]))}+\sigma \|C\|_{L^{r}(\Omega; C^{1/p'}([0,T]))}\bigr),
\]
see \cite{berghlofstroom}. Then the real interpolation norm is given by
\begin{equation}\label{Eq302}
\|A^{[q]}\|_{\theta,1}\,\eqsim\, \sum_{l=-\infty}^\infty 2^{-l\theta}K(2^l, A^{[q]}),
\end{equation}
and we bound for $l\ge  0$ simply $K(2^{l}, A^{[q]})\le \|A^{[q]}\|_
{L^{1}(\Omega; \pv([0,T]))}$. For $l<0$, on the other hand, we take $N\eqsim 2^{{-l/3}}$ and  then estimate
\begin{align}\begin{split}\label{Eq303}&
	K(2^l, A^{[q]}) \\& \quad\le\, \biggl\| \int_0^\cdot 
	\int_\T m( \Pi_{\le N} h_s)\px^3\Pi_{\le N}h_s 
	\px\Pi_{\le N} q\,+\, \frac{\partial_x (D_N(2\pi  \cdot ))}{4}    m'(\Pi_{\le N}  h_s)  \partial_x \Pi_{\le N }q  \,-\, A^{[q]} \biggr\|_{	L^{1}(\Omega; \pv([0,T]))}
	\\&\qquad +\,2^l \biggl\| 	\int_0^\cdot 
	\int_\T m( \Pi_{\le N} h_s)\px^3\Pi_{\le N}h_s 
	\px\Pi_{\le N} q\,+\, \frac{\partial_x (D_N(2\pi  \cdot ))}{4}    m'(\Pi_{\le N}  h_s)  \partial_x \Pi_{\le N }q  \biggr\|_{	L^{r}(\Omega; C^{1/p'}([0,T]))}
	\\& \quad  \lesssim_{(\epsilon,m,p',q,r,T)}\, 2^{(1/2-\epsilon)l/3}\,+\, 2^{l - 5l/6}\,\lesssim\, 2^{(1/2-\epsilon)l/3},
\end{split}\end{align}
based on the previously obtained \eqref{Eq271} and \eqref{Eq301}. In particular, we found that
\begin{equation}\label{Eq311}
A^{[q]}\,\in \, \bigl(	
L^{1}(\Omega; \pv([0,T])), L^{r}(\Omega; C^{1/p'}([0,T]))
\bigr)_{\theta,1}
\end{equation}
for $\theta\in (0,1/6)$, since  choosing $\epsilon$ sufficiently small makes \eqref{Eq302}  finite for this range. 
Now we use \cite[Section 1.18.4]{triebel_95} together with \cite[Theorem 2.1]{bergh_Vp} to deduce that
\begin{align}\begin{split}\label{eqn22}&
\bigl(	
L^{1}(\Omega; \pv([0,T])), L^{r}(\Omega; C^{1/p'}([0,T]))
\bigr)_{\theta,1} \,\hookrightarrow \, 
L^{1} 
\bigl(	 \Omega ;
\bigl(	 {B}^{1/p}_{p,\infty}((0,T)),  B^{1/p'}_{p,\infty}((0,T))
\bigr)_{\theta,1}\bigr)
\\&
\quad \hookrightarrow \, 
L^{1} 
(	 \Omega ;
{B}_{p,1}^{1/\tilde{p}}((0,T))
)\, 	\hookrightarrow \, 
L^{1} 
(	 \Omega ;
{B}_{\tilde{p},1}^{1/\tilde{p}}((0,T))
),
\end{split}
\end{align}
where we assume that $p'\le p$ and
\[
\frac{1}{\tilde{p}}\,=\, (1-\theta)\frac{1}{p} \,+\, \theta\frac{1}{p'}.
\]
The last space in the chain \eqref{eqn22} embeds into $L^{1}(\Omega; \ptv([0,T]))$ by the other inclusion from \cite[Theorem 2.1]{bergh_Vp}. Taking then $\theta \sim 1/6$, $p'\sim 1$ and $p\sim 2$ results in  $\tilde{p}\sim 12/7<2$.

In combination with \eqref{Eq311}, this yields improved $ \ptv([0,T])$-regularity of $A^{[q]}$. To strengthen \eqref{Eq267} accordingly,  we estimate for $2^l > N^{-3}$
\begin{align}&
K\biggl(2^l,\int_0^\cdot 
\int_\T m( \Pi_{\le N} h_s)\px^3\Pi_{\le N}h_s 
\px\Pi_{\le N} q\,+\, \frac{\partial_x (D_N(2\pi  \cdot ))}{4}    m'(\Pi_{\le N}  h_s)  \partial_x \Pi_{\le N }q   \,-\, A^{[q]} 
\biggr)\\& \quad \lesssim_{(\epsilon,m,p,q,T)}\, 2^{(1-\lambda)(1/2-\epsilon)l/3}N^{-\lambda (1/2-\epsilon)},
\end{align}
by \eqref{Eq271} for any $\lambda \in (0,1)$. For $2^{l}\le N^{-3}$ we use \eqref{Eq301} and \eqref{Eq303} to deduce that
\begin{align}
&K\biggl( 2^l,\int_0^\cdot 
\int_\T m( \Pi_{\le N} h_s)\px^3\Pi_{\le N}h_s 
\px\Pi_{\le N} q\,+\, \frac{\partial_x (D_N(2\pi  \cdot ))}{4}    m'(\Pi_{\le N}  h_s)  \partial_x \Pi_{\le N }q    \,-\, A^{[q]} 
\biggr)
\\&\quad \le \, K\biggl( 2^l,\int_0^\cdot 
\int_\T m( \Pi_{\le N} h_s)\px^3\Pi_{\le N}h_s 
\px\Pi_{\le N} q\,+\, \frac{\partial_x (D_N(2\pi  \cdot ))}{4}    m'(\Pi_{\le N}  h_s)  \partial_x \Pi_{\le N }q 
\biggr) \,+\,K( 2^l, A^{[q]})
\\&\quad \lesssim_{(\epsilon,m,p',q,r,T)}\, 2^{l}N^{5/2}\,+\, 2^{(1/2-\epsilon)l/3}\,\le\,2^{(1-\lambda)l/6}N^{-\lambda/2}\,+\,  2^{(1-\lambda)(1/2-\epsilon)l/3 }N^{-\lambda(1/2-\epsilon)},
\end{align}
instead.  By reducing as above to $l<0$ and choosing $\lambda$ sufficiently small, we deduce that the limit \eqref{Eq267} is attained in \eqref{Eq311} for admissible $(p,p',r,\theta)$ and therefore also in $L^{1}(\Omega; \ptv([0,T]))$ for $\tilde{p}\sim 12/7$.

\emph{Step 4. (Recovering It\^o's formula)} Having established the improved convergence similar to \cite[Corollary 3.17]{GubPer_uniqueness}, we can essentially follow \cite[Lemma 4.14]{GubPer_generator} to arrive at \eqref{Eq158_New2}. Indeed, the latter relies on applying It\^o's formula to the semimartingales
\begin{equation}\label{Eq275}
S_{N,t}^{[q^i]}\,=\, \int_\T h_0 q^i \,+\, \int_0^t  \int_{\T} m( \Pi_{\le N} h_s)\px^3\Pi_{\le N}h_s 
\px\Pi_{\le N} q^i \,+\, \frac{\partial_x (D_N(2\pi  \cdot ))}{4}    m'(\Pi_{\le N}  h_s)  \partial_x \Pi_{\le N }q^i 
\,+\, M_t^{[q^i]},
\end{equation}
where 
\begin{equation}\label{Eq310}
M_t^{[q^i] }\,=\, \int_\T h_t q^i \,-\,  \int_\T h_0 q^i \,-\, A_t^{[q^i]}
\end{equation}
are martingales with quadratic variation
\[
\langle M^{[q^i] },
M^{[q^j] }\rangle_t \,=\,  2
\int_0^t \int_\T 
m(h_s) \px q^i  \px q^j ,\]
as follows by Definition \ref{defi_energy_sol_iii'}~\eqref{Item_MP_New} and  polarization. 
This yields, for any $N$, that
\begin{align}\begin{split}
	\label{Eq277}
	&
	\Phi(S_{N,t}^{[q^1]},\dots , S_{N,t}^{[q^k]} )\,-\, 	\eta(h_0)\,  -\, \int_0^t 
	\sum_{i=1}^k
	\partial_i\Phi (S_{N,s}^{[q^1]},\dots , S_{N,s}^{[q^k]} ) \,\d A^{[q^i]}_{N,s}
	\\& \quad -\, 
	\int_0^t
	\sum_{i,j=1}^k 
	\partial_{ij}\Phi(S_{N,s}^{[q^1]},\dots , S_{N,s}^{[q^k]} )   \int_\T 
	m( h_s)  \px q^i \px q^j 
\end{split}
\end{align}
is a local martingale with quadratic variation
\[
2\sum_{i,j=1}^k
\int_0^t (\partial_i\Phi\partial_j\Phi ) (S_{N,s}^{[q^1]},\dots , S_{N,s}^{[q^k]} )   \int_\T 
m( h_s)  \px q^i \px q^j  ,
\]
where we use the notation
\[
A^{[q^i]}_{N,t} \,=\, \int_0^t  \int_{\T} m( \Pi_{\le N} h_s)\px^3\Pi_{\le N}h_s 
\px\Pi_{\le N} q^i \,+\, \frac{\partial_x (D_N(2\pi  \cdot ))}{4}    m'(\Pi_{\le N}  h_s)  \partial_x \Pi_{\le N }q^i  .
\]
Since by assumption $\Phi \in C^2_b(\R^k)$, 
\eqref{Eq277} is a uniformly integrable family of martingales. Up to passing to a subsequence, we have moreover that 
\begin{align}\begin{split}
	\label{Eq333}&
	\biggl(
	\Phi(S_{N,t}^{[q^1]},\dots , S_{N,t}^{[q^k]} )\,-\, 	\eta(h_0)\,  -\, 
	\int_0^t
	\sum_{i,j=1}^k 
	\partial_{ij}\Phi(S_{N,s}^{[q^1]},\dots , S_{N,s}^{[q^k]} )   \int_\T 
	m( h_s)  \px q^i \px q^j  \biggr)
	\\&\qquad \to\, \biggl(\eta(h_t) \,-\, 	\eta(h_0) \, -\, 
	\int_0^t
	\sum_{i,j=1}^k
	\partial_{ij}\Phi \biggl(\int_{\T}  h_s q^1,\dots , \int_{\T}  h_s q^k \biggr) \int_\T 
	m(h_s) \px q^i \partial_x q^j  \biggr)	,	\end{split}
\end{align}
in $C([0,T])$, $\P$-a.s., due to the convergence \eqref{Eq267} in $L^1(\Omega;C([0,T]))$.  Invoking the previous step, we can  similarly conclude that 
$A_N^{[q^i]} \to A^{[q^i]} $  in  $\tilde{p}\text{-}\mathrm{Var}([0,T])$  and thus 
$S_{N}^{[q^i]} \to \int_\T h q^i$ in $(2+\delta)\text{-}\mathrm{Var}([0,T])$, $\P$-a.s., for  all $\delta>0$.  
In light of the regularity of $\Phi$, the latter implies that
also
\begin{equation}\label{Eq306}
\partial_i\Phi (S_{N}^{[q^1]},\dots , S_{N}^{[q^k]} ) \,\to\, 
\partial_i\Phi \biggl(\int_\T h q^1 ,\dots ,\int_\T h q^k \biggr)
\end{equation}
in $(2+\delta)\text{-}\mathrm{Var}([0,T])$, $\P$-almost surely, for all such $\delta$. We conclude that along this subsequence
\[
\biggl|\sum_{i=1}^k 
\int_0^t
\partial_i\Phi (S_{N,s}^{[q^1]},\dots , S_{N,s}^{[q^k]} )  \,-\, \partial_i\Phi \biggl(\int_\T h_s q^1 ,\dots ,\int_\T h_s q^k \biggr)  \,\d A^{[q^i]}_{N,s} \biggr| \,\to\, 0,
\]
$\P$-a.s.,
as $N\to\infty$ 
in light of  the  Young--Lo\`{e}ve estimate \cite[Theorem 6.8]{friz_victoir}. Together with \eqref{Eq333} and  the preservation of martingality by uniform integrability of \eqref{Eq277}, the proof is complete.
\end{proof}
We deduce  a version of Proposition \ref{Prop_ES_to_MP}  for general $\eta \in \D(\EE)$. 
\begin{proposition}\label{gen_MP}
Let $\frac{\d \nu}{ \d \mu}\in L^2(\mu)$ be a probability density and  $h$ be an energy solution  to \eqref{EqSTFE}  starting from $\nu$ in the sense of Definition \ref{defi_energy_sol_iii'}. Then, for all $\eta\in \D(\EE)$  and $\D(\L)\ni \eta_N \to \eta$ in $\D(\EE)$
\begin{align}\label{Eq141_New}
\eta(h_t)\,-\, \eta(h_0)\,-\, \lim_{N\to \infty} \int_0^t 
\L \eta_N(h_s)\,\d s
\end{align}
is a martingale with respect to the filtration generated by $h$. 
\end{proposition}
\begin{proof}Similarly to the discussion below Lemma \ref{Lemma_conv_gens}, we see that the limit in \eqref{Eq141_New} 	exists independently of the approximating sequence  $\eta_N$ in $L^1(\Omega; C([0,T]))$ by combining Definition \ref{defi_energy_sol_iii'}~\eqref{Item_energy_est_New} with the continuity of $\L\colon \D(\EE) \to \D(\EE)^*$. Accordingly, we write 
\[\int_0^t 
\L \eta (h_s)
\,\d s \,=\, \lim_{N\to \infty} \int_0^t 
\L \eta_N(h_s)
\,\d s
\]
for $\eta\in \D(\EE)$ and \eqref{Eq142_New} implies the bound
\begin{equation}\label{Eq282}
\E\biggl[
\sup_{t\le T} \biggl|\int_0^t  \L  \eta (h_s) \,\d s \biggr|
\biggr] \,\lesssim \,(T^{1/2}+T) \|\eta \|_{\D(\EE)}.
\end{equation}

Regarding \eqref{Eq141_New}, we argue that it suffices to check the claim only for 
$\eta$ as in  \eqref{Eq11_New} with  $q^i\in H_N^3(\T)$  and $\Phi\in C_b^2(\R^k)$. Indeed, since the latter are dense in $\D(\EE)$ by construction, we can find a sequence of such $\eta^{(l)}$ that converges to $\eta$ in $\D(\EE)$.  
Using \eqref{Eq282} we see then that
\[\int_0^t
\L \eta^{(l)}(h_s) \,\d s \,\to\, 
\int_0^t
\L\eta (h_s) \,\d s ,
\]
as $l\to\infty$
in $L^1(\Omega)$. Additionally, due to Definition \ref{defi_energy_sol_iii'}~\eqref{Item_inc_New}, we have
\[
\eta^{(l)}(h_t) -\eta^{(l)}(h_0)\, \to  \, \eta (h_t) -\eta(h_0)
\]
in $L^1(\Omega)$. Since martingality is preserved under limits in $L^1(\Omega)$, the above reduction is justified. 

Continuing with $\eta$ as in \eqref{Eq11_New}, we obtain from Lemma \ref{Lemma_generator_approx_problem} that 
$\eta \in \D(\check{\L}_{N})$ and thus $\tilde{\eta}_N = \L^{-1} \check{\L}_{N} \eta \in \D(\L)$. Writing $\eta = \tilde{\eta} + \lambda$ for $(\tilde{\eta},\mathbf{1}_\MM)_{L^2(\mu)}= 0$ and $\lambda\in \R$, it follows $ \L\tilde{\eta}_N \to \L \tilde{\eta} $ in $\D(\EE)^*$ by Lemma \ref{Lemma_conv_gens} and therefore $ \eta_N = \tilde{\eta}_N  + \lambda  \to \eta$ in $\D(\EE)$.
Using this as approximating sequence, we may by Lemma \ref{Lemma_generator_approx_problem} identify the limit in 
\eqref{Eq141_New} as 
\begin{align}\begin{split}
& \lim_{N\to \infty} \biggl(\int_0^t  \sum_{i=1}^k
\partial_i	\Phi\biggl(\int_{\T}   h_s q^1,\dots , \int_{\T}  h_s q^k \biggr)
\bigl( h_s, 
\partial_x^2 \Pi_{>N}q^i
\bigr)_{\dot{H}^1(O)  } 
\\& \qquad\qquad +\,\sum_{i,j=1}^k \partial_{ij} \Phi\biggl(\int_{\T}  h_s  q^1,\dots , \int_{\T}  h_s  q^k \biggr) \int_\T  \partial_x \Pi_{>N} q^i \partial_x \Pi_{>N}  q^j 
\\& \qquad \qquad +\, 	\sum_{i=1}^k
\partial_i\Phi\biggl(\int_{\T}  h_s  q^1,\dots , \int_{\T}  h_s  q^k \biggr)\int_\T m( \Pi_{\le N} h_s)\px^3\Pi_{\le N}h_s 
\px\Pi_{\le N} q^i
\\&\qquad\qquad
+\, \frac{1}{4}	\sum_{i=1}^k
\partial_i\Phi\biggl(\int_{\T}  h_s  q^1,\dots , \int_{\T}  h_s  q^k \biggr)\int_\T {\partial_x (D_N(2\pi  \cdot ))}   m'(\Pi_{\le N}  h_s)  \partial_x \Pi_{\le N }q^i
\\& \qquad\qquad +\,
\sum_{i,j=1}^k 
\partial_{ij} 
\Phi\biggl(\int_{\T}  h_s  q^1,\dots , \int_{\T}  h_s  q^k \biggr)  \int_\T  
m(\Pi_{\le N} h_s)  \partial_x \Pi_{\le N} q^i \partial_x \Pi_{\le  N}  q^j 
\,\d s\biggr).
\end{split}\end{align}
Together with \eqref{Eq161}--\eqref{Eq163}\noeqref{Eq162} and  Definition \ref{defi_energy_sol_iii'}~\eqref{Item_inc_New},  
martingality of \eqref{Eq141_New} may be deduced from Proposition~\ref{Prop_ES_to_MP}.
\end{proof}
\subsection{Uniqueness of energy solutions}\label{ss:uniqueness}
Here, we  conclude from the preceding results the uniqueness in law of energy solutions to \eqref{EqSTFE}. Combining it with
Proposition \ref{Prop_DF_yields_ES} yields that any energy solution coincides in law with the diffusion process associated to $(\EE,\D(\EE))$, completing the proof of Theorem \ref{thm_characterization}.
\begin{corollary}Let $\frac{\d \nu}{ \d \mu}\in L^2(\mu)$ be a probability density and   $h$ be an energy solution  to \eqref{EqSTFE}  starting from $\nu$ in the sense of Definition \ref{defi_energy_sol_iii'}. Then, we have 
\begin{align}\label{Eq149}
\E\biggl[\prod_{l=0}^{k} \zeta_l (h_{t_l}) \biggr] \,=\, \int_{\MM}
\zeta_0T_{t_1 - t_0} \bigl(  \zeta_1 T_{t_2-t_1} (  \dots  T_{t_k - t_{k-1}} \zeta_k  )   \bigr)
\,\d \nu,
\end{align}
for
$0=t_0<\dots<t_k$ and bounded $\zeta_l\in L^2(\mu)$.
\end{corollary}
\begin{proof}Having established Proposition \ref{gen_MP},
we follow the argument in \cite[Section 4.2, Section 4.4]{GubPer_generator}, which we recall for the reader's convenience. Firstly, we observe that if $\eta\in \D(\L)$ no approximating sequence $\eta_N$ is needed in Proposition \ref{gen_MP} and hence
\begin{equation}\label{Eq147}
\eta(h_t)\,-\, \eta(h_0)\,-\, \int_0^t 
\L\eta(h_s)
\,\d s
\end{equation}
is a martingale.
More generally, if $\eta \colon [0,T]\to \D(\L)$ is continuous with $ \partial_t\eta \in C([0,T];L^2(\mu))$, then  
\[
\eta(t,h_t)\,-\, \eta(0,h_0)\,-\, \int_0^t 
(\partial_t + \L)\eta(s,h_s)\,\d s
\]
is a martingale, as follows from a temporal discretization procedure together with Definition \ref{defi_energy_sol_iii'}~\eqref{Item_inc_New}, see \cite[Lemma A.3]{GubPer_generator}. Finally, since the base case is clear, we can proceed by induction and assume that \eqref{Eq149} holds for time instances  $0=t_0<\dots <t_{k-1}$ and bounded $\zeta_l\in L^2(\mu)$. Then, assuming at first   $\zeta_{{k}}\in \D(\L)$ we set $\eta(t) =  T_{t_k-t}\zeta_{k}$ on  $[0,t_k]$ so that $(\partial_t + \L)\eta = 0$. By our previous considerations,
\[
T_{t_k-t}\zeta_{k}(h_t) - T_{t_k}\zeta_{k}(h_0)
\]
is a martingale and therefore
\[
\E[\zeta_{k}(h_{t_k}) | \mathfrak{F}_{t_{k-1}}^{h} ] \,=\, T_{t_k-t_{k-1}}\zeta_{{k}}(h_{t_{k-1}}).
\]
Consequently, we obtain  \eqref{Eq149} also for $k+1$ time instances by the tower property of conditional expectations and the general case  $\zeta_{{k}}\in L^2(\mu)$ follows by approximation.
\end{proof}

	\appendix

	\section{Prerequisites from Dirichlet form theory}\label{app:DFs}
	In this appendix we review some foundational notions from Dirichlet form theory, which are used throughout this manuscript. 
	\subsection{Dirichlet forms and diffusion processes}
	Let us first introduce some basic notions from Dirichlet form theory and refer to \cite{FOT,ma2012introduction} for a complete introduction to the topic. For now, let  $(\X,\nu)$ be a $\sigma$-finite measure space. Then a (in this manuscript always symmetric) Dirichlet form is defined as follows.
	\begin{definition}[Dirichlet form]
		Let $\D(\EE)\subset L^2(\nu)$ be a dense linear subspace. A bilinear mapping $\EE\colon \D(\EE) \times \D(\EE) \to \R$ is called a Dirichlet form if $(\EE,\D(\EE))$ is 
		\begin{enumerate}[(i)]
			\item \emph{symmetric:} for all $\eta,\zeta \in \D(\EE)$ we have $\EE(\eta,\zeta) = \EE(\zeta,\eta )$, 
			\item \emph{non-negative definite:} for all $\eta \in \D(\EE)$ we have $\EE(\eta,\eta) \ge 0 $,  
			\item \emph{closed:} $\D(\EE)$ is a Hilbert space with the inner product $(\eta,\zeta)_{L^2(\nu)} + \EE (\eta ,\zeta)$,  
			\item \emph{Markovian:} For all $\eta\in \D(\EE)$, we have $\zeta  = (\eta\vee 0) \wedge 1 \in \D(\EE)$ with $\EE(\zeta,\zeta )\le \EE(\eta,\eta)$.
		\end{enumerate}
	\end{definition}
	In this case, making \eqref{e:generator} more precise, the generator of $(\EE,\D(\EE))$ is the densely defined, non-positive definite, self-adjoint operator $(\L,\D(\L))$ on $L^2(\nu)$  characterized by
	\begin{equation}\label{Eq42}
		\D(\L) \,=\, \bigl\{
		\eta \in \D(\EE)\,\big|\, \exists \L \eta\in L^2(\nu)\,\forall \zeta \in \D(\EE) :
		\EE(\eta,\zeta) \,=\, -(\L\eta,\zeta)_{L^2(\nu)} 
		\bigr\}.
	\end{equation}
	Conversely, one can recover $(\EE,\D(\EE))$ from $(\L,\D(\L))$ using spectral calculus via
	\begin{equation}\label{Eq233}
		\D(\EE) \,=\, \D(\sqrt{-\L}),\qquad \EE(\eta,\zeta)\,=\, ( \sqrt{-\L}\eta, \sqrt{-\L}\zeta)_{L^2(\nu)}.
	\end{equation}
	By semigroup theory, or again by spectral calculus, one may define $T_t = \exp(t \L)$, which is a bounded operator on $L^2(\nu)$ for any $t>0$. Moreover, it is a contraction, i.e., $\|T_t\eta\|_{L^2(\nu)} \le \|\eta \|_{L^2(\nu)}$. The family $(T_t)_{t>0}$ is a strongly continuous semigroup of operators, which means that $
	T_t T_s \,=\, T_{t+s}$ for $s,t>0$ and $\lim_{t\searrow 0 } T_t \eta =\eta $ in $L^2(\nu)$ for all $\eta \in L^2(\nu)$.
	The generator $\L$ can be recovered from the family $(T_t)_{t>0}$ using \eqref{e:generation}. More precisely, we have
	\begin{equation}
		\D(\L) \,=\, \Bigl\{
		\eta \in L^2(\nu)\,\Big|\, \exists \L \eta\,=\,   \lim_{t\searrow 0}  \frac{1}{t} (T_t\eta - \eta ) \, \in \, L^2(\nu)
		\Bigr\}.
	\end{equation}
	The Markovian property of $(\EE,\D(\EE))$ is precisely what guarantees that the semigroup $(T_t)_{t>0}$ is Markovian, which means that $0\le\eta \le 1$, $\nu$-a.e., implies that also $0\le T_t\eta \le 1$, $\nu$-a.e., for any $t>0$.

	\emph{Next, we assume that $\X$ is a Polish space}, in order to impose suitable path properties on Markov processes. We equip it  with its Borel $\sigma$-field, on which we assume  $\nu$ to be defined, without further mention. Since Polish spaces are a subclass of Lusin and therefore also Suslin spaces, this falls into the scope of the introduction to Markov processes from \cite[Appendix A.2]{FOT} and \cite[Section IV.1]{ma2012introduction}. Moreover, we restrict our discussion to conservative diffusion processes on $\X$, i.e., diffusion processes with infinite life time, which circumvents further subtleties, and \emph{we assume for simplicity that $\nu(\X)=1$.}
	\begin{definition}[Conservative diffusion process]\label{defi_markov_proc}
		Let $(\Omega, \mathfrak{A} )$ be a measurable space, $(Z_t)_{t\ge 0}$ be an $\X$-valued stochastic process with continuous paths and $(P_z)_{z\in\X}$ a family of probability measures on $(\Omega,\mathfrak{A} )$. Then $(\Omega, \mathfrak{A}, (Z_t)_{t\ge 0} , (P_z)_{z\in\X} )$ is called a conservative diffusion process if there exists a right-continuous filtration  $ \mathfrak{F}$   such that $(Z_t)_{t\ge 0}$ is $\mathfrak{F}$-adapted and the following holds: 
		\begin{enumerate}[(i)]
			\item \label{defi_MP_i1}For each $t\ge 0$ and Borel measurable $B\subset \X$, the mapping $z\mapsto P_z(\{Z_t \in B\}) $ is Borel measurable.
			\item For each $z\in \X$ we have $P_z(Z_0 =z) = 1$.
			\item For each $t\ge 0$, there exists a mapping $\theta_t \colon \Omega \to\Omega$ such that $Z_s \circ\theta_t  = Z_{s+t}$ for all $s\ge 0$.
			\item \label{defi_MP_i4} For any probability measure $\mu $ on $\X$, Borel measurable $B\subset \X$ and  $\mathfrak{F}$-stopping time $\tau<\infty$, we have
			\begin{align}
				P_\mu \bigl(\{Z_{\tau +s}\in B\} \,\big|\,\mathfrak{F}_\tau \bigr) \,=\, P_{Z_\tau} ( \{Z_s\in B\}),
			\end{align}
			$P_\mu$-a.s.,
			where $ P_\mu(\Lambda) = \int_{\X} P_z(\Lambda)\, \d\mu$.
		\end{enumerate}
	\end{definition}
	If $(\Omega, \mathfrak{A}, (Z_t)_{t\ge 0} , (P_z)_{z\in\X} )$ is a conservative diffusion process, then the filtration $\FF$ can be taken to be its natural (or minimum completed admissible) filtration 
	\begin{equation}\label{e:filtr}
		\FF_t \,=\, \bigcap_{\mu }\overline{\sigma \{ Z_s | s\in [0,t]\}}^{P_\mu},
	\end{equation}
	where the intersection is taken over all probability measures on $\X$. Accordingly, one can also restrict the underlying $\sigma$-field to the canonically defined $\mathfrak{F}_\infty$, which we do in the following without saying.
	Moreover, we say sometimes that $( (Z_t)_{t\ge 0} , (P_z)_{z\in\X} )$ is a conservative diffusion process, completely  suppressing the underlying measure space. \emph{Throughout this manuscript, we also  refer to a conservative diffusion process  as a diffusion  process or just as a diffusion.} 
	
	In any case, each  diffusion process induces a transition function $p_t \eta(z) = E_z[F(Z_t)]$ for bounded and Borel measurable $\eta \colon \X \to \R$. By property \eqref{defi_MP_i1} of Definition \ref{defi_markov_proc}, the above defines again a Borel measurable mapping  and  using \eqref{defi_MP_i4} one finds the semigroup property $p_t (p_s \eta ) = p_{t+s}\eta $. Let's assume 
	that $(p_t)_{t>0}$ is $\nu$-symmetric, i.e., that
	\[
	\int_{\X} \eta (p_t \zeta) \,\d\nu \,=\, \int_{\X} (p_t\eta)  \zeta \,\d\nu 
	\]
	for any $t>0$ and non-negative Borel measurable $\eta,\zeta \colon \X\to \R$. Then by Jensen's inequality, we find
	\[
	\int_{\X} |p_t \eta|^2 \,\d\nu \,=\, \int_{\X} |E_z [\eta(Z_t)]|^2 \,\d\nu (z) \,\le\, 
	\int_{\X} E_z [\eta^2(Z_t)] \,\d\nu (z) \,=\, \int_{\X} p_t (\eta^2)\,\d\nu\,=\, \int_{\X} \eta^2\,\d\nu,
	\] 
	for bounded, Borel measurable $\eta$ and any $t>0$. Thereby, one can extend  $p_t$ to a contraction operator $T_t$ on $L^2(\nu)$. \emph{If the latter defines the Markovian semigroup of a Dirichlet form $(\EE,\D(\EE))$, we say that the  diffusion process $( (Z_t)_{t\ge 0} , (P_z)_{z\in\X} )$ and $(\EE,\D(\EE))$ are associated to each other.} In this case the diffusion may be analyzed by means of $(\EE,\D(\EE))$, as we do throughout this manuscript, and vice versa.

	A classical question in Dirichlet form theory is to find necessary and sufficient conditions under which a given Dirichlet form admits an associated  diffusion process (or is associated to other types of Markov processes). Since the foundational work towards the end of the last century, these conditions are nowadays well-understood, see, e.g., \cite[Chapter IV]{ma2012introduction}, and a central role is played by the associated capacity
	\begin{equation}
		\Ca (O) \,=\, \inf \bigl\{ \EE(\eta, \eta )\,+\, \|\eta \|^2_{L^2(\nu)} \,\big|\, 
		\eta \in \D(\EE), \, \eta \ge 1_O \; \text{$\nu$-a.e.}
		\bigr\},
	\end{equation}
	for an open subset $O\subset \X$ and a given  Dirichlet form $(\EE,\D(\EE))$ on $L^2(\nu)$. Then, for a general subset  $S\subset \X$, one sets 
	\[
	\Ca (S) \,=\, \inf \bigl\{ \Ca (O) \,\big|\, 
	S\subset O, \; \text{$O$ open}
	\bigr\}.
	\]
	One calls moreover a set $N\subset\X$ \emph{$\EE$-exceptional} if $\Ca(N) = 0$ and an increasing sequence of closed sets $(S_k)_{k\in \N}$ is called an \emph{$\EE$-nest} if $\Ca(\X\setminus S_k) \to 0$, as $k\to\infty$. If a property holds outside an \emph{$\EE$-exceptional} subset of $\X$, we say that it holds \emph{$\EE$-quasi everywhere}.  Similarly, a function $\eta\colon \X \to\R$ is called \emph{$\EE$-quasi continuous}, if there exists an $\EE$-nest such that  $\eta|_{S_k}$ is continuous for each $k\in \N$.
	\begin{definition}[Quasi-regularity]\label{def_QR}
		Let $ (\EE,\D(\EE))$ be a Dirichlet form on $L^2(\X,\nu)$. Then $(\EE,\D(\EE))$ is called quasi-regular if 
		\begin{enumerate}[(i)]
			\item \label{IA} there exists an $\EE$-nest consisting of compact sets, 
			\item \label{IB} there exists a dense subset of $(\D(\EE), (\cdot,\cdot)_{L^2(\nu)} + \EE )$ whose elements have $\EE$-quasi continuous $\nu$-versions,
			\item \label{IC} there exists $\{  \eta_k| k\in \N\} \subset \D(\EE)$, such that each $\eta_k$ has a $\EE$-quasi continuous $\nu$-version $\tilde{\eta}_k $, and an $\EE$-exceptional $N\subset \X$, such that $\{\tilde{\eta}_k|k\in \N\}$ separates the points of $\X \setminus N$.
		\end{enumerate}
	\end{definition}
	\begin{remark}[Regular Dirichlet forms]
		A more classical class of Dirichlet forms is the class of regular Dirichlet forms, which form a subset of the quasi-regular ones. Regularity is defined when  $\X$ is locally compact and  $\nu$ has full support. Then, a Dirichlet form $(\EE,\D(\EE))$ is regular if there exists a set $ \mathcal{C} \subset \D(\EE)\cap C_c(\X)$ which is dense in both $\D(\EE)$ and $C_c(\X)$. In this case, the set $\mathcal{C}$ is called a core of $(\EE,\D(\EE))$.
	\end{remark}
	We define  the support of a measurable  $\eta \colon \X \to\R$ as the smallest closed set outside which the measure $|\eta | \d\nu$ vanishes. Then, that a Dirichlet form $ (\EE,\D(\EE))$  on $L^2(\X,\nu)$ has the \emph{local property} refers to the fact that $\EE(\eta,\zeta) = 0$ for $\eta,\zeta \in \D(\EE)$ with disjoint support. Moreover, $(\EE,\D(\EE))$ is called conservative, if its associated semigroup satisfies $T_t\mathbf{1}_{\X} =\mathbf{1}_{\X}$ for all $t>0$.  
	\begin{theorem}[Albeverio, Ma, R\"ockner: Existence of an associated diffusion]\label{thm:qr_implies_ex}
		Let $ (\EE,\D(\EE))$ be a quasi-regular, local and conservative Dirichlet form on $L^2(\nu)$. Then, there exists an associated  diffusion process.
	\end{theorem}
	Due to the elaborate state of Dirichlet theory, we only state this specific condition for the existence of an associated Markov process as it is sufficient for our purposes. For more advanced versions allowing also for Markov processes with jumps and including  converse directions, we refer the interested reader to the aforementioned literature. We also remind ourselves that while the exposition from \cite{FOT} concerns mostly Dirichlet forms on locally compact state spaces, one can transfer these results to Polish (and more general) state spaces using the local compactification method detailed in \cite[Chapter VI]{ma2012introduction} if $(\EE,\D(\EE))$ is quasi-regular.
	\subsection{Decompositions of additive functionals}\label{app:decomp}
	We continue by stating two results relating additive functionals of a Markov process to an associated Dirichlet form. To this end, we continue in the setting of Theorem \ref{thm:qr_implies_ex}: $\X$ is a Polish space equipped with a measure $\nu$ of unit mass, and  $(\EE,\D(\EE))$ is a quasi-regular, local and conservative Dirichlet form on $L^2(\nu)$. Moreover, we let $( (Z_t)_{t\ge 0} , (P_z)_{z\in\X} )$ be an associated  diffusion process and recall the filtration $\FF$ defined in \eqref{e:filtr} and the shift operator $\theta_t$ from Definition \ref{defi_markov_proc}.
	\begin{definition}[Additive functional]An $\FF$-adapted, extended real-valued stochastic process $(A_t)_{t\ge 0}$ is called an additive functional of $( (Z_t)_{t\ge 0} , (P_z)_{z\in\X} )$, if there exists an $\EE$-exceptional $N\subset \X$  and $\Lambda\in  \FF_\infty$ with $P_z(\Lambda ) = 1$ for all $z\in \X\setminus N$ and $\theta_t \Lambda \subset \Lambda$, such that the following holds on $\Lambda$:
		\begin{enumerate}[(i)]
			\item $(A_t)_{t\ge 0}$ is a càdlàg, real-valued process,
			\item $A_0 =0$ and $A_{t+s} = A_t +A_s \theta_t$ for all $s,t\ge 0$.
		\end{enumerate}
	\end{definition}
	The sets $N$ and $\Lambda$ as above are called the \emph{exceptional set} and \emph{defining set} of an additive functional $(A_t)_{t\ge 0}$. Two additive functionals are identified with each other if there exists a common defining set on which they coincide. An additive functional is called continuous if there exists a defining set on which it has continuous paths.
	
	A specific class of interest is given by 
	\begin{equation}\label{e:AF_eta}
		A^{[{\eta}]}_t \,=\, 	\tilde{\eta }(Z_t ) \, - \,\tilde{\eta }(Z_0 ) ,
	\end{equation}
	where $\tilde{\eta}$ is an $\EE$-quasi-continuous $\nu$-version of $\eta\in \D(\EE)$. We remark that the latter always exists and that the additive functional is independent of the choice of this $\EE$-quasi-continuous version. 
	Moreover, since $( (Z_t)_{t\ge 0} , (P_z)_{z\in\X} )$ is a  diffusion, $A^{[\eta]}$ is continuous.

	Let us also introduce the energy of an additive functional as 
	\[
	e(A) \,=\, \lim_{t\searrow 0} \frac{1}{2t}E_\nu[A^2_t].
	\]
	We say that an additive functional $(M_t)_{t\ge 0}$ is a \emph{martingale additive functional of finite energy} if  for each  $t\ge 0$ we have $E_z[M_t^2]<\infty$ and $E_z[M_t]=0$ for $\EE$-quasi every $z\in \X$ as well as $e(M)<\infty$. We say that a continuous additive functional $(N_t)_{t\ge 0}$ is a \emph{continuous additive functional of zero energy} if for each  $t\ge 0$ we have $E_z[|N_t|]<\infty$  for $\EE$-quasi every $z\in \X$ and  $e(N)=0$. We state Fukushima's celebrated decomposition theorem.
	\begin{theorem}[Fukushima's decomposition]\label{thm_fukushima} Let $\eta\in \D(\EE)$ and $(A^{[\eta]}_t)_{t\ge 0}$ be the additive functional defined in \eqref{e:AF_eta}. Then there exists a unique martingale additive functional of finite energy $(M_t)_{t\ge 0}$ and a continuous additive functional of zero energy $(N_t)_{t\ge 0}$ such that 
		\[
		A^{[\eta ]} \,=\, M^{[\eta]} \,+\,N^{[\eta]},
		\]
		as additive functionals.
	\end{theorem}
	We remark that since in our situation $(A_t^{[\eta]})_{t\ge 0}$ is continuous and $(N^{[\eta]}_t)_{t\ge 0}$ is continuous by definition, so is also $(M^{[\eta]}_t)_{t\ge 0}$. 
	Based on the symmetry of the Dirichlet form, a related decomposition can be obtained when the process is started from its equilibrium measure $\nu$, cf.\ \cite{LyonsZheng,LyonsZhang}.
	\begin{theorem}[Lyons--Zheng decomposition]\label{LZ}
		Let $T\in(0,\infty)$ and define the backwards filtration $\mathfrak{G}_t^{T} = \sigma\{Z_{T-s}| s \in [0,t]\}$. Then, under $P_\nu$,
		there exists a continuous square integrable $\mathfrak{G}^{T}$-martingale $(\overline{M}_t)_{t\in [0,T]}$ with $\overline{M}_0=0$ such that 
		\begin{equation}
			A^{[\eta ]}_t \,=\, \frac{1}{2}M^{[\eta]}_t \,-\, \frac{1}{2}\bigl(\overline{M}_T^{[\eta]} - \overline{M}_{T-t}^{[\eta]} \bigr),
		\end{equation}
		for $t\in [0,T]$,
		$P_\nu$-almost surely.
	\end{theorem}
	More can be said about the martingales in this representation when $(\EE,\D(\EE))$ admits a square field operator, i.e., if there exists a bilinear and continuous mapping $\Gamma \colon \D(\EE)\times \D(\EE) \to L^1(\nu)$ such that
	\begin{align}\label{eqn45}
		2 \EE(\eta, \eta \zeta) \,-\, \EE(\eta^2,\zeta )\,=\, \int_{\X} \zeta \Gamma(\eta ,\eta ) \,\d\nu,
	\end{align}
	for all bounded $\eta ,\zeta \in \D(\EE)$, where the class is an algebra for any Dirichlet form. In this case the latter is called the \emph{square field operator} of $(\EE,\D(\EE))$ and the quadratic variation processes of $M^{[\eta]}$ and $\overline{M}^{[\eta]}$ are given by 
	\begin{align}\label{e:QV}\int_0^t \Gamma(\eta,\eta)(Z_s)\,\d s
		\qquad \text{and}\qquad
		\int_0^t \Gamma(\eta,\eta)(Z_{T-s})\,\d s,
	\end{align}
	respectively, under $P_\nu$.
	\subsection{Mosco convergence in varying Hilbert spaces}\label{App:Mosco}
	We  recall the concept of Mosco convergence in varying Hilbert spaces introduced in \cite{kuwae_shioya}, and refer the reader also to \cite{Grothaus_Wittmann,KOLESNIKOV2006382} for the more concrete setting that we use here. Let $\G$ and $ \G_N$, for $N\in \N$, be Hilbert spaces. We assume that there exist a dense linear subspace $\CC_0 \subset \G$ and linear maps  $\Psi_N\colon \CC_0 \to \G_N$ such that $\|\Psi_N \zeta \|_{\G_N} \to \|\zeta \|_{\G}$, for all $\zeta \in \CC_0$. 
	We say that a sequence $\eta_N \in \G_N$ \emph{converges strongly} to $\eta \in \G$ if $(\eta_N,\Psi_N \zeta)_{\G_N} \to (\eta,\zeta)_{\G}$ for all $\zeta \in \CC_0$ and at the same time $\|\eta_N\|_{\G_N} \to \|\eta\|_\G$, as $N\to\infty$.
	Strong convergence may also be characterized by the existence of a sequence $ \CC_0 \ni \zeta_k \to \eta $ in $\G$, such that 
	\[\lim_{k\to\infty} 
	\limsup_{N\to\infty} \|\Psi_N \zeta_k - \eta_N \|_{\G_N} \,=\, 0.
	\]
	 If for  $\zeta_N \in  \G_N$ and $\zeta \in \G$, it holds $(\eta_N,\zeta_N)_{\G_N} \to (\eta,\zeta)_{\G}$ for every  $\eta_N\in \G_N$ strongly converging to  $\eta \in \G$, the sequence $\zeta_N$ \emph{converges weakly} to $\zeta$. Also in this generalized setting, it holds that weakly converging sequences are norm bounded and that 
	 \[
	 \| \zeta\|_{\G} \,\le\, \liminf_{N\to \infty} \|\zeta_N\|_{\G_N}
	 \]
	 in this case.
	 
	\begin{definition}[Mosco convergence]For each $N\in \N$, let $(\EE_N,\D(\EE_N))$ be a symmetric, non-negative, and closed bilinear form on $\G_N$, and $(\EE,\D(\EE))$ on $\G$. Then the sequence $\EE_N$ converges to $\EE$ in Mosco sense, if the following holds.
		\begin{enumerate}[(i)]
			\item For every  sequence $ \eta_N \in \G_N$  converging weakly to $\eta\in \G$, we have 
			\[
			\EE(\eta,\eta) \le \liminf_{N\to\infty} \EE_N(\eta_N,\eta_N).
			\]
			\item\label{IIII2} For every $\eta \in \D_0$, where $\D_0$ is  a dense subspace of $( \D(\EE), (\cdot,\cdot)_{\G}+\EE) $, there exists a sequence $\eta_N \in \G_N$ converging strongly to $\eta$ such that 
			\[
			\limsup_{N\to\infty} \EE_N(\eta_N, \eta_N) \le \EE(\eta,\eta) .
			\] 
		\end{enumerate}
	\end{definition}
We remark that we could have required the property in \eqref{IIII2} also for all $\eta\in \D(\EE)$. While this is a seemingly stronger condition, the two are equivalent, cf.\ \cite[Theorem 3.4]{Grothaus_Wittmann}.   
	We state the Mosco--Kuwae--Shioya convergence theorem as  in \cite[Theorem 2.6]{KOLESNIKOV2006382}.
	\begin{theorem}[Mosco, Kuwae, Shioya]\label{Mosco_theorem_app}
		For each $N\in \N$, let $(\EE_N,\D(\EE_N))$ be a symmetric, non-negative, and closed bilinear form on $\G_N$, and $(\EE,\D(\EE))$ on $\G$. Moreover, let $(T_{N,t})_{t>0}$ and $(T_t)_{t>0}$ be their associated strongly continuous semigroups of contractions. Then the following are equivalent:
		\begin{enumerate}[(i)]
\item $\EE_N$ converges to $\EE$ in Mosco sense.
\item For every $\eta_N\in \G_N$ converging strongly to  $\eta\in \G$, $T_{N,t} \eta_N $ converges strongly to $ T_t\eta$ for any $t>0$.
		\end{enumerate}
	\end{theorem}

	\section{Piecewise interpolation of H\"older functions}\label{sec_holder}
	
	For the reader's convenience, we include the proof of Lemma \ref{Lemma_numerical_approx-n}.
	\begin{proof}[Proof of Lemma \ref{Lemma_numerical_approx-n}] Firstly, we can replace $Q_N h$ defined in \eqref{Eq101} by $
		{R}_N h \,=\, \sum_{i=0}^{N} h(x_i) \hat{\vp}_i$, because they only differ by the constant
		\begin{equation}\label{Eq122-n}
			\biggl|
			\frac{1}{N} \biggl(\frac{h(x_0)+h(x_N)}{2}+ \sum_{k=1}^{N-1} h(x_k)
			\biggr)
			\biggr| \,\le \, \sum_{\alpha=1}^{N}\int_{I_\alpha} \biggl|h(x) - \frac{h(x_{\alpha+}) + h(x_{\alpha-})}{2} \biggr|\,dx \,\le \, N^{-\tilde{s}}\|h\|_{C^{\tilde{s}}(\T)},
		\end{equation}
		where we used that $h$ is mean-free. Next, we compute 
		\begin{align}\begin{split}\label{Eq120-n}
				&|R_N h(x) - h(x)| \,\le  \,\biggl|
				(h(x)-h(x_{i+1}))\frac{x - x_i}{x_{i+1} - x_i}
				\biggr| \,+\, 
				\biggl|
				(h(x)-h(x_i))\frac{x_{i+1} - x}{x_{i+1} - x_i}
				\biggr| \,\le\, N^{-\tilde{s}} \|h\|_{C^{\tilde{s}}(\T)},
			\end{split}
		\end{align}
		assuming $x\in [x_i, x_{i+1}]$.
		To bound also $[R_N h - h]_{C^s(\T)}$, we write 
		\begin{align}
			&
			\bigl|	(R_Nh(y) - h(y) )\, -\, (R_Nh(x) -h (x)  ) \bigr|
			\\&\quad =\,
			\biggl|
			h(x) \,-\, \biggl(
			h(x_i) + (x-x_i )\frac{h(x_{i+1}) - h(x_{i})}{x_{i+1}-x_i}\biggr) \,-\, 
			\biggl(
			h(y) \,-\, \biggl(
			h(x_j) + (y-x_j )\frac{h(x_{j+1}) - h(x_{j})}{x_{j+1}-x_j}\biggr)
			\biggr)
			\biggr|,\label{Eq121-n}
		\end{align}
		where we assume additionally  that $y \in [x_j, x_{j+1}] $ and distinguish three cases:
		\begin{enumerate}[(i)]
			\item If $|x-y|\ge 1/N$, we can estimate \eqref{Eq121-n} by
			$
			2 N^{-\tilde{s}} \|h\|_{C^{\tilde{s}}(\T)} \,\le\, 2 N^{-(\tilde{s}-s)} |x-y|^s\|h\|_{C^{\tilde{s}}(\T)}
			$
			using  the previously obtained \eqref{Eq120-n}. 
			\item If $|x-y|<1/N$ and $i=j$, we bound \eqref{Eq121-n} instead by
			\begin{align}&
				|h(x) - h(y)|\,+\, N|x-y||h(x_{i+1}) - h(x_i)|
				\,\le \, 2 N^{ - (\tilde{s} - s)}|x-y|^s\|h\|_{C^{\tilde{s}}(\T)}.
			\end{align}
			\item Otherwise, we have $|x-y|<1/N$ and we can assume that $j=i+1$ so that we can bound \eqref{Eq121-n} by
			\begin{align}&
				\biggl|
				h(x) \,-\, \biggl(
				h(x_{i+1}) + (x-x_{i+1} )\frac{h(x_{i}) - h(x_{i+1})}{x_{i}-x_{i+1}}\biggr) \,-\, 
				\biggl(
				h(y) \,-\, \biggl(
				h(x_{i+1}) + (y-x_{i+1} )\frac{h(x_{i+2}) - h(x_{i+1})}{x_{i+2}-x_{i+1}}\biggr)
				\biggr)
				\biggr|\\& \quad 
				\le \, |h(x)-h(y)| \,+\, N|x-x_{i+1}|| h(x_{i+1}) -h(x_{i})| \,+\, N|y-x_{i+1}||h(x_{i+2}) - h(x_{i+1})|
				\\& \quad
				\le\, 2 N^{-(\tilde{s}-s)}|x-y|^s\|h\|_{C^{\tilde{s}}(\T)}.
			\end{align}
		\end{enumerate}
		In any case, together with \eqref{Eq122-n}--\eqref{Eq120-n} the proof is finished.
	\end{proof}

	\section{Tightness and Mosco convergence of  spectral Galerkin approximations}\label{App:spectral}
	Here, we give the details for the claim that under Assumption \ref{assumption_mosco_concrete}  the findings from Theorem \ref{thm_existence_deg__fin_case} and Theorem \ref{thm_convergence} also hold for the approximate Dirichlet forms  \eqref{E2} with  metric defined in \eqref{eqn_spectral_Galerkin_metric}. 
	\subsection{Existence and tightness of diffusion processes}
	The existence part of Theorem \ref{thm_existence_deg__fin_case} follows readily from Theorem \ref{thm_abstract}. Regarding the uniform tightness of the associated  diffusions, we need, as in Section \ref{S:discrete}, upper bounds on the measures and metrics. 
	
	As in Lemma \ref{p1}, the former may be deduced from  two ingredients: the fact that  $\Pi_{\le N}$ is an orthoprojection from  $\dot{H}^1(O)$ onto $ \check{\MM}_N$, cf.\ Lemma \ref{lemma_is_ortho}, which follows from the cosine modes being the eigenfunctions of the Neumann Laplacian, and an estimate
	 \[
	\|\Pi_{\le N} h - h\|_{C^s(\T)} \lesssim 
	N^{-(\tilde{s}-s -\epsilon)}\| h\|_{C^{\tilde{s}}(\T)},\qquad 0<s< \tilde{s}<1,\] 
	for mean-free functions $h$, 
	as stated in Lemma \ref{Lemma_numerical_approx-n} for $Q_N$.
	Here, we allow for an additional $\epsilon>0$, since Fourier (and thus cosine) series do not converge in supremum-based spaces. As they are bounded in $L^p$ and therefore in Bessel-potential spaces $H^{s,p}$, see \cite[Proposition 4.1.6 and Theorem 4.1.7]{grafakos2014classical}, we have however
	\begin{equation}\label{e4}
	\|\Pi_{\le N} h - h\|_{C^s(\T)} \lesssim 
	\| \Pi_{\le N} h - h\|_{H^{{s+\epsilon/2},p}(\T)} \,\lesssim \, N^{-(\tilde{s}-s -\epsilon)} \|  h\|_{H^{{\tilde{s}-\epsilon/2},p}(\T)}  \,\lesssim \, N^{-(\tilde{s}-s -\epsilon)} \|  h\|_{C^{{\tilde{s}}}(\T) },
	\end{equation}
	for sufficiently large $p$ and small $\epsilon>0$. Combining these observations as in Lemma \ref{p1} yields that
		\begin{align}\label{e:Bspp-bound_MMcheck}
			\sup_{N\ge 1}	\int_{\check{\MM}_N}  \|h\|_{C^\gamma(\T)}^p \,\d \check{\mu}_N \,<\,\infty,
		\end{align}
		for each $\gamma\in (0,1/2)$ and $p\in (0,\infty)$. 
		
		Regarding the metric tensor, we consider differential forms 
		\begin{align}\label{eqn924-app}
			q_\cdot\dot{h} = \int_{\T}\dot{h} q =  \int_{\T}\dot{h} \Pi_{\le N}q ,
		\end{align}
		which, for ${p}$ as in \eqref{eqn_spectral_Galerkin_metric},  we may evaluate by means of the Legendre formula 
		\begin{align}\begin{split}\label{eqn_explicit_formula_for_metric_galerkin}
			\frac{1}{2}	\check{g}^{h} (q,q ) \,&=\, \sup_{\dot {{h}} \in \mathcal{T}_h\check{\MM}_N} \biggl\{ \int_\T m(h) \px \Pi_{\le N} q \px {p}  \,-\, \frac{1}{2} \int_{\T} m(h) |\partial_x {p}|^2 \biggr\}
			\\&  = \, \frac{1}{2} \int_{\T} m(h) |\partial_x \Pi_{\le N}{q}|^2 .
		\end{split}\end{align}
		We remark that compared to \eqref{E89} we have an identity, since by taking $\dot{h} = -\Pi_{\le N}\partial_x   ( m(h) \partial_x  \Pi_{\le N} q)$ one can maximize the above.

		To prove also uniform space-time regularity, by the above spatial regularity assertion, it suffices to prove that 
		\begin{equation}\label{e3}
			E_{\check\mu_N}^{1/p} 
		|h_{t}(x)  - h_s(x)|^p \,\lesssim_{(p,T)}\, |t-s|^{1/8-\delta}
		\end{equation}
		for all $0\le s\le t \le T$, $x,y\in \T$ and $\delta>0$, 
		as in the proof of Proposition \ref{proposition_unif_reg}. 	Then, for a standard mollifier $\psi_\epsilon$ 
		 with $\supp \psi_\epsilon = [-\epsilon,\epsilon]$, we  split 
		\begin{align}&
			E_{\check{\mu}_N} ^{1/p}
			|h_{t}(x)  - h_s(x)|^p \\&\quad \le\,
			E_{\check{\mu}_N} ^{1/p}
			|h_{t}(x)  - \psi_\epsilon*h_t(x)|^p \,+\,
			E_{\check{\mu}_N} ^{1/p}
			|\psi_\epsilon*h_{t}(x)  - \psi_\epsilon*h_s(x)|^p
			\,+\, 
			E_{\check{\mu}_N} ^{1/p}
			|\psi_\epsilon*h_{s}(x)  - h_s(x)|^p,
		\end{align}
		where the convolution with respect to the even extension can be rephrased using 
		$\overline{\psi_\epsilon}$ as in \eqref{eqn327}--\eqref{eqn8324}, i.e., as
		\[
		\psi_\epsilon*h(x) \,=\,  \int_{\T} \overline{\psi_\epsilon}(x,z) h(z)\,\d z.
		\]
		The Lyons--Zheng decomposition and \eqref{eqn_explicit_formula_for_metric_galerkin} yield the upper bound 
		\begin{align}E_{\check{\mu}_N} ^{1/p}	|\psi_\epsilon*h_{t}(x)  - \psi_\epsilon*h_s(x)|^p\,\lesssim \, |t-s|^{1/2}\biggl( {\int_{\check{\MM}_N}}  \biggl(\int_{\T} m(h)( \partial_z \Pi_{\le N}\overline{\psi_\epsilon}(x,\cdot) )^2\,\d z\biggr)^{p/2} \,\d\check{\mu}_N (h)\biggr)^{1/p}.
		\end{align}
	Since $\Pi_{\le N}$ decreases the $\dot{H}^1(O)$-norm, \eqref{e3} follows by interpolation with \eqref{e:Bspp-bound_MMcheck}. 
		
	\subsection{Upper bound in the Mosco limit} To verify Theorem \ref{thm_mosco_abstract}~\eqref{Mosco_I1}, we let $p\in \dot{H}^1(O)$. Then, by \eqref{e4}, \eqref{eqn_explicit_formula_for_metric_galerkin} and the H\"older regularity asserted in Lemma \ref{lemma_reg}, we find that
	\begin{align} \check{g}^{\Pi_{\le N}h} (  p ,  p )  \,=\,  \int_O m(\Pi_{\le N} h) |\partial_x \Pi_{\le N} p |^2 \,\to\, \int_O m( h) |\partial_x {p}|^2 ,\qquad N\to\infty,
	\end{align}
	for $\mu$-a.e.\ $h$. Uniform $\mu$-integrability of the left-hand side follows from Assumption \ref{assumption_mosco_concrete}.

	\subsection{Lower bound in the Mosco limit}

	Let $\dot{h}\in \dot{H}^1(O)$ and, for now,  $h\in \dot{C}^\epsilon (\overline{O})$ with $\epsilon\in (0,1/2)$. By \eqref{eqn_spectral_Galerkin_metric}, there exists for each $N$ a ${p}_N $ such that 
	\begin{align}&
		\frac{1}{2}	\check{g}_{ \Pi_{\le N} h}(\Pi_{\le N} \dot h, \Pi_{\le N} \dot h ) \,= \,\int_O {p}_N \Pi_{\le N} \dot h \,-\, \frac{1}{2}\int_{\T} m(\Pi_{\le N} h) |\partial_x {p}_N  |^2.
	\end{align}
	Then, since 
	\[  \int_{\T} m(\Pi_{\le N} h) | \partial_x {p}_N  |^2\,=\, \int_O {p}_N \Pi_{\le N} \dot h 
	\]
	implies uniform boundedness of $ {p}_N$ in $\dot{H}^1(O)$, up to passing to a subsequence we have $
	\partial_x {p}_N \rightharpoonup \partial_x {p}$ in $L^2(O)$,  for some $ {p}\in \dot{H}^1(O)$. This implies
	\[
	\int_O {p}_N \Pi_{\le N} \dot h \,=\, \int_O  \dot h \Pi_{\le N}{p}_N \,\to\, \int_O {p} \dot h
	\]
	and since also $\sqrt{m}(\Pi_{\le N}h)\partial_x {p}_N \rightharpoonup \sqrt{m}(h) \partial_x {p}$ in $L^2(O)$ by \eqref{e4}, we conclude that
	\begin{align}&
		\limsup_{N\to\infty}	\frac{1}{2} \check	g_{ \Pi_{\le N} h}(\Pi_{\le N} \dot h, \Pi_{\le N} \dot h )  \,=\,\int_O {p} \dot h\,-\,\frac{1}{2}\liminf_{N\to\infty} \int_\T \bigl| \sqrt{m}(\Pi_{\le N} h ) \partial_x \Pi_{\le N}{p}_N \bigr|^2  \\&\quad \le\, \int_O {p} \dot h \,-\, \frac{1}{2} \int_\T \bigl| \sqrt{m}(h) \partial_x {p} \bigr|^2 .
	\end{align}
 Because $\mu$ is supported on $\dot{C}^\epsilon(\overline{\T})$, the first part of Theorem \ref{thm_mosco_abstract}~\eqref{Mosco_I2} follows. The required uniform integrability is a consequence of 
	\begin{align}&
	\check g_{ \Pi_{\le N} h}(\Pi_{\le N} \dot h, \Pi_{\le N} \dot h ) \,= \,\int_{\T} m(\Pi_{\le N} h) |\partial_x {p}_N  |^2 \, \lesssim_c \,
		\|\dot{h}\|_{\dot{H}^{-1}(\T)}^2
	\end{align}
being bounded uniformly in $h$ .

	\section{Integration by parts formula for  spectral Galerkin approximations}
	\label{App:proof_IBP}
		\begin{proof}[Proof of Lemma \ref{Lemma_generator_approx_problem}]
		For $\eta$ as in \eqref{Eq11_New} and $\zeta\in \D_0$ 
		we write   
		\begin{align}
			\begin{split}
				\label{Eq111}
				-\check{\EE}_N( \eta ,\zeta ) \,=\, &\sum_{i=1}^k\int_{\MM}  
			\partial_i   
			\Phi\biggl(\int_{\T}  h q^1,\dots , \int_{\T}  h q^k \biggr) 
			\int_O  \Pi_{\le N}\partial_x( m(\Pi_{\le N} h)  \partial_x \Pi_{\le N}
			q^i )
			\diff|_h\zeta  \,\d x
			\,\d\mu(h)
			\\
			& +\, \sum_{i=1}^k\int_{\MM} 
			\partial_i   
			\Phi\biggl(\int_{\T}  h q^1,\dots , \int_{\T}  h q^k \biggr) \int_O     (\Pi_{> N} \partial_x^2  \Pi_{> N}q^i )     
			\diff|_h \zeta 
			\,\d x
			\,\d\mu(h)\,=\, I_1 \,+\, I_2,
			\end{split}
		\end{align}
		using integration by parts in the $x$-variable and the orthogonality of $\Pi_{\le N} $   in $L^2(O)$. For the former we use in particular that each $q^i$ and therefore also  $\Pi_{>N}q^i$ satisfies Neumann boundary conditions.  
		The vector field 
		\[
		X_2 (h)\,=\, \sum_{i=1}^k \partial_i   
		\Phi\biggl(\int_{\T}  h q^1,\dots , \int_{\T}  h q^k \biggr)  (\Pi_{> N} \partial_x^2  \Pi_{> N}q^i )     
		\]
		integrated against $\diff \zeta $  in $I_2$ has precisely the form \eqref{eqn_XX}. Using the Gaussian integration  by parts formula \eqref{GI_eqn}, we find that
		\begin{align}I_2 \,=\,- \int_\MM 
			\zeta \div_\mu(X_2) 
			\,\d \mu,
		\end{align} 
		where 
		\begin{align} \label{e9}\div_\mu(X_2)\,=\, &- \sum_{i=1}^k
			\partial_i	\Phi\biggl(\int_{\T}  \cdot  q^1,\dots , \int_{\T}  \cdot  q^k \biggr)
			\bigl( \cdot, 
			\partial_x^2 \Pi_{>N}q^i
			\bigr)_{\dot{H}^1(O)  }
			\\&- \sum_{i,j=1}^k  \partial_{ij} \Phi\biggl(\int_{\T}  \cdot  q^1,\dots , \int_{\T}  \cdot  q^k \biggr) \int_\T  \partial_x \Pi_{>N} q^i \partial_x \Pi_{>N}  q^j.
		\end{align}
		
		We claim that the same argument can be applied to the vector field in $I_1$, i.e., 
		\begin{align}X_1(h) \,=\, \sum_{i=1}^k 
			\partial_i   
			\Phi\biggl(\int_{\T}  h q^1,\dots , \int_{\T}  h q^k \biggr) 
			\Pi_{\le N}\partial_x( m(\Pi_{\le N} h)  \partial_x \Pi_{\le N}
			q^i )
			.
		\end{align}
		To this end, we introduce the functions
		\begin{align}
			\Psi_{i,l}(r_1,\dots , r_k,s_1,\dots, s_N ) \,=\, -\partial_{i}\Phi(r_1,\dots,r_k )\int_\T \partial_x p^l m( s_1 \vp_1 +\dots+ s_N \vp_N ) \partial_x \Pi_{\le N} q^i ,
		\end{align}
		for $i=1,\dots, k$ and $l=1,\dots, N$,
		so that 
		\[
		X_1(h)\,=\,\sum_{i=1}^k\sum_{l=1}^N \Psi_{i,l}\biggl(
		\int_{\T}  h q^1,\dots , \int_{\T}  h q^k ,
		\int_{\T}  h p^1,\dots , \int_{\T}  h p^N \biggr) \vp_l,
		\]
		where $\vp_k$ and $p^k$ are taken as in \eqref{eqn112}. By \eqref{Eq220}, the above is of the form \eqref{eqn_XX} and therefore 
		\begin{align}I_1 \,=\, - \int_{\MM} \zeta \div_\mu(X_1)\,\d \mu, 
		\end{align} 
		where 
		\begin{align}
			\begin{split}\label{e8}
				\div_\mu(X_1) \,=\,&-\,\sum_{i=1}^k\sum_{l=1}^N
				\Psi_{i,l}\biggl(
				\int_{\T}  \cdot q^1,\dots , \int_{\T}  \cdot q^k ,
				\int_{\T}  \cdot p^1,\dots , \int_{\T}  \cdot p^N \biggr) (\cdot, \vp_l)_{\dot{H}^1(O)}
				\\& +\, \sum_{i,j=1}^k\sum_{l=1}^N \partial_{r_j} \Psi_{i,l}\biggl(
				\int_{\T}  \cdot q^1,\dots , \int_{\T}  \cdot q^k ,
				\int_{\T}  \cdot p^1,\dots , \int_{\T}  \cdot p^N \biggr) \int_O \vp_l q^j 
				\\& +\, \sum_{i=1}^k\sum_{j,l=1}^N \partial_{s_j} \Psi_{i,l}\biggl(
				\int_{\T}  \cdot q^1,\dots , \int_{\T}  \cdot q^k ,
				\int_{\T}  \cdot p^1,\dots , \int_{\T}  \cdot p^N \biggr)\int_O  \vp_l p^j\,=\,J_1\,+\, J_2\,+\, J_3.
			\end{split}
		\end{align}
		Inserting the definition of $\Psi_{i,l}$, we see that 
		\begin{align}\label{e5}J_1(h) \,=\,-\sum_{i=1}^k	\partial_i   
			\Phi\biggl(\int_{\T}  h q^1,\dots , \int_{\T}  h q^k \biggr) \bigl( h, 
			\Pi_{\le N}\partial_x( m(\Pi_{\le N} h)  \partial_x \Pi_{\le N}
			q^i )				\bigr)_{\dot{H}^1(O)  }.
		\end{align}
		Since 
		\begin{align}\partial_{r_j} 
			\Psi_{i,l}(r_1,\dots , r_k,s_1,\dots, s_N ) \,=\, -\partial_{ij}\Phi(r_1,\dots,r_k )\int_\T \partial_x p^l m( s_1 \vp_1 +\dots+ s_N \vp_N ) \partial_x \Pi_{\le N} q^i ,
		\end{align}
		we obtain for the same reason that 
		\begin{align}\label{e6}
			J_2(h) \,=\,  -\sum_{i,j=1}^k\partial_{ij}\Phi
			\biggl(\int_{\T}  h q^1,\dots , \int_{\T}  h q^k \biggr) 
			\int_\T  m(\Pi_{\le N} h ) \partial_x \Pi_{\le N} q^i \partial_x \Pi_{\le N} q^j.
		\end{align}
		To evaluate also $J_3$, we observe that
		\begin{align}
			\partial_{s_j} 
			\Psi_{i,l}(r_1,\dots , r_k,s_1,\dots, s_N ) \,=\, -\partial_{i}\Phi(r_1,\dots,r_k )
			\int_\T \partial_x p^l m'( s_1 \vp_1 +\dots+ s_N \vp_N ) \vp_j \partial_x \Pi_{\le N} q^i ,
		\end{align}
		and consequently 
		\begin{align}
			J_3(h) \,=\, -\sum_{i=1}^k	\partial_i   
			\Phi\biggl(\int_{\T}  h q^1,\dots , \int_{\T}  h q^k \biggr) \sum_{j=1}^N
			\int_O    m'(\Pi_{\le N}  h) \vp_j\partial_x  p^j \partial_x \Pi_{\le N } q^i .\label{eqn632}
		\end{align}
		Using the elementary identity
		\[ \vp_j\partial_x  p^j \,=\, -2  \pi j  \cos(\pi j x )\sin(\pi j x )
		\,=\, -\pi j \sin(2\pi j x),
		\]
		we can rewrite 
		\[\sum_{j=1}^N
		\int_O    m'(\Pi_{\le N}  h) \vp_j\partial_x  p^j \partial_x \Pi_{\le N } q^i \,=\,\frac12 \sum_{j=1}^N
		\int_O  \partial_x (\cos(2\pi j \cdot ))   m'(\Pi_{\le N}  h)  \partial_x \Pi_{\le N } q^i .
		\]
		Using that 
		\[
		\sum_{j = 1}^N \cos(2\pi j x ) \,=\, \frac{D_N(2 \pi x ) -1}{2}
		\]
		for the Dirichlet kernel $D_N$ defined in \eqref{eqn_def_Dirichlet kernel}, and that the constant is annihilated by differentiation, we can further express this as
		\[
		\sum_{j=1}^N
		\int_O    m'(\Pi_{\le N}  h) \vp_j\partial_x  p^j \partial_x \Pi_{\le N } q^i \,=\, \frac{1}{4}\int_O  \partial_x (D_N(2\pi \cdot ))  m'(\Pi_{\le N}  h)  \partial_x \Pi_{\le N } q^i.
		\]
		Inserting this 
		in \eqref{eqn632}, and \eqref{e5}, \eqref{e6} and \eqref{eqn632} in \eqref{e8} yields together with \eqref{e9} the desired \eqref{Eq74}.
	\end{proof}

	\vspace{.2cm}
	\noindent
	\textbf{Acknowledgements.}
	The first author acknowledges support by the Max Planck Society
	through the Research Group “Stochastic Analysis in the Sciences (SAiS)” and the DFG
	CRC/TRR 388 “Rough Analysis, Stochastic Dynamics and Related Fields”, Project A11.   This research was partially carried out during the third author's visits at the TU Berlin supported by the European Union (ERC,
	FluCo, grant agreement No. 101088488). Views
	and opinions expressed are however those of the author(s) only and do not necessarily reflect those
	of the European Union or of the European Research Council. Neither the European Union nor the
	granting authority can be held responsible for them.
	
	\vspace{.2cm}
	
		\noindent
	\textbf{AI usage declaration.}
	When this manuscript was completed, we used a large language model for a final check of mathematical correctness and language. In doing so, a previously imposed lower bound on the mobility in Theorem \ref{thmB}~\eqref{I2} could be removed based on the current proof of Lemma \ref{lemma_IBP_fd}. The authors take full responsibility for the contents of this manuscript.
	
	\vspace{.2cm}
	\noindent
	\textbf{Data availability.} This manuscript has no associated data.

	\vspace{.2cm}
	\noindent
	\textbf{Declaration – Conflict of interest.} The authors have no conflict of interest.


\begin{thebibliography}{10}
		
		\bibitem{Agresti_Sauerbrey}
		{\sc A.~Agresti and M.~Sauerbrey}, {\em Well-posedness of the stochastic
			thin-film equation with an interface potential}, Commun. Math. Phys., 407
		(2026), p.~66.
		\newblock Id/No 158.
		
		\bibitem{AR_90_closability}
		{\sc S.~Albeverio and M.~R\"ockner}, {\em Classical {D}irichlet forms on
			topological vector spaces---closability and a {C}ameron-{M}artin formula}, J.
		Funct. Anal., 88 (1990), pp.~395--436.
		
		\bibitem{AR_91}
		\leavevmode\vrule height 2pt depth -1.6pt width 23pt, {\em Stochastic
			differential equations in infinite dimensions: solutions via {D}irichlet
			forms}, Probab. Theory Related Fields, 89 (1991), pp.~347--386.
		
		\bibitem{berghlofstroom}
		{\sc J.~Bergh and J.~L\"ofstr\"om}, {\em Interpolation spaces. {A}n
			introduction}, vol.~No. 223 of Grundlehren der Mathematischen Wissenschaften,
		Springer-Verlag, Berlin-New York, 1976.
		
		\bibitem{bergh_Vp}
		{\sc J.~Bergh and J.~Peetre}, {\em On the spaces {$V\sb{p}$} {$(0<p\leq \infty
				)$}}, Boll. Un. Mat. Ital. (4), 10 (1974), pp.~632--648.
		
		\bibitem{BERNIS1990}
		{\sc F.~Bernis and A.~Friedman}, {\em Higher order nonlinear degenerate
			parabolic equations}, Journal of Differential Equations, 83 (1990),
		pp.~179--206.
		
		\bibitem{Boga_Gaussian}
		{\sc V.~I. Bogachev}, {\em Gaussian measures}, vol.~62 of Mathematical Surveys
		and Monographs, American Mathematical Society, Providence, RI, 1998.
		
		\bibitem{Chen_Fukushima}
		{\sc Z.-Q. Chen and M.~Fukushima}, {\em Symmetric {M}arkov processes, time
			change, and boundary theory}, vol.~35 of London Mathematical Society
		Monographs Series, Princeton University Press, Princeton, NJ, 2012.
		
		\bibitem{dapratozabczyk}
		{\sc G.~Da~Prato and J.~Zabczyk}, {\em Stochastic equations in infinite
			dimensions}, vol.~152 of Encyclopedia of Mathematics and its Applications,
		Cambridge University Press, Cambridge, second~ed., 2014.
		
		\bibitem{dareiotis2021nonnegative}
		{\sc K.~Dareiotis, B.~Gess, M.~V. Gnann, and G.~Gr{\"u}n}, {\em Non-negative
			martingale solutions to the stochastic thin-film equation with nonlinear
			gradient noise}, Archive for Rational Mechanics and Analysis, 242 (2021),
		pp.~179--234.
		
		\bibitem{DMES2005}
		{\sc B.~Davidovitch, E.~Moro, and H.~A. Stone}, {\em Spreading of viscous fluid
			drops on a solid substrate assisted by thermal fluctuations}, Phys. Rev.
		Lett., 95 (2005), p.~244505.
		
		\bibitem{EK}
		{\sc S.~N. Ethier and T.~G. Kurtz}, {\em Markov processes}, Wiley Series in
		Probability and Mathematical Statistics: Probability and Mathematical
		Statistics, John Wiley \& Sons, Inc., New York, 1986.
		\newblock Characterization and convergence.
		
		\bibitem{feller_annals}
		{\sc W.~Feller}, {\em The parabolic differential equations and the associated
			semi-groups of transformations}, Ann. of Math. (2), 55 (1952), pp.~468--519.
		
		\bibitem{fischer_gruen_2018}
		{\sc J.~Fischer and G.~Grün}, {\em {Existence} of positive solutions to
			stochastic thin-film equations}, SIAM Journal on Mathematical Analysis, 50
		(2018), pp.~411--455.
		
		\bibitem{friz_victoir}
		{\sc P.~K. Friz and N.~B. Victoir}, {\em Multidimensional stochastic processes
			as rough paths}, vol.~120 of Cambridge Studies in Advanced Mathematics,
		Cambridge University Press, Cambridge, 2010.
		\newblock Theory and applications.
		
		\bibitem{FOT}
		{\sc M.~Fukushima, Y.~Oshima, and M.~Takeda}, {\em Dirichlet forms and
			symmetric {M}arkov processes}, vol.~19 of De Gruyter Studies in Mathematics,
		Walter de Gruyter \& Co., Berlin, extended~ed., 2011.
		
		\bibitem{wuchen}
		{\sc Y.~Gao, W.~Li, and J.-G. Liu}, {\em Fluctuations in {W}asserstein dynamics
			on graphs}, 2024.
		
		\bibitem{GGKO21}
		{\sc B.~Gess, R.~Gvalani, F.~Kunick, and F.~Otto}, {\em Thermodynamically
			consistent and positivity-preserving discretization of the thin-film equation
			with thermal noise}, Mathematics of Computation, 92 (2023), pp.~1931--1976.
		
		\bibitem{GJ_ARMA}
		{\sc P.~Gon\c{c}alves and M.~Jara}, {\em Nonlinear fluctuations of weakly
			asymmetric interacting particle systems}, Arch. Ration. Mech. Anal., 212
		(2014), pp.~597--644.
		
		\bibitem{grafakos2014classical}
		{\sc L.~Grafakos}, {\em Classical {F}ourier analysis}, Springer, New York,
		2014.
		
		\bibitem{Grothaus_Wittmann}
		{\sc M.~Grothaus and S.~Wittmann}, {\em Mosco convergence of gradient forms
			with non-convex interaction potential}, Integral Equations Operator Theory,
		96 (2024), pp.~Paper No. 29, 69.
		
		\bibitem{GG_LK}
		{\sc G.~Gr\"un and L.~Klein}, {\em Finite speed of propagation for a class of
			stochastic thin-film equations}, SIAM J. Math. Anal., 57 (2025),
		pp.~4822--4866.
		
		\bibitem{gruen_rumpf}
		{\sc G.~Gr\"un and M.~Rumpf}, {\em Nonnegativity preserving convergent schemes
			for the thin film equation}, Numer. Math., 87 (2000), pp.~113--152.
		
		\bibitem{GruenMeckeRauscher2006}
		{\sc G.~Grün, K.~Mecke, and M.~Rauscher}, {\em Thin-film flow influenced by
			thermal noise}, Journal of Statistical Physics, 122 (2006), pp.~1261--1291.
		
		\bibitem{GubPer_uniqueness}
		{\sc M.~Gubinelli and N.~Perkowski}, {\em Energy solutions of {KPZ} are
			unique}, J. Amer. Math. Soc., 31 (2018), pp.~427--471.
		
		\bibitem{GubPer_generator}
		\leavevmode\vrule height 2pt depth -1.6pt width 23pt, {\em The infinitesimal
			generator of the stochastic {B}urgers equation}, Probab. Theory Related
		Fields, 178 (2020), pp.~1067--1124.
		
		\bibitem{gvalani2023stochastic}
		{\sc R.~S. Gvalani and M.~Tempelmayr}, {\em Stochastic estimates for the
			thin-film equation with thermal noise}, Commun. Math. Phys., 407 (2026),
		p.~87.
		\newblock Id/No 189.
		
		\bibitem{HMW}
		{\sc M.~Hairer, J.~Maas, and H.~Weber}, {\em Approximating rough stochastic
			{PDE}s}, Comm. Pure Appl. Math., 67 (2014), pp.~776--870.
		
		\bibitem{janson}
		{\sc S.~Janson}, {\em Brownian excursion area, {W}right's constants in graph
			enumeration, and other {B}rownian areas}, Probab. Surv., 4 (2007),
		pp.~80--145.
		
		\bibitem{kallenberg1997foundations}
		{\sc O.~Kallenberg}, {\em Foundations of modern probability. {In} 2 volumes},
		Cham: Springer, 2021.
		
		\bibitem{KOLESNIKOV2006382}
		{\sc A.~V. Kolesnikov}, {\em {M}osco convergence of {D}irichlet forms in
			infinite dimensions with changing reference measures}, Journal of Functional
		Analysis, 230 (2006), pp.~382--418.
		
		\bibitem{kuwae_shioya}
		{\sc K.~Kuwae and T.~Shioya}, {\em Convergence of spectral structures: a
			functional analytic theory and its applications to spectral geometry}, Comm.
		Anal. Geom., 11 (2003), pp.~599--673.
		
		\bibitem{lepingle}
		{\sc D.~L{\'e}pingle}, {\em La variation d'ordre p des semi-martingales}, Z.
		Wahrscheinlichkeitstheor. Verw. Geb., 36 (1976), pp.~295--316.
		
		\bibitem{LyonsZhang}
		{\sc T.~J. Lyons and T.~S. Zhang}, {\em Decomposition of {D}irichlet processes
			and its application}, Ann. Probab., 22 (1994), pp.~494--524.
		
		\bibitem{LyonsZheng}
		{\sc T.~J. Lyons and W.~A. Zheng}, {\em A crossing estimate for the canonical
			process on a {D}irichlet space and a tightness result}, Ast{\'e}risque,
		(1988), pp.~249--271.
		
		\bibitem{ma2012introduction}
		{\sc Z.-M. Ma and M.~R{\"o}ckner}, {\em Introduction to the theory of
			(non-symmetric) Dirichlet forms}, Springer Science \& Business Media, 2012.
		
		\bibitem{metzger2022existence}
		{\sc S.~Metzger and G.~Grün}, {\em Existence of nonnegative solutions to
			stochastic thin-film equations in two space dimensions}, Interfaces and Free
		Boundaries, 24 (2022), pp.~307--387.
		
		\bibitem{Nualart_related}
		{\sc D.~Nualart}, {\em The {M}alliavin calculus and related topics},
		Probability and its Applications (New York), Springer-Verlag, Berlin,
		second~ed., 2006.
		
		\bibitem{ODB97}
		{\sc A.~Oron, S.~H. Davis, and S.~G. Bankoff}, {\em Long-scale evolution of
			thin liquid films}, Rev. Mod. Phys., 69 (1997), pp.~931--980.
		
		\bibitem{Otto_PME}
		{\sc F.~Otto}, {\em The geometry of dissipative evolution equations: the porous
			medium equation}, Comm. Partial Differential Equations, 26 (2001),
		pp.~101--174.
		
		\bibitem{rock_schmu_tightness}
		{\sc M.~R\"ockner and B.~Schmuland}, {\em Tightness of general {$C_{1,p}$}
			capacities on {B}anach space}, J. Funct. Anal., 108 (1992), pp.~1--12.
		
		\bibitem{Sauerbrey_2021}
		{\sc M.~Sauerbrey}, {\em Martingale solutions to the stochastic thin-film
			equation in two dimensions}, Ann. Inst. Henri Poincar\'e{} Probab. Stat., 60
		(2024), pp.~373--412.
		
		\bibitem{sheng1990extension}
		{\sc S.~Y. Sheng}, {\em The extension of the theorems of {{\v{C}}. V.
				Stanojevi{\'c}} and {V. B. Stanojevi{\'c}}}, Proceedings of the American
		Mathematical Society, 110 (1990), pp.~895--904.
		
		\bibitem{triebel_95}
		{\sc H.~Triebel}, {\em Interpolation theory, function spaces, differential
			operators}, Johann Ambrosius Barth, Heidelberg, second~ed., 1995.
		
		\bibitem{Renesse_Sturm}
		{\sc M.-K. von Renesse and K.-T. Sturm}, {\em {Entropic measure and Wasserstein
				diffusion}}, The Annals of Probability, 37 (2009), pp.~1114 -- 1191.
		
		\bibitem{zhuzhu_phi43}
		{\sc R.~Zhu and X.~Zhu}, {\em Dirichlet form associated with the {$\Phi^4_3$}
			model}, Electron. J. Probab., 23 (2018), pp.~1--31.
		
	\end{thebibliography}
\end{document}